\documentclass[a4paper,11pt]{amsart}
\title{A rotational hyperbolic theory for surface homeomorphisms}
\date{\today}

\author{Pierre-Antoine Guih\'eneuf}
\address{Sorbonne Universit\'e, Universit\'e Paris Cit\'e, CNRS, IMJ-PRG, F-75005 Paris, France --- 
IRL Jean-Christophe Yoccoz CNRS / IMPA, Estr. Dona Castorina, 110
Jardim Bot\^anico, Rio de Janeiro, Brasil}
\email{pierre-antoine.guiheneuf@imj-prg.fr}
\thanks{P.-A.\ G.\ thanks the Jean-Christophe Yoccoz international laboratory CNRS/IMPA for the semester in Brazil during which the ideas of this work were born.}

\usepackage[utf8]{inputenc}
\usepackage[english]{babel}
\usepackage[T1]{fontenc}  
\usepackage{amsfonts}
\usepackage{amsmath}
\usepackage{amssymb}
\usepackage{amsthm}
\usepackage{mathtools}

\usepackage{enumitem}
\setlist{noitemsep}
\usepackage{tikz}
\usetikzlibrary{arrows}
\PassOptionsToPackage{hyphens}{url}
\usepackage[pdftex,colorlinks=true,linkcolor=blue,citecolor=blue,urlcolor=blue]{hyperref}

\newtheorem{lemma}{Lemma}[section]
\newtheorem{theorem}[lemma]{Theorem}
\newtheorem{theo}{Theorem}

\newtheorem{propo}[theo]{Proposition}

\newtheorem{prop}[lemma]{Proposition}

\newtheorem{coro}[lemma]{Corollary}
\newtheorem{claim}[lemma]{Claim}

\theoremstyle{definition}
\newtheorem{definition}[lemma]{Definition}
\theoremstyle{remark}
\newtheorem{rem}[lemma]{Remark}

\newcommand{\F}{\mathcal{F}}
\newcommand{\Hy}{\mathbf{H}}
\newcommand{\N}{\mathbf{N}}

\newcommand{\R}{\mathbf{R}}

\newcommand{\G}{\mathcal{G}}

\newcommand{\Q}{\mathbf{Q}}
\newcommand{\Z}{\mathbf{Z}}
\newcommand{\X}{\mathcal{X}}
\newcommand{\varep}{\varepsilon}
\newcommand{\Homeo}{\operatorname{Homeo}}
\newcommand{\supp}{\operatorname{supp}}

\newcommand{\Fix}{\operatorname{Fix}}
\newcommand{\rot}{\operatorname{rot}}
\newcommand{\rote}{\operatorname{rot}_{\mathrm{erg}}}

\newcommand{\conv}{\operatorname{conv}}
\newcommand{\diam}{\operatorname{diam}}
\newcommand{\inte}{\operatorname{int}}
\newcommand{\dom}{\operatorname{dom}}
\newcommand{\Id}{\operatorname{Id}}
\newcommand{\dd}{\,\mathrm{d}}
\newcommand{\cl}{\mathcal{N}}
\newcommand{\wt}{\widetilde}
\newcommand{\wh}{\widehat}

\newcommand{\pr}{\operatorname{pr}}
\newcommand{\Tr}{\mathcal{T}}

\newcommand{\Me}{\mathcal{M}^{\mathrm{erg}}}
\newcommand{\Merg}{\mathcal{M}^{\mathrm{erg}}_{\vartheta>0}}

\begin{document}

\maketitle

\begin{abstract}
We develop a rotational hyperbolic theory for surface homeomorphisms. We use the equivalence relation on ergodic measures that have nontrivial rotational behaviour defined in \cite{alepablo} to define a rotational counterpart of homoclinic classes. These allows to produce a network of horseshoes representing the whole rotational behaviour of the homeomorphism.
We also study the counterpart of heteroclinic connections and give 5 different characterizations of such connections.

The main technical tool is the forcing theory of Le Calvez and Tal \cite{lct1, lct2}, and in particular a result of creation of periodic points that can also be seen as a statement of homotopically bounded deviations \cite{paper1PAF}. 

This theoretical article is followed by a paper focused of some applications of it to the case of homeomorphisms with big rotation set \cite{G25Cvx2}.

\end{abstract}

\tableofcontents

\section{Introduction}

The goal of this article is to start building a rotational hyperbolic theory for surface homeomorphisms. We will study the dynamics of $f$ on these hyperbolic-like classes, with a focus on a counterpart of the notion of heteroclinic connections. 

This ``toolkit'' paper is followed up with another work \cite{} focusing on applications of the theory we set up here to homeomorphisms whose rotation set is big enough. We hope this second paper is only an illustration of the interest of this theory and that it could be applied to the study of rotational properties of any homeomorphism of a closed surface of genus $g\ge 2$.

It turns out that a good strategy for defining hyperbolic-like sets is to pass through the help of ergodic theory and define hyperbolic-like (in a rotational meaning) ergodic measures. We first need to describe the rotational dynamics of such measures.

\subsection*{Framework}

More formally, fix $S$ a closed surface (compact, connected, orientable, without boundary) of genus $g\ge 2$. We equip $S$ with a Riemannian metric $d$ of constant curvature $-1$. We denote $\Homeo_0(S)$ the set of homeomorphisms of $S$ that are isotopic to the identity.

We will need to consider $\wt S$ the universal cover of $S$; by the uniformization theorem $\wt S$ is isometric to the hyperbolic plane $\Hy^2$ (with a metric we also denote by $d$). This universal cover (as any Gromov hyperbolic space) has a boundary at infinity that we will denote by $\partial\wt S$. We also denote $\G$ the group of deck transformations of $\wt S$ (\emph{i.e.} the set of lifts of $\Id_S$ to $\wt S$). Every homeomorphism $f\in\Homeo_0(S)$ has a preferred lift $\wt f\in\Homeo_0(\wt S)$ (the only one homotopic to $\Id_{\wt S}$); this lifts commutes with elements of $\G$ and extends continuously to $\wt S\cup\partial \wt S$ with $\Id_{\partial\wt S}$. The compactification $\wt S\cup \partial\wt S$ will be equipped with a finite diameter distance (\emph{e.g.} coming from the euclidean distance on the unit disk in the Poincar\'e disk model).

\subsection*{Rotation sets}

We denote $\mathcal{M}(f)$ the set of $f$-invariant Borel probability measures, and $\mathcal{M}^{\textnormal{erg}}(f)$ the subset of $\mathcal{M}(f)$ made of $f$-ergodic measures. 

Let us define the homological rotation set of a homeomorphism $f\in \Homeo_0(S)$; this definition is due to Schwarzman \cite{MR88720} and was adapted for surface homeomorphisms by Pollicott \cite{pollicott}.
We recall that as $S$ is a closed surface of genus $g$, the homology group $H_1(S,\R)$ is a real vector space of dimension $2g$.
Given \(a \in \G\), we denote \([a] \in H_1(S,\R)\) its {homology class}.

Fix a bounded and measurable fundamental domain \(D \subset \widetilde S\) for the action of \(\G\) on \(\widetilde S\) and denote \(\widetilde x\) the lift of \(x \in S\) to \(D\). 
For each \( y \in S\) let \(a_{y}\) be the unique element of $\G$ such that \(\wt f(\wt y) \in a_y D\).\label{Defay1}
For any path $\beta : [0,1]\to S$, we consider $\wt\beta : [0,1]\to \wt S$ the lift of $\beta$ such that $\wt\beta(0)\in D$, and $T_\beta\in \G$ such that $\wt\beta(1)\in T_\beta D$. This allows to define $[\beta] = [T_\beta]\in H_1(S,\Z)$ .

\begin{definition}\label{DefHomRotVect1}
Given an \(f\)-invariant probability measure \(\mu\), the \emph{homological rotation vector} of \(\mu\) is
\begin{equation}\label{eq:homologyequation1}
\rho(\mu) = \int_{S}[a_y]\dd\mu(y).
\end{equation}
\end{definition}

Note that by Birkhoff ergodic theorem, if moreover \(\mu\) is ergodic, then for \(\mu\)-almost every \(x \in S\)
\begin{equation}\label{eq:homologyequation21}
\rho(\mu) = \int_{S}[a_y]\dd\mu(y) = \lim\limits_{n \to +\infty}\frac{1}{n}\sum_{i=0}^{n-1}[a_{f^i(x)}].
\end{equation}
If $x\in S$ is such that the right equality of \eqref{eq:homologyequation21} holds, we will denote $\rho(x) = \rho(\mu)$. More generally, we will denote $\rho(x)$ the set of accumulation points of the sequence
\[\left(\frac{1}{n}\sum_{i=0}^{n-1}[a_{f^i(x)}]\right)_n.\]

\begin{rem}\label{RemIndpD}
This definition is independent of the choice of the fundamental domain $D$. To see this, note that by $f$-invariance of $\mu$, \eqref{eq:homologyequation1} can be written, for any $n>0$,
\[\rho(\mu) = \frac{1}{n}\int_{S}\sum_{i=0}^{n-1}[a_{f^i(y)}]\dd\mu(y).\]
But the deck transformation $(a_xa_{f(x)}\cdots a_{f^{n-1}(x)})^{-1}$ sends $\wt f^n(\wt x)$ to $D$, and two fundamental domains are at bounded Hausdorff distance,
and hence the sums $\sum_{i=0}^{n-1}[a_{f^i(y)}]$ associated to two different fundamental domains only differ by a constant uniformly bounded in $n$ and $x$.

By construction, the map $\mu\mapsto \rho(\mu)$ is affine. It is also continuous: fix $\mu_0\in\mathcal M(f)$ and choose a fundamental domain $D$ such that $\mu_0(\partial D) = 0$. Then the map $y\mapsto [a_y]$ is piecewise constant with a discontinuity set of zero measure, hence by Portmanteau theorem $\mu\mapsto \rho(\mu)$ is continuous at $\mu_0$. 
\end{rem}

%

\begin{definition}[Homological rotation sets]\label{DefErgHomRot1}
Let $f \in \Homeo_0(S)$. 
The \emph{(homological) rotation set} $\rot(f)$ of $f$ is the set of vectors $\rho\in H_1(S,\R)$ such that there exist $(x_k)_k\in S^\N$ and $(n_k)_k\in\N^\N$ with $\lim_{k\to+\infty} n_k = +\infty$ and such that 
\[\lim\limits_{k \to +\infty}\frac{1}{n_k}\sum_{i=0}^{n_k-1}[a_{f^{i}(x_k)}] = \rho.\]
The \emph{ergodic (homological) rotation set} $\rote(f)$ of $f$ is
\[\rote(f) = \big\{\rho(\mu) \mid \mu \in \Me(f)\big\}.\] 
\end{definition}

We will denote $\conv(A)$ the convex hull of a set $A$. 
%

\subsection*{Rotational properties of ergodic measures}

The following is a combination of \cite[Lemma~1.6]{alepablo} and \cite[Theorem~B]{alepablo}.
As usual, we will parametrize geodesics by arclength. 

\begin{theorem}\label{DefTrackGeod}
Let $\mu\in\Me(f)$. Then there exists a constant \(\vartheta_\mu\in\R_+\) --- called the \emph{rotation speed} of $\mu$ --- such that
for \(\mu\)-almost every point \(z \in S\), there exists a geodesic $\gamma_z\subset T^1S$ --- called the \emph{tracking geodesic} of $z$ ---, and for each lift \(\wt z\) of \(z\) to $\wt S$, a lift $\wt \gamma_{\wt z}$ of $\gamma_z$, such that:
\begin{equation}\label{eq:trackingequation}
\lim\limits_{n \to +\infty}\frac{1}{n}d\big(\wt f^n(\wt z), \wt \gamma_{\wt z}(n \vartheta_\mu)\big) = \lim\limits_{n \to +\infty}\frac{1}{n}d\big(\wt f^{-n}(\wt z), \wt \gamma_{\wt z}(-n \vartheta_\mu)\big) = 0.
\end{equation}	
\end{theorem}

Note that if $\vartheta_\mu = 0$, then $\gamma_z$ can be chosen as any tracking geodesic of $S$; otherwise $\gamma_z$ is unique.

We denote by \(\Merg(f)\) the set of $\mu\in\Me(f)$ such that $\vartheta_\mu>0$. 
The geodesic $\wt\gamma_{\wt z}$ will be parametrized such that $d(\wt z,\wt\gamma_{\wt z}) = d(\wt z,\wt\gamma_{\wt z}(0))$. 


There is no reason for the map $z\mapsto\gamma_z$ to be $\mu$-a.e.~constant (and there are examples where it is not, see \cite[Subsection~7.1]{pa}). The following result gives a sense to the expression ``the closure of the set of tracking geodesics associated to a measure'' (\emph{a priori}, tracking geodesics are only defined almost everywhere):

\begin{theorem}[{\cite[Theorem~D]{alepablo}}]\label{thm:equidistributiontheoremintro}
For each \(\mu \in \Merg\) there exists a closed set $\dot\Lambda_{\mu}\subset T^1 S$ that is invariant under the geodesic flow and satisfies 
\[\dot\Lambda_{\mu}:= \overline{\dot\gamma_z(\R)} \]
for \(\mu\)-a.e.~\(z \in S\). Moreover, for \(\mu\)-a.e.~\(z \in S\), the geodesic $\gamma_z$ is recurrent.
\end{theorem}

\begin{definition}\label{DefRelEqMeas}
We define the equivalence relation $\sim$ on $\Merg$ by: $\mu_1\sim \mu_2$ if one of the following is true:
\begin{itemize}
\item $\dot\Lambda_{\mu_1} = \dot\Lambda_{\mu_2}$;
\item There exist $\tau_1,\dots,\tau_m\in\Merg$ such that $\tau_1=\mu_1$, $\tau_m=\mu_2$ and for all $1\le i<m$, the measures $\tau_i$ and $\tau_{i+1}$ are \emph{dynamically transverse}, \emph{i.e.}~there exist two geodesics $\gamma\subset \dot\Lambda_{\tau_i}$ and $\gamma'\subset \dot\Lambda_{\tau_{i+1}}$ that have a transverse intersection.
\end{itemize}  
We then denote $\{\cl_i\}_{i\in I}=\Merg/\sim$ the equivalence classes of $\sim$. For $i\in I$, we denote 
\begin{equation}\label{EqDefRhoi}
\rho_i = \{\rho(\mu)\mid \mu\in \cl_i\}
\qquad\text{and}\qquad
\dot\Lambda_i = \bigcup_{\mu\in \cl_i}\dot\Lambda_\mu.
\end{equation}
\end{definition}

\begin{definition}\label{DefClassesI1}
We call $I^1$ the set of classes with the property that any two measures $\mu_1$ and $\mu_2$ of $\cl_i$ satisfy $\dot\Lambda_{\mu_1} = \dot\Lambda_{\mu_2}$; by \cite[Theorem~5.8]{alepablo} this implies that the geodesics spanned by vectors in \(\dot\Lambda_{\mu_1}\) are simple.
Let $I_{\mathrm h}$ denote the other classes, which are such that for any $\mu\in \cl_i$ with \(i \in I_{\mathrm h}\), there exists $\mu'\in\Merg$ such that $\mu$ and $\mu'$ are dynamically transverse. Classes $\cl_i$ for which $i\in I_{\mathrm h}$ will be called \emph{chaotic classes}. 
\end{definition}

For $z$ a periodic point, we will sometimes use the abuse of notation $z\in\cl_i$ when the uniform measure on the orbit of $z$ belongs to $\cl_i$. 
%
%

\subsection*{Dynamics on chaotic classes}

Let us come to the results of this article.
We will show some results suggesting that the rotational dynamics associated to chaotic classes is quite similar to the one on a hyperbolic set of a $C^1$-diffeomorphism. In particular, there is a phenomenon resembling Markov partitions and a shadowing in rotation (Subsection~\ref{SubSecConstrG}).

As an application of these ideas, we will get the following result. Points 2.~and 3.~are partially adaptations in higher genus of the results of \cite{zbMATH00009916, llibremackay} for torus homeomorphisms (Theorem~\ref{thm:ShapeRotationSet} 
states that for $i\in I_{\mathrm h}$, the set $\rho_i$ is ``almost convex'').

\begin{propo}\label{PropRealByCompact}
Let $f\in\Homeo_0(S)$. Then:
\begin{enumerate}
\item For any $i\in I_{\mathrm{h}}$ and any $\rho\in \overline{\rho_i}$, there exists $x\in S$, such that $\rho(x) = \rho$.
\item If $\rho\in \inte(\rho_i)$ (the interior is taken inside the span of the convex set), then there exists a compact $f$-invariant set $K_\rho\subset S$ and $L_\rho>0$ such that for any $x\in K_\rho$ and any $n\in\N$,
\[d\big([a_x^n],\, n\rho \big)\le L_\rho.\]
\item If $C\subset \inte(\rho_i)$ is a compact connected set, then there exists $x\in S$ such that $\rho(x) = C$. 
\end{enumerate}
\end{propo}

As noted in \cite[Figure~14]{alepablo}, there is no ``exactness'' property of periods of periodic points in chaotic classes (see also \cite[Figure~4]{G25Cvx2}) as it holds for the torus \cite{MR0958891}, \emph{i.e.}~if $\rho\in \inte(\rho_i)\cap q^{-1}H_1(s,\Z)$ for some $q\in\N^*$, then there is not necessarily $z\in S$ that is $q$-periodic and satisfies $\rho(z) = \rho$. 
However, for any $i\in I_{\mathrm h}$, one can prove that for any finite collection $v_1,\dots,v_\ell\in\inte(\rho_i)$, there exists $M>0$ such that if $p\in q H_1(S,\Z) \cap q \conv(\{v_1,\dots,v_\ell\})$, then there exists a periodic point of period $\le qM$ realizing the rotation vector $(1/q)p$ (see the paragraph after \cite[Remark~6.6]{alepablo}). 
One can wonder if a stronger result holds, that is: for $i\in I_{\mathrm h}$, does there exist $M>0$ such that if $p\in q H_1(S,\Z) \cap q \rho_i$, then there exists a periodic point of period $\le qM$ realizing the rotation vector $(1/q)p$?

In this article we do not study generalizations of stable/unstable manifolds for such homeomorphisms, we refer to \cite{garcia2024fullychaotic, militon2024generalizedfoliations} for recent avenues in this direction that could be used in further works. Another natural question is to determine to what extend the network of horseshoes associated to a chaotic class is related to the chaotic sea defined in \cite{zbMATH06908424}.

\subsection*{Heteroclinic connections}

We will also focus on heteroclinic connections between  classes of $I_{\mathrm h}$. We will define 5 relations between the classes of $I_{\mathrm{h}}$: 
\begin{itemize}
\item $\overset\F\to$ that is stated in terms of $\F$-transverse intersections in the sense of Le Calvez-Tal(Definition~\ref{DefToStar});
\item $\overset *\to$ that is stated in terms of convergence of sequences of empirical measures in weak-$*$ topology (Definition~\ref{DefRelTo});
\item $\overset M\to$ that is stated in terms of Markovian intersections between rectangles in $G$ (Definition~\ref{DefRelMarkov});
\item $\overset \wedge\to$ that is stated in terms of intersections of essential loops (Definition~\ref{DefToWedge});
\item $\overset O\to$ that is stated in terms of intersections of open sets (Definition~\ref{DefToOpen}).
\end{itemize}

\begin{theo}\label{TheoEquiConnec2}
For any $f\in \Homeo_0(S)$, the five relations $\overset\F\to$, $\overset *\to$, $\overset M\to$, $\overset \wedge\to$ and $\overset O\to$ coincide. 
\end{theo}

These 5 identical binary relations are in fact order relations (Proposition~\ref{PropToOrderRel}). 

Finally, we link heteroclinic connections with the geometry of the surface, with the help of a graph we denote $\Tr$ (Subsection~\ref{SubSecTree}). 

These considerations allow to exhibit some subsets of the rotation set of $f$ (Corollary~\ref{CoroPropRotfEqualRotTr}) and identify some essential $f$-invariant open subsets of $f$ bearing some rotational properties of $f$ (Proposition~\ref{LemOImpliesF}).
\bigskip

In the companion paper \cite{}, building on the present work, we conduct a case study of homeomorphisms whose rotation set is big enough (the precise condition is $\inte(\conv(\rot(f)))\neq\emptyset$). These homeomorphisms can be considered as having a ``rotational Axiom A'' behaviour; one can understand very well a lot of their rotational properties, including: the shape of their rotation sets, bounded deviation results and realization results (see also \cite{MR4578317} for the study of rotation sets of Axiom A surface diffeomorphisms).

\subsection*{Tools}

We will set two theoretical tools. The first one is the rotation set associated to a collection of Markovian intersections of rectangles, it is included in the rotation set of the homeomorphism (Proposition~\ref{PropConnectRectEnsRot}). 
The second tool is a simple criterion of creation of heteroclinic connections between topological horseshoes in terms of the forcing theory (Theorem~\ref{ThConnectionHorse}, see Figure~\ref{Fig1}).

\begin{figure}
\begin{center}

\tikzset{every picture/.style={line width=0.75pt}} 

\begin{tikzpicture}[x=0.75pt,y=0.75pt,yscale=-1,xscale=1]

\draw [color={rgb, 255:red, 144; green, 19; blue, 254 }  ,draw opacity=1 ]   (80,151) .. controls (111.36,127.48) and (166.5,192.25) .. (124,192.75) .. controls (81.5,193.25) and (124.35,140.91) .. (166.5,138.75) .. controls (208.65,136.59) and (241,167.25) .. (217,181.75) .. controls (193,196.25) and (157.5,135.75) .. (225.5,118.25) ;
\draw  [draw opacity=0][fill={rgb, 255:red, 245; green, 166; blue, 35 }  ,fill opacity=1 ] (118.5,177.63) .. controls (118.5,175.76) and (120.01,174.25) .. (121.88,174.25) .. controls (123.74,174.25) and (125.25,175.76) .. (125.25,177.63) .. controls (125.25,179.49) and (123.74,181) .. (121.88,181) .. controls (120.01,181) and (118.5,179.49) .. (118.5,177.63) -- cycle ;
\draw  [draw opacity=0][fill={rgb, 255:red, 245; green, 166; blue, 35 }  ,fill opacity=1 ] (203.5,164.63) .. controls (203.5,162.76) and (205.01,161.25) .. (206.88,161.25) .. controls (208.74,161.25) and (210.25,162.76) .. (210.25,164.63) .. controls (210.25,166.49) and (208.74,168) .. (206.88,168) .. controls (205.01,168) and (203.5,166.49) .. (203.5,164.63) -- cycle ;
\draw  [color={rgb, 255:red, 0; green, 93; blue, 203 }  ,draw opacity=1 ][fill={rgb, 255:red, 0; green, 93; blue, 203 }  ,fill opacity=0.1 ] (346.79,121.86) -- (406.82,128.01) -- (400.67,188.04) -- (340.64,181.89) -- cycle ;
\draw  [color={rgb, 255:red, 208; green, 2; blue, 27 }  ,draw opacity=1 ][fill={rgb, 255:red, 208; green, 2; blue, 27 }  ,fill opacity=0.1 ] (477.36,141.98) -- (529.7,123.14) -- (548.53,175.48) -- (496.19,194.31) -- cycle ;
\draw  [color={rgb, 255:red, 0; green, 93; blue, 203 }  ,draw opacity=1 ][fill={rgb, 255:red, 0; green, 93; blue, 203 }  ,fill opacity=0.1 ] (378,97.13) .. controls (404.75,97.13) and (381.73,197.55) .. (393.73,200.64) .. controls (405.73,203.73) and (493.73,223.91) .. (511,211) .. controls (528.27,198.09) and (505,139.73) .. (501.55,126.27) .. controls (503.55,125.55) and (504.82,124.45) .. (507.91,124.09) .. controls (510.27,138.09) and (534.27,196.45) .. (514.82,214.64) .. controls (495.36,232.82) and (405.73,218.45) .. (386.45,203.36) .. controls (367.18,188.27) and (399.5,111.13) .. (378.25,106.13) .. controls (357.75,102.63) and (354.45,183.18) .. (354.14,199.57) .. controls (349.24,200.04) and (352.71,200.14) .. (344.14,198.71) .. controls (347.91,181) and (350.25,98.38) .. (378,97.13) -- cycle ;
\draw  [color={rgb, 255:red, 208; green, 2; blue, 27 }  ,draw opacity=1 ][fill={rgb, 255:red, 208; green, 2; blue, 27 }  ,fill opacity=0.1 ] (495.25,97.88) .. controls (524.16,85.51) and (544.27,182.82) .. (549.91,193.36) .. controls (546.64,194.45) and (545.36,195.91) .. (541.91,197.18) .. controls (537.18,182.64) and (517.2,96.97) .. (497.75,105.88) .. controls (478.3,114.78) and (503,186.27) .. (512.09,204.27) .. controls (508.27,205) and (506.64,206.45) .. (504.32,207.44) .. controls (498.09,191.73) and (466.34,110.24) .. (495.25,97.88) -- cycle ;
\draw  [draw opacity=0][fill={rgb, 255:red, 245; green, 166; blue, 35 }  ,fill opacity=1 ] (372.25,117.38) .. controls (372.25,115.51) and (373.76,114) .. (375.63,114) .. controls (377.49,114) and (379,115.51) .. (379,117.38) .. controls (379,119.24) and (377.49,120.75) .. (375.63,120.75) .. controls (373.76,120.75) and (372.25,119.24) .. (372.25,117.38) -- cycle ;
\draw  [draw opacity=0][fill={rgb, 255:red, 245; green, 166; blue, 35 }  ,fill opacity=1 ] (499.25,117.13) .. controls (499.25,115.26) and (500.76,113.75) .. (502.63,113.75) .. controls (504.49,113.75) and (506,115.26) .. (506,117.13) .. controls (506,118.99) and (504.49,120.5) .. (502.63,120.5) .. controls (500.76,120.5) and (499.25,118.99) .. (499.25,117.13) -- cycle ;
\draw [color={rgb, 255:red, 144; green, 19; blue, 254 }  ,draw opacity=1 ]   (132,150.63) .. controls (142.5,143.75) and (155.67,139.21) .. (166.5,138.75) ;
\draw [shift={(151.63,141.47)}, rotate = 160.72] [fill={rgb, 255:red, 144; green, 19; blue, 254 }  ,fill opacity=1 ][line width=0.08]  [draw opacity=0] (8.04,-3.86) -- (0,0) -- (8.04,3.86) -- (5.34,0) -- cycle    ;

\draw (268,147.4) node [anchor=north west][inner sep=0.75pt]  {$\implies$};

\end{tikzpicture}

\caption{Idea of the statement of Theorem~\ref{ThConnectionHorse}: if there is a trajectory under the isotopy like the one in the left of the figure in the space of leaves, then there exists two rotational horseshoes for $f$ having a heteroclinic connection.}\label{Fig1}
\end{center}
\end{figure}
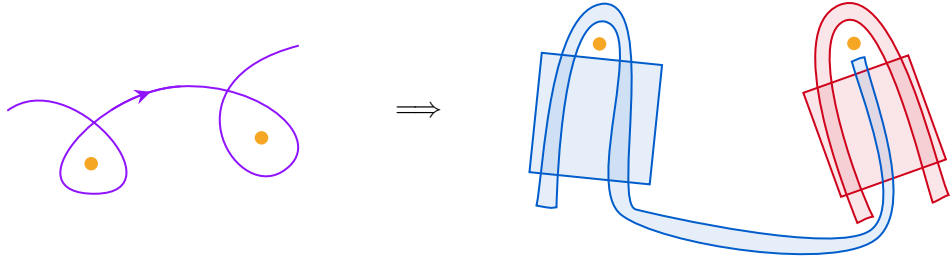

Besides these two results, we will make a systematic use the forcing theory of Le Calvez and Tal \cite{lct1}, and also a result (Theorem~\ref{ThGT}) due to the author and Tal \cite{paper1PAF} (and itself also based on the forcing theory), that allows to create periodic orbits with prescribed rotational behaviour when there exist some orbit with big deviations with respect to some other periodic orbits.

\section{Preliminaries}

For $\alpha$ a loop, the notation $[\alpha]$ will denote either its class in $\pi_1(S)$, or its class in $H_1(S,\R)$; whether it is the first or the second one will be clear from the context.

\subsection{Forcing theory}

\paragraph{Foliations and isotopies.}

Given an identity isotopy $I = \{f_t\}_{t\in[0,1]}$ for $f$ (\emph{i.e.}~$I^0 = \Id_S$ and $I^1 = f$), we define its fixed point set $\Fix(I) = \bigcap_{t\in[0,1]} \Fix(f_t)$, and denote its \emph{domain} $\textnormal{dom}(I) := S \backslash \Fix(I)$. Note that $\dom(I)$ is an oriented boundaryless surface, not necessarily closed, not necessarily connected.	

In this subsection we will consider an oriented surface $\Sigma$ without boundary, not necessarily closed or connected (with the idea to apply it to $\Sigma = \dom(I)$), and a non singular oriented topological foliation $\F$ on $\Sigma$. We will denote $\wh\Sigma$ the universal covering space of $\Sigma$ and $\wh{\F}$ the lifted foliation on $\wh{\Sigma}$.

For every point $z\in\Sigma$, we denote $\phi_{z}$ the leaf of $\F$ containing $z$. 
The complement of any simple injective proper path $\wh\alpha$ of $\wh\Sigma$ has two connected components, denoted by $L(\wh\alpha)$ and $R(\wh\alpha)$, chosen accordingly to some fixed orientation of $\wh\Sigma$ and the orientation of $\wh\alpha$.
Given a simple injective oriented proper path $\wh\alpha$ and $\wh z\in \wh\alpha$, we denote $\wh\alpha^+_{\wh z}$ and $\wh\alpha^-_{\wh z}$ the connected components of $\wh\alpha\setminus\{\wh z\}$, chosen accordingly to the orientation of $\wh\alpha$; their respective projections on $\Sigma$ are denoted respectively $\alpha^+_{z}$ and $\alpha^-_{z}$.

\paragraph{$\F$-transverse paths and $\F$-transverse intersections.}

We say that path $\eta:J\to\Sigma$ is \emph{positively transverse}\footnote{In the sequel, ``transverse'' will mean ``positively transverse''.} to $\F$ if it crosses locally each leaf of $\F$ it meets from left to right. 
The property of being positively transverse stays true for every lift $\wh\eta:J\to\wh\Sigma$ of a positively transverse path $\eta$ and that for every $a<b$ in $J$, the path $\wh\eta|_{[a,b]}$ meets once every leaf $\wh\phi$ of $\wh \F$ such that $L(\wh\phi_{\wh \eta(a)})\subset L(\wh\phi)\subset L(\wh\phi_{\wh \eta(b)})$ and that $\wh\eta|_{[a,b]}$ does not meet any other leaf.

Two transverse paths $\wh\eta_1:J_1\to\wh\Sigma$ and $\wh\eta_2:J_2\to\wh\Sigma$ are called \emph{equivalent} if they meet the same leaves of $\wh{\F}$. Two transverse paths $\eta_1:J_1\to\Sigma$ and $\eta_2:J_2\to\Sigma$ are \emph{equivalent} if they have lifts to $\wh\Sigma$ that are equivalent.
\bigskip

We will say that a transverse path $\alpha : [a, b] \to \dom(I )$ is \emph{admissible of order $n$} if it is equivalent to a path $I^{[0,n]}_\F(z)$ for some $z\in\dom(I)$.

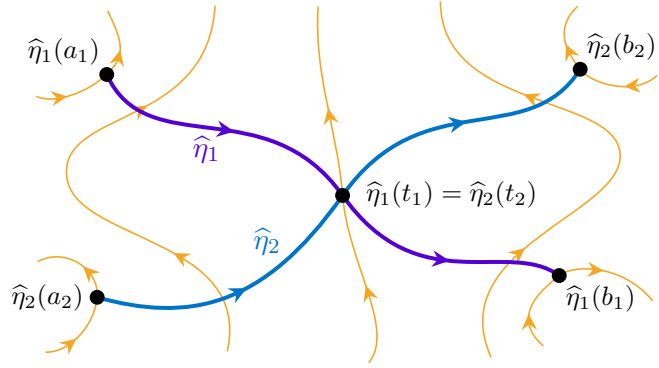
\begin{figure}
\begin{center}

\tikzset{every picture/.style={line width=0.6pt}} 

\begin{tikzpicture}[x=0.75pt,y=0.75pt,yscale=-1,xscale=1]

\draw [color={rgb, 255:red, 245; green, 166; blue, 35 }  ,draw opacity=1 ]   (261.43,50.73) .. controls (259.25,85.76) and (271.02,112.03) .. (273.64,145.74) .. controls (276.25,179.45) and (300.23,213.16) .. (286.72,229.8) ;
\draw [shift={(266.37,101.91)}, rotate = 78.81] [fill={rgb, 255:red, 245; green, 166; blue, 35 }  ,fill opacity=1 ][line width=0.08]  [draw opacity=0] (8.04,-3.86) -- (0,0) -- (8.04,3.86) -- (5.34,0) -- cycle    ;
\draw [shift={(285.23,191.05)}, rotate = 71.75] [fill={rgb, 255:red, 245; green, 166; blue, 35 }  ,fill opacity=1 ][line width=0.08]  [draw opacity=0] (8.04,-3.86) -- (0,0) -- (8.04,3.86) -- (5.34,0) -- cycle    ;
\draw [color={rgb, 255:red, 245; green, 166; blue, 35 }  ,draw opacity=1 ]   (156.35,52.92) .. controls (164.64,69.56) and (165.07,76.12) .. (155.92,84.88) .. controls (146.76,93.64) and (130.19,100.2) .. (120.6,95.83) ;
\draw [shift={(162.66,72.57)}, rotate = 103] [fill={rgb, 255:red, 245; green, 166; blue, 35 }  ,fill opacity=1 ][line width=0.08]  [draw opacity=0] (8.04,-3.86) -- (0,0) -- (8.04,3.86) -- (5.34,0) -- cycle    ;
\draw [shift={(136.28,95.94)}, rotate = 172] [fill={rgb, 255:red, 245; green, 166; blue, 35 }  ,fill opacity=1 ][line width=0.08]  [draw opacity=0] (8.04,-3.86) -- (0,0) -- (8.04,3.86) -- (5.34,0) -- cycle    ;
\draw [color={rgb, 255:red, 245; green, 166; blue, 35 }  ,draw opacity=1 ]   (391,50.6) .. controls (378.79,61.98) and (385.69,73.93) .. (392.23,82.25) .. controls (398.77,90.57) and (420.57,95.39) .. (430.6,85.32) ;
\draw [shift={(385.1,69.68)}, rotate = 66] [fill={rgb, 255:red, 245; green, 166; blue, 35 }  ,fill opacity=1 ][line width=0.08]  [draw opacity=0] (8.04,-3.86) -- (0,0) -- (8.04,3.86) -- (5.34,0) -- cycle    ;
\draw [shift={(414.19,90.98)}, rotate = -5] [fill={rgb, 255:red, 245; green, 166; blue, 35 }  ,fill opacity=1 ][line width=0.08]  [draw opacity=0] (8.04,-3.86) -- (0,0) -- (8.04,3.86) -- (5.34,0) -- cycle    ;
\draw [color={rgb, 255:red, 245; green, 166; blue, 35 }  ,draw opacity=1 ]   (209.98,50.29) .. controls (206.49,112.03) and (133.24,101.96) .. (135.86,135.67) .. controls (138.48,169.38) and (228.73,165.88) .. (216.09,224.11) ;
\draw [shift={(176.16,100.54)}, rotate = 149.51] [fill={rgb, 255:red, 245; green, 166; blue, 35 }  ,fill opacity=1 ][line width=0.08]  [draw opacity=0] (8.04,-3.86) -- (0,0) -- (8.04,3.86) -- (5.34,0) -- cycle    ;
\draw [shift={(189.24,174.67)}, rotate = 28.12] [fill={rgb, 255:red, 245; green, 166; blue, 35 }  ,fill opacity=1 ][line width=0.08]  [draw opacity=0] (8.04,-3.86) -- (0,0) -- (8.04,3.86) -- (5.34,0) -- cycle    ;
\draw [color={rgb, 255:red, 245; green, 166; blue, 35 }  ,draw opacity=1 ]   (322.47,48.1) .. controls (322.47,103.27) and (410.46,94.79) .. (412.2,125) .. controls (413.94,155.21) and (323.34,172.01) .. (347.76,227.17) ;
\draw [shift={(364.02,94.87)}, rotate = 22.65] [fill={rgb, 255:red, 245; green, 166; blue, 35 }  ,fill opacity=1 ][line width=0.08]  [draw opacity=0] (8.04,-3.86) -- (0,0) -- (8.04,3.86) -- (5.34,0) -- cycle    ;
\draw [shift={(364.67,172.23)}, rotate = 140.57] [fill={rgb, 255:red, 245; green, 166; blue, 35 }  ,fill opacity=1 ][line width=0.08]  [draw opacity=0] (8.04,-3.86) -- (0,0) -- (8.04,3.86) -- (5.34,0) -- cycle    ;
\draw [color={rgb, 255:red, 84; green, 0; blue, 213 }  ,draw opacity=1 ][line width=1.5]    (155.92,84.88) .. controls (175.97,129.54) and (235.71,91.89) .. (273.64,145.74) .. controls (311.57,199.59) and (356.04,167.19) .. (381.77,186.02) ;
\draw [shift={(219.54,113.55)}, rotate = 188.99] [fill={rgb, 255:red, 84; green, 0; blue, 213 }  ,fill opacity=1 ][line width=0.08]  [draw opacity=0] (8.75,-4.2) -- (0,0) -- (8.75,4.2) -- (5.81,0) -- cycle    ;
\draw [shift={(327.1,178.63)}, rotate = 187.51] [fill={rgb, 255:red, 84; green, 0; blue, 213 }  ,fill opacity=1 ][line width=0.08]  [draw opacity=0] (8.75,-4.2) -- (0,0) -- (8.75,4.2) -- (5.81,0) -- cycle    ;
\draw [color={rgb, 255:red, 0; green, 116; blue, 201 }  ,draw opacity=1 ][line width=1.5]    (151.12,196.96) .. controls (196.46,209.66) and (233.96,202.65) .. (273.64,145.74) .. controls (313.31,88.82) and (366.07,123.85) .. (392.23,82.25) ;
\draw [shift={(225.84,192.34)}, rotate = 153.5] [fill={rgb, 255:red, 0; green, 116; blue, 201 }  ,fill opacity=1 ][line width=0.08]  [draw opacity=0] (8.75,-4.2) -- (0,0) -- (8.75,4.2) -- (5.81,0) -- cycle    ;
\draw [shift={(335.43,108.69)}, rotate = 169.41] [fill={rgb, 255:red, 0; green, 116; blue, 201 }  ,fill opacity=1 ][line width=0.08]  [draw opacity=0] (8.75,-4.2) -- (0,0) -- (8.75,4.2) -- (5.81,0) -- cycle    ;
\draw [color={rgb, 255:red, 245; green, 166; blue, 35 }  ,draw opacity=1 ]   (121.47,179.01) .. controls (134.99,171.57) and (152.43,182.51) .. (151.12,196.96) .. controls (149.81,211.41) and (137.17,224.98) .. (125.4,224.55) ;
\draw [shift={(144.21,181.45)}, rotate = 45.25] [fill={rgb, 255:red, 245; green, 166; blue, 35 }  ,fill opacity=1 ][line width=0.08]  [draw opacity=0] (8.04,-3.86) -- (0,0) -- (8.04,3.86) -- (5.34,0) -- cycle    ;
\draw [shift={(140.71,217.82)}, rotate = 141.55] [fill={rgb, 255:red, 245; green, 166; blue, 35 }  ,fill opacity=1 ][line width=0.08]  [draw opacity=0] (8.04,-3.86) -- (0,0) -- (8.04,3.86) -- (5.34,0) -- cycle    ;
\draw [color={rgb, 255:red, 245; green, 166; blue, 35 }  ,draw opacity=1 ]   (428.42,198.71) .. controls (412.72,179.45) and (387.87,180.76) .. (381.77,186.02) .. controls (375.66,191.27) and (355.61,205.72) .. (373.92,221.48) ;
\draw [shift={(403.94,183.63)}, rotate = 185.09] [fill={rgb, 255:red, 245; green, 166; blue, 35 }  ,fill opacity=1 ][line width=0.08]  [draw opacity=0] (8.04,-3.86) -- (0,0) -- (8.04,3.86) -- (5.34,0) -- cycle    ;
\draw [shift={(366.62,205.34)}, rotate = 88.52] [fill={rgb, 255:red, 245; green, 166; blue, 35 }  ,fill opacity=1 ][line width=0.08]  [draw opacity=0] (8.04,-3.86) -- (0,0) -- (8.04,3.86) -- (5.34,0) -- cycle    ;
\draw  [fill={rgb, 255:red, 0; green, 0; blue, 0 }  ,fill opacity=1 ] (270.32,145.74) .. controls (270.32,143.9) and (271.81,142.41) .. (273.64,142.41) .. controls (275.47,142.41) and (276.95,143.9) .. (276.95,145.74) .. controls (276.95,147.58) and (275.47,149.07) .. (273.64,149.07) .. controls (271.81,149.07) and (270.32,147.58) .. (270.32,145.74) -- cycle ;
\draw  [fill={rgb, 255:red, 0; green, 0; blue, 0 }  ,fill opacity=1 ] (388.92,82.25) .. controls (388.92,80.41) and (390.4,78.92) .. (392.23,78.92) .. controls (394.06,78.92) and (395.55,80.41) .. (395.55,82.25) .. controls (395.55,84.09) and (394.06,85.58) .. (392.23,85.58) .. controls (390.4,85.58) and (388.92,84.09) .. (388.92,82.25) -- cycle ;
\draw  [fill={rgb, 255:red, 0; green, 0; blue, 0 }  ,fill opacity=1 ] (378.45,186.02) .. controls (378.45,184.18) and (379.94,182.69) .. (381.77,182.69) .. controls (383.6,182.69) and (385.08,184.18) .. (385.08,186.02) .. controls (385.08,187.86) and (383.6,189.35) .. (381.77,189.35) .. controls (379.94,189.35) and (378.45,187.86) .. (378.45,186.02) -- cycle ;
\draw  [fill={rgb, 255:red, 0; green, 0; blue, 0 }  ,fill opacity=1 ] (147.8,196.96) .. controls (147.8,195.12) and (149.29,193.63) .. (151.12,193.63) .. controls (152.95,193.63) and (154.44,195.12) .. (154.44,196.96) .. controls (154.44,198.8) and (152.95,200.29) .. (151.12,200.29) .. controls (149.29,200.29) and (147.8,198.8) .. (147.8,196.96) -- cycle ;
\draw  [fill={rgb, 255:red, 0; green, 0; blue, 0 }  ,fill opacity=1 ] (152.6,84.88) .. controls (152.6,83.04) and (154.09,81.55) .. (155.92,81.55) .. controls (157.75,81.55) and (159.23,83.04) .. (159.23,84.88) .. controls (159.23,86.72) and (157.75,88.21) .. (155.92,88.21) .. controls (154.09,88.21) and (152.6,86.72) .. (152.6,84.88) -- cycle ;

\draw (284.07,145.38) node [anchor=west] [inner sep=0.75pt]  [font=\small]  {$\wh{\eta }_{1}( t_{1}) =\wh{\eta }_{2}( t_{2})$};
\draw (205.68,114.4) node [anchor=north] [inner sep=0.75pt]  [color={rgb, 255:red, 84; green, 0; blue, 213 }  ,opacity=1 ]  {$\wh{\eta }_{1}$};
\draw (244.15,176.67) node [anchor=south east] [inner sep=0.75pt]  [color={rgb, 255:red, 0; green, 116; blue, 201 }  ,opacity=1 ]  {$\wh{\eta }_{2}$};
\draw (394.23,78.85) node [anchor=south west] [inner sep=0.75pt]  [font=\small]  {$\wh{\eta }_{2}( b_{2})$};
\draw (383.77,189.42) node [anchor=north west][inner sep=0.75pt]  [font=\small]  {$\wh{\eta }_{1}( b_{1})$};
\draw (145.8,196.96) node [anchor=east] [inner sep=0.75pt]  [font=\small]  {$\wh{\eta }_{2}( a_{2})$};
\draw (153.92,81.48) node [anchor=south east] [inner sep=0.75pt]  [font=\small]  {$\wh{\eta }_{1}( a_{1})$};

\end{tikzpicture}

\caption{Example of $\wh{\F}$-transverse intersection.\label{Fig:extransverse}}
\end{center}
\end{figure}

\begin{definition}[$\wh\F$-transverse intersection]\label{DefInterTrans}
Let $\wh \phi_1$, $\wh\phi_2$ and $\wh\phi_3$ be three leaves of $\wh \F$. We say that $\wh\phi_1$ \emph{is above $\wh \phi_2$ relative to $\wh\phi_3$} if there exist disjoint paths $\wh \delta_1$ and $\wh\delta_2$ linking $\wh\phi_1$ resp. $\wh\phi_2$ to $\wh\phi_3$, disjoint from these leaves but at their extremities, and such that $\wh\delta_1\cap\wh\phi_3$ is after $\wh\delta_2\cap\wh\phi_3$ for the order on $\wh\phi_3$.

Let $\wh\eta_1:J_1\to \wh\Sigma$ and $\wh\eta_2:J_2\to \wh \Sigma$ be two transverse paths such that there exist $t_1\in J_1$ and $t_2\in J_2$ satisfying $\wh\eta_1(t_1)=\wh\eta_2(t_2)$. We will say that $\wh\eta_1$ and $\wh\eta_2$ have an \emph{$\wh{\F}$-transverse intersection} at $\wh\eta_1(t_1)=\wh\eta_2(t_2)$ (see Figure~\ref{Fig:extransverse}) 
if there exist $a_1, b_1\in J_1$ satisfying $a_1<t_1<b_1$  and $a_2, b_2\in J_2$ satisfying $a_2<t_2<b_2$ such that:
\begin{itemize}
	\item $\wh\phi_{\wh\eta_1(a_1)}$ is above $\wh\phi_{\wh\eta_2(a_2)}$ relative to $\wh\phi_{\wh\eta_2(t_2)}$;
	\item $\wh\phi_{\wh\eta_1(b_1)}$ is below $\wh\phi_{\wh\eta_2(b_2)}$ relative to $\wh\phi_{\wh\eta_2(t_2)}$.
\end{itemize}
\end{definition}

A transverse intersection means that there is a ``crossing'' between the two paths naturally defined by $\wh\eta_1$ and $\wh\eta_2$ in the space of leaves of $\widehat{\F}$, which is a one-dimensional topological manifold, usually non Hausdorff.

\bigskip
Now, let $\eta_1:J_1\to \Sigma$ and $\eta_2:J_2\to \Sigma$ be two transverse paths such that there exist $t_1\in J_1$ and $t_2\in J_2$ satisfying $\eta_1(t_1)=\eta_2(t_2)$. We say that $\eta_1$ and $\eta_2$ have an \emph{${\F}$-transverse intersection} at $\eta_1(t_1)=\eta_2(t_2)$ if, given $\wh\eta_1:J_1\to \wh \Sigma$ and $\wh\eta_2:J_2\to \wh \Sigma$ any two lifts of $\eta_1$ and $\eta_2$ such that $\wh\eta_1(t_1)=\wh\eta_2(t_2)$, we have that $\wh\eta_1$ and $\wh\eta_2$ have a $\wh{\F}$-transverse intersection at $\wh\eta_1(t_1)=\wh\eta_2(t_2)$.
If $\eta_1=\eta_2$ one speaks of a \emph{$\F$-transverse self-intersection}. In this case, if $\widehat \eta_1$ is a lift of $\eta_1$, then there exists $T\in\mathcal G$ such that $\widehat\eta_1$ and $T\widehat\eta_1$ have a $\widehat{\F}$-transverse intersection at $\widehat\eta_1(t_1)=T\widehat\eta_1(t_2)$.

\paragraph{Recurrence and equivalence.} 

We will say a transverse path $\eta:\R\to \Sigma$ is \emph{positively recurrent} if, for every $a<b$, there exist $c<d$, with $b<c$, such that $\eta|_{[a,b]}$ and $\eta|_{[c,d]}$  are equivalent. Similarly, $\eta$ is \emph{negatively recurrent} if $t\mapsto\eta(-t)$ is positively recurrent. Finally $\eta$ is \emph{recurrent} if it is both positively and negatively recurrent.
\bigskip 

Two transverse paths $\eta_1:\R\to \Sigma$ and $\eta_2:\R\to \Sigma$ are said \emph{equivalent at $+\infty$} (denoted $\eta_1\sim_{+\infty}\eta_2$) if there exists $a_1,a_2\in \R$ such that $\eta_1{}|_{[a_1,+\infty)}$ and  $\eta_2{}|_{[a_2,+\infty)}$ are equivalent. Similarly $\eta_1$ and $\eta_2$ are \emph{equivalent at $-\infty$} (denoted $\eta_1\sim_{-\infty}\eta_2$) if $t\mapsto \eta_1(-t)$ and $t\mapsto \eta_2(-t)$ are equivalent at $+\infty$.

\paragraph{Accumulation property}

We say that a transverse path $\eta_1:\R\to S$ \emph{accumulates positively} on the transverse path $\eta_2:\R\to \Sigma$ if there exist real numbers $a_1$ and $a_2<b_2$ such that  $\eta_1{}|_{[a_1,+\infty)}$ and $\eta_2{}|_{[a_2,b_2)}$ are $\F$-equivalent. Similarly, $\eta_1$ \emph{accumulates negatively} on $\eta_2$ if there exist $b_1$ and $a_2<b_2$ such that $\eta_1{}|_{(-\infty,b_1]}$ and $\eta_2{}|_{(a_2,b_2]}$ are $\F$-equivalent. 
Finally $\eta_1$ \emph{accumulates} on $\eta_2$ if it accumulates  positively or negatively on $\eta_2$.

\paragraph{Brouwer-Le Calvez foliations and forcing theory}

If $\F$ is a singular foliation of a surface $S$, denote $\mathrm{Sing}(\F)$ the set of singularities of $\F$, and $\dom(\F) = S \setminus \mathrm{Sing}(\F)$. 
The forcing theory grounds on the following result of existence of transverse foliations, which can be obtained as a combination of the main theorems of \cite{lecalvezfoliations} and \cite{bguin2016fixed}.

\begin{theorem}\label{TheoExistIstop}
Let $S$ be a surface and $f\in\Homeo_0(S)$.
Then there exist an identity isotopy $I$ for $f$ and a transverse topological oriented singular foliation $\F$ on $S$ with $\dom(\F) = \dom(I)$, such that:
For any $z\in \dom(\F)$, there exists an $\F$-transverse path $\big(I_\F^t(z)\big)_{t\in[0,1]}$ linking $z$ to $f(z)$ and that is homotopic in $\dom(\F)$, relative to endpoints, to the isotopy path $(I^t(z))_{t\in[0,1]}$.
\end{theorem}

This allows to define the path $I_{\F}^\Z (x)$ as the concatenation of the paths $\big(I_\F^t(f^n(z))\big)_{t\in[0,1]}$ for $n\in\Z$.

The following statement is a reformulation of the main technical result of the forcing theory \cite[Proposition~20]{lct1}:

\begin{prop}\label{propFondalct1}
Suppose that $I^{[t, t']}_\F(z)$ and $I^{[s,s']}_\F(z')$ intersect $\F$-transversally at $I^{t''}_\F(z) = I^{s''}_\F(z')$. Then the path $I^{[t, t'']}_\F(z) I^{[s'',s']}_\F(z')$ is $f$-admissible or order $t'-t+s'-s$.
\end{prop}

\subsection{Classification of ergodic rotation sets}\label{SubSecRotAlepablo}

%

The following is contained in \cite[Theorem~F]{alepablo}.

\begin{theorem}[Shape of ergodic rotation sets]\label{thm:ShapeRotationSet}
Let $f \in \Homeo_0(S)$, where $S$ has genus $g$. Then, its ergodic rotation set $\rote(f)$ can be written as
\[\rote(f) = \rho^1 \cup \rho^{\mathrm h},\]
where
\begin{enumerate}
\item The set $\rho^1$ is included in the union of at most $3g-3$ lines.
\item The set $\rho^{\mathrm h}$ is the union of at most $2g-2$ sets $(\rho_i)_{i\in I_{\mathrm h}}$, such that, for every ${i\in I_{\mathrm h}}$:
\begin{itemize}
\item The set ${\rho_i}$ spans a linear subspace $V_i$ which has a basis formed by elements of $H_1(S,\Z)$;
\item The set $\overline{\rho_i}$ is a convex set containing $0$;
\item We have $\operatorname{int}_{V_i}(\overline{\rho_i})= \operatorname{int}_{V_i}(\rho_i)$ (in other words, $\rho_i$ is convex up to the fact that elements of $\partial_{V_i}(\rho_i)\setminus\operatorname{extrem}(\rho_i)$ can be in the complement of $\rho_i$);
\item Every element of $\operatorname{int}_{V_i}(\rho_i) \cap H_1(S,\Q)$ is the rotation vector of some $f$-periodic orbit (because $V_i$ has a rational basis, such elements are dense in $\operatorname{int}_{V_i}(\rho_i)$). 
\end{itemize} 
\end{enumerate}
\end{theorem}

Let us define some surfaces associated with the classes $\cl_i$, \(i \in I_{\mathrm h}\). Consider the projection $\Lambda_i$ of $\dot\Lambda_i$ on $S$, and the lift $\wt\Lambda_i$ of $\Lambda_i$ to $\wt S$. Take a connected component $\wt C$ of $\wt\Lambda_i$, denote $\wt S_i = \conv(\wt C)$ (for the hyperbolic metric) and set $S_i$ as the projection of $\wt S_i$ on $S$ (see Figure~\ref{FigTree0} page~\pageref{FigTree0} for an example of such surfaces). 
\cite[Lemma~6.7]{alepablo} asserts that $S_i$ is an open surface whose boundary is made of closed geodesics, and \cite[Lemma~6.8]{alepablo} states that for $i,j\in I_{\mathrm{h}}$, $i\neq j$, one has $S_i\cap S_j = \emptyset$. 
\bigskip

Let us finish this subsection with two technical results. 
%
%

\begin{prop}\label{LemNotSimpleTracking}
Let $f\in \Homeo_0(S)$, $\mu$ a measure belonging to a chaotic class and $z$ a $\mu$-typical point. Then for any $\varep>0$ there exists $z'$ a periodic orbit of $f$, belonging to the same chaotic class\footnote{Recall that when $z$ is periodic, we say that $z\in\cl_i$ if the uniform measure on the orbit of $z$ belongs to $\cl_i$.} as $z$, whose tracking geodesic $\gamma_{z'}$ is not simple and has a lift $\wt\gamma_{\wt z'}$ to $\wt S$ that is $\varep$-close to a lift $\wt\gamma_{\wt z}$ of $\gamma_z$ to $\wt S$, and such that $\|\rho(z)-\rho(z')\|\le \varep$. 
\end{prop}

\begin{proof}
Let $\mu\in \cl_i$ for some $i\in I_{\mathrm h}$ and $z\in S$ that is $\mu$-typical. By definition, there exist $\mu''\in \cl_i$ and $z''$ that is typical for $\mu''$ such that $\gamma_z$ and $\gamma_{z''}$ intersect transversally. Let $\wt z$ and $\wt z''$ be lifts of $z$ and $z''$ to $\wt S$ such that $\wt \gamma_{\wt z}$ and $\wt \gamma_{\wt z''}$ intersect transversally.

Let $\varep>0$. By \cite[Theorem~5.8]{alepablo}, there exists $\mu'\in \cl_i$, $z'$ that is typical for $\mu'$ and periodic and $\wt z'_1,\wt z'_2$ two lifts of $z'$ such that $\rho(z')\in B(\rho(z),\varepsilon)$, $d(\wt\gamma_{\wt z'_1},\wt\gamma_{\wt z})<\varepsilon$ and $d(\wt\gamma_{\wt z'_2},\wt\gamma_{\wt z''})<\varepsilon$. As $\wt \gamma_{\wt z}$ and $\wt \gamma_{\wt z''}$ intersect transversally, if $\varep$ is small enough, the two geodesics $\wt\gamma_{\wt z'_1}$ and $\wt\gamma_{\wt z'_2}$ intersect transversally, which means that $\gamma_{z'}$ is not simple.
\end{proof}

\begin{lemma}\label{LemNotSimpleIntersect}
Let $f\in \Homeo_0(S)$ and $\mu_1, \mu_2\in\Merg(f)$. Suppose that $\mu_1$ and $\mu_2$ are dynamically transverse and that neither $\dot\Lambda_{\mu_1}$ nor $\dot\Lambda_{\mu_2}$ are made of a single simple closed geodesic.

Then for $\mu_1$-a.e.\ $z_1$ and $\mu_2$-a.e.\ $z_2$ the transverse trajectories $I^\Z_\F(z_1)$ and $I^\Z_\F(z_2)$ intersect $\F$-transversally.
More precisely, if $\wt z_1$ and $\wt z_2$ are lifts of $z_1$ and $z_2$ to $\wt S$ such that $\wt\gamma_{\wt z_1}$ and $\wt\gamma_{\wt z_2}$ intersect transversally, then the transverse trajectories $I^\Z_{\wt\F}(\wt z_1)$ and $I^\Z_{\wt\F}(\wt z_2)$ intersect $\wt\F$-transversally.
\end{lemma}

Note that this lemma can be applied to a single measure $\mu$ such that $\dot\Lambda_{\mu}$ is not a geodesic lamination, it implies that for $\mu$-a.e.\ $z$ the transverse trajectory $I^\Z_\F(z)$ has a self $\F$-transverse intersection.

\begin{proof}
By the proof of \cite[Theorem~5.8]{alepablo}, there are three possibilities (as explained in the beginning of Paragraph 5.3.1, the very end of Paragraph 5.3.1, and the beginning of Paragraph 5.3.2 of \cite{alepablo}):
\begin{enumerate}
\item\label{item1} either $I^\Z_\F(z_1)$ accumulates in $I^\Z_\F(z_2)$;
\item\label{item2} or $I^\Z_\F(z_2)$ accumulates in $I^\Z_\F(z_1)$;
\item\label{item3} or $I^\Z_\F(z_1)$ and $I^\Z_\F(z_2)$ intersect $\F$-transversally.
\end{enumerate}
But both \ref{item1}.~and \ref{item2}.~are impossible, because of \cite[Proposition 3.3]{guiheneuf2023area}.
\end{proof}

\subsection{Bounded deviations in homotopy}

An important part of this article's proofs is based on the following criterion of existence of periodic orbits with certain rotational behaviour \cite[Corollary~4.10]{paper1PAF}.

For $\alpha\subset S$ a loop and $\beta : [a,b]\to S$ a path, we call \emph{geometric intersection number} between $\alpha$ and $\beta$ the minimal number of sets $T\wt\alpha$ a path homotopic to $\wt\beta$ rel.~endpoints intersects, where $T\in\G$ and $\wt\alpha$, $\wt\beta$ are lifts of $\alpha$ and $\beta$ to $\wt S$.

For $E$ a set and $R>0$, denote $B_R(E) = \{x\mid d(x, E)<R\}$.

\begin{theorem}\label{ThGT}
Let $f\in\Homeo_0(S)$ and $\gamma_1, \gamma_2$ two closed geodesics that are tracking geodesics for some $f$-ergodic measures and that are not simple geodesics. 
Let $T_1, T_2\in \G$ be primitive deck transformations associated to these closed geodesics.

Then there exist periodic points $z_1$ and $z_2$ such that $\gamma_{z_1} = \gamma_1$ and $\gamma_{z_2} = \gamma_2$.

Moreover, for any $M>0$ there exists $D'>0$ and $m_1\ge 0$ such that the following is true.
For $i=1,2$, suppose that there exist $4$ deck transformations $(R_i^j)_{1\le j\le 4}\in\G$ such that the following properties hold:
\begin{itemize}
\item the sets $R_i^j B_{D'}(\wt\gamma_i)$ are pairwise disjoint and have the same orientation;
\item there exists $0\le n'_0\le n_0$, with $n'_0\ge m_1$ and $n_0-n'_0\ge m_1$ such that for any $1\le j\le 4$, the points $\wt y_0$ and $\wt f^{n'_0}(\wt y_0)$ lie in different sides of the complement of $R_1^j B_{D'}(\wt\gamma_1)$, and the points $\wt f^{n'_0}(\wt y_0)$ and $\wt f^{n_0}(\wt y_0)$ lie in different sides of the complement of $R_2^j B_{D'}(\wt\gamma_2)$.
\end{itemize}

Then there exists an $\wt f$-admissible transverse path $\wt\beta$ of order $n_0+2m_1$ and parametrized by $[t_0, t_2]$, and some $t_1\in (t_0, t_2)$ such that $\wt\beta|_{[t_0, t_1]}$ and $R_3^1 T_1^3(R_2^1)^{-1}\wt\beta|_{[t_0, t_1]}$ intersect $\F$-transversally, and that $\wt\beta|_{[t_1, t_2]}$ and $R_2^2 T_2^{-3}(R_3^2)^{-1}\wt\beta|_{[t_1, t_2]}$ intersect $\F$-transversally.

The path $\wt\beta$ is made of the concatenation of some paths $I^{[s_1,t_1]}_\F(z_1)$, $I^{[u_1, u_2]}_\F(y_0)$ and $I^{[s_2,t_2]}_\F(z_2)$, with $t_1-s_1\ge M$ and $t_2-s_2\ge M$.

Finally, if $\gamma_1 = \gamma_2$, then there exists a constant $d_0>0$ depending only on $z$ (and neither on $y_0$ nor on $n_0$) such that the tracking geodesic $\gamma_p$ of $p$ is freely homotopic to the concatenation $I^{[t_2,t_3]}_{\F}( y_0)\delta$, where $\diam(\wt\delta)\le d_0$ (with $\wt\delta$ a lift of $\delta$ to $\wt S$). 
\end{theorem}

This theorem will often be combined with the following result \cite[Lemma~2.2]{paper1PAF}.

\begin{lemma}\label{LemUseResidFinite}
Let $\gamma$ be a closed geodesic on $S$. Then for any $M_0>0$ and any $R>0$, there exists $N_0\in\N$ such that for any path $\alpha : [0,1]\to S$ whose geometric intersection number with $\gamma$ is bigger than $N_0$, any lift $\wt\alpha$ of $\alpha$ to $\wt S$ crosses geometrically $M_0$ lifts of $\gamma$ that are pairwise disjoint, have the same orientation and are pairwise at distance $\ge R$.
\end{lemma}

\section{Heteroclinic horseshoes in forcing theory}

\subsection{Markovian intersections}\label{SubSecmarkov}

We now recall some properties of Markovian intersections as stated in \cite[Chapter 9, Section 2]{pa}. 
Note that \cite[Proposition 9.12]{pa} is false and is replaced here by Proposition~\ref{LemPointFixe}, which is sufficient in practice (and also in all the applications made in \cite{pa}).

\begin{definition}\label{Defmarkov}
Let $S$ be a surface. We call \emph{rectangle} of $S$ a subset $R \subset S$ satisfying $R = h([0,1]^2)$ for some homeomorphism $h: [0,1]^2\to h([0,1]^2)\subset S$. We call \emph{sides} of $R$ the image by $h$ of the sides of $[0,1]^2$. We call \emph{horizontal} the sides $R^- = h([0,1] \times \left\{ 0 \right\})$ and $R^+ = h([0,1] \times \left\{ 1 \right\})$ and \emph{vertical} the two others. We say that a rectangle $R' \subset R$ is a \emph{horizontal} (resp. vertical) \emph{subrectangle} of $R$ if the vertical (resp. horizontal) sides of $R'$ are included in those of $R$.
\end{definition}

Note that the relation ``being a horizontal subrectangle'' is transitive: a horizontal subrectangle $R''$ of a horizontal subrectangle $R'$ of a rectangle $R$ is a horizontal subrectangle of $R$ (and the same holds for vertical subrectangles). 

Given $x \in \R^2$, we will denote by $\pi_2(x)$ its second coordinate. Following \cite{MR2060531}, we define Markovian intersections in the following way:

\begin{definition}\label{def:markov}
Let $R_1$ and $R_2$ be two rectangles of a surface $S$. We say that the intersection $R_1 \cap R_2$ is \emph{pre-Markovian} if there exists a homeomorphism $h$ from a neighbourhood of $R_1\cup R_2$ to an open subset of $\R^2$ such that (see Figure~\ref{FigExMarkov}):
\begin{itemize}
\item $h(R_2) = [0,1]^2$;
\item either $h(R_1^+) \subset \left\{x\mid \pi_2(x) > 1 \right\}$ and $h(R_1^-) \subset \left\{x\mid \pi_2(x) < 0 \right\}$,\\
or $h(R_1^-) \subset \left\{x\mid \pi_2(x) > 1 \right\}$ and $h(R_1^+) \subset \left\{x\mid \pi_2(x) < 0 \right\}$;
\item $h(R_1) \subset \left\{x\mid \pi_2(x) < 0 \right\} \cup [0,1]^2 \cup \left\{x\mid \pi_2(x) > 1 \right\}$.
\end{itemize}

We say that the intersection $R_1 \cap R_2$ is \emph{Markovian}, and denote it $R_1\cap_M R_2$, if there exists a horizontal subrectangle $R'_1$ of $R_1$ such that the intersection $R'_1\cap R_2$ is pre-Markovian\footnote{Equivalently, one can replace this definition by asking that there exist a vertical subrectangle $R'_2$ of $R_2$ such that the intersection $R_1\cap R'_2$ is pre-Markovian.}.
\end{definition}

\begin{figure}[h!]
\begin{center}

\tikzset{every picture/.style={line width=0.75pt}} 

\begin{tikzpicture}[x=0.75pt,y=0.75pt,yscale=-.7,xscale=.7]

\draw  [fill={rgb, 255:red, 100; green, 100; blue, 100 }  ,fill opacity=.3 ] (210,120) -- (370,120) -- (370,280) -- (210,280) -- cycle ;
\draw  [dash pattern={on 4.5pt off 4.5pt}]  (100,120) -- (210,120) ;
\draw  [dash pattern={on 4.5pt off 4.5pt}]  (370,120) -- (470,120) ;
\draw  [dash pattern={on 4.5pt off 4.5pt}]  (100,280) -- (210,280) ;
\draw  [dash pattern={on 4.5pt off 4.5pt}]  (370,280) -- (470,280) ;
\draw  [fill={rgb, 255:red, 107; green, 192; blue, 254 }  ,fill opacity=.5 ] (268.15,116.25) .. controls (215,182.25) and (208,66.75) .. (208.15,66.25) .. controls (208.29,65.75) and (265.5,39.25) .. (300,50) .. controls (334.5,60.75) and (360.67,70.17) .. (360,70) .. controls (359.33,69.83) and (353,94.25) .. (348.15,116.25) .. controls (343.29,138.25) and (292.15,298) .. (338.15,306.25) .. controls (384.15,314.5) and (377.29,346.67) .. (378.15,346.25) .. controls (379,345.83) and (321.5,329.25) .. (300,330) .. controls (278.5,330.75) and (227.96,346) .. (228.15,346.25) .. controls (228.33,346.5) and (246,290.25) .. (230,250) .. controls (214,209.75) and (321.29,50.25) .. (268.15,116.25) -- cycle ;
\draw [color={rgb, 255:red, 208; green, 2; blue, 27 }  ,draw opacity=1, line width=1.2pt ]   (208.15,66.25) .. controls (207.5,66.25) and (265.67,39.89) .. (300,50) .. controls (334.33,60.11) and (360.5,70.75) .. (360,70) ;
\draw [color={rgb, 255:red, 208; green, 2; blue, 27 }  ,draw opacity=1, line width=1.2pt ]   (228.15,346.25) .. controls (227.67,345) and (283.44,330.56) .. (300,330) .. controls (316.56,329.44) and (379,345.22) .. (378.15,346.25) ;

\draw (378,222.4) node [anchor=north west][inner sep=0.75pt]    {$h(R_{2})$};
\draw (247,200) node [anchor=north west][inner sep=0.75pt]  [color={rgb, 255:red, 0; green, 9; blue, 99 }  ,opacity=1 ]  {$h(R_{1})$};
\draw (246,55) node [anchor=north west][inner sep=0.75pt]  [color={rgb, 255:red, 142; green, 0; blue, 16 }  ,opacity=1 ]  {$h(R_{1}^{+})$};
\draw (251,295) node [anchor=north west][inner sep=0.75pt]  [color={rgb, 255:red, 142; green, 0; blue, 16 }  ,opacity=1 ]  {$h(R_{1}^{-})$};
\end{tikzpicture}
\hfill
\begin{tikzpicture}[x=0.75pt,y=0.75pt,yscale=-.95,xscale=.95]

\draw  [fill={rgb, 255:red, 74; green, 74; blue, 74 }  ,fill opacity=0.08 ] (260,80) -- (380,80) -- (380,200) -- (260,200) -- cycle ;
\draw  [dash pattern={on 4.5pt off 4.5pt}]  (200,80) -- (260,80) ;
\draw  [dash pattern={on 4.5pt off 4.5pt}]  (380,80) -- (440,80) ;
\draw  [dash pattern={on 4.5pt off 4.5pt}]  (200,200) -- (260,200) ;
\draw  [dash pattern={on 4.5pt off 4.5pt}]  (380,200) -- (440,200) ;
\draw  [fill={rgb, 255:red, 0; green, 118; blue, 255 }  ,fill opacity=0.15 ] (370,120) .. controls (364.6,119) and (357.4,117.8) .. (350,120) .. controls (353,96.6) and (336.86,53.08) .. (330,70) .. controls (323.14,86.92) and (357.6,210) .. (355.8,232.2) .. controls (354,254.4) and (324.17,263.68) .. (306.2,252.6) .. controls (288.23,241.52) and (277,113.8) .. (280,50) .. controls (294.38,50.26) and (288.6,50.6) .. (300,50) .. controls (304.2,151.4) and (308.63,239.17) .. (330,240) .. controls (351.37,240.83) and (310.7,76.82) .. (320,50) .. controls (329.3,23.18) and (371.4,69.4) .. (370,120) -- cycle ;
\draw  [draw opacity=0][fill={rgb, 255:red, 0; green, 114; blue, 255 }  ,fill opacity=0.3 ] (330,70) .. controls (323.14,86.92) and (357.6,210) .. (355.8,232.2) .. controls (346.11,233.44) and (340.11,247) .. (330,240) .. controls (351.37,240.83) and (310.7,76.82) .. (320,50) .. controls (324.11,55.22) and (325.22,62.33) .. (330,70) -- cycle ;
\draw [color={rgb, 255:red, 208; green, 2; blue, 27 }  ,draw opacity=1 ]   (320,50) .. controls (325.22,59.89) and (329.22,70.33) .. (330,70) ;
\draw [color={rgb, 255:red, 208; green, 2; blue, 27 }  ,draw opacity=1 ]   (330,240) .. controls (338.33,247.44) and (348.11,233.44) .. (355.8,232.2) ;

\draw (293.57,225.19) node [anchor=east] [inner sep=0.75pt]  [color={rgb, 255:red, 0; green, 89; blue, 193 }  ,opacity=1 ]  {$h( R_{1})$};
\draw (259.23,132.52) node [anchor=east] [inner sep=0.75pt]  [color={rgb, 255:red, 0; green, 0; blue, 0 }  ,opacity=1 ]  {$h( R_{2})$};
\draw (340.33,135.86) node [anchor=west] [inner sep=0.75pt]  [color={rgb, 255:red, 0; green, 89; blue, 193 }  ,opacity=1 ]  {$h( R_{1}')$};

\end{tikzpicture}

\caption{A pre-Markovian intersection (left) and a Markovian intersection(right). The horizontal sub-rectangle for the pre-markovian intersection is denoted $R'_1$.\label{FigExMarkov}}

\end{center}
\end{figure}
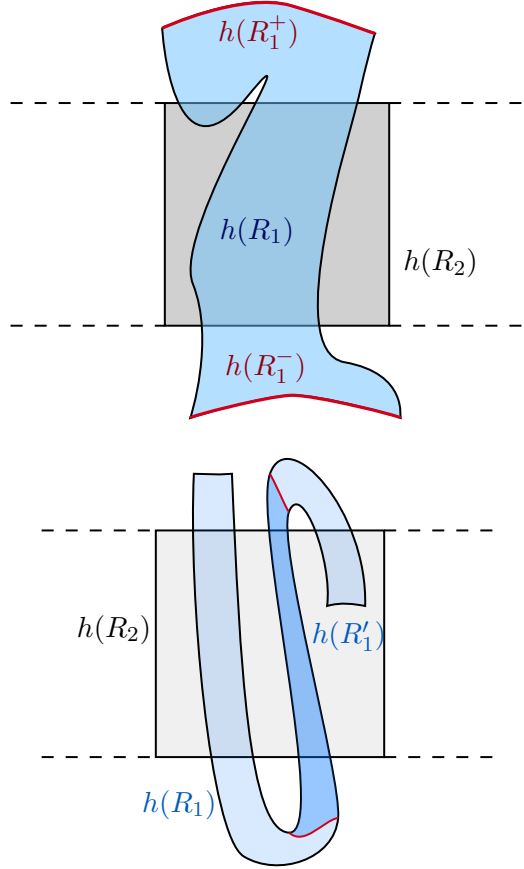

The following is a particular case of Homma's generalization \cite{MR58194} of Schoenflies theorem.

\begin{theorem}[Homma]\label{PropHomma}
Any homeomorphism of 
\[\Big(\big((\R\times\{0\})\cup (\R\times \{1\}) \cup (\{0\}\times [0,1]) \cup (\{1\}\times [0,1]) \big)\cap B(0,10)\Big)\cup \partial B(0,10)\]
to its image in $\R^2$ can be extended to a self-homeomorphism of $\R^2$.
\end{theorem}

Homma's theorem will be used to find rectangles and Markovian intersections.

\begin{rem}\label{RemHomma}
Homma's theorem (Theorem~\ref{PropHomma}) also implies directly that if the intersection $R_1\cap R_2$ is pre-Markovian, then for any vertical subrectangle $R'_1$ of $R_1$ and any horizontal subrectangle $R'_2$ of $R_2$, the intersection $R'_1\cap R'_2$ is also pre-Markovian.
\end{rem}

The proof of the following result can be obtained as a combination of Theorem 16 and Corollary 12 of \cite{MR2060531}. They are stated in terms of (following our terminology) pre-Markovian intersections but the previous paragraph ensures they are also valid for Markovian intersections. 
%

\begin{prop}\label{LemPointFixe}
Let $(R_i)_{0\le i\le n}$ be rectangles and $(f_i)_{1\le i\le n}$ be homeomorphisms of $S$ such that for any $1\le i\le n$, the intersection $f_i(R_{i-1}) \cap R_i$ is Markovian. Then there exists $x\in \inte(R_0)$ such that for any $1\le i\le n$, we have $f_if_{i-1}\dots f_1 (x)\in \inte(R_i)$.

Moreover, if $R_0=R_n$, then we can suppose that $f_nf_{n-1}\dots f_1 (x) = x$.
\end{prop}

\begin{proof}
Let us prove that the property for pre-Markovian intersections implies the property for Markovian intersections.

For any $1\le i\le n$, the intersection $f_i(R_{i-1}) \cap R_i$ is Markovian, so there exists a horizontal subrectangle $R'_{i-1}$ of $R_{i-1}$ such that the intersection $f_i(R'_{i-1})\cap R_i$ is pre-Markovian. Using Remark~\ref{RemHomma}, we deduce that for any $1\le i\le n$, the intersection $f_i(R'_{i-1}) \cap R'_i$ is pre-Markovian, and one can apply the property for pre-Markovian intersections.
\end{proof}

The next property is a direct consequence of the definition.

\begin{lemma}\label{LemPerturbHorse}
Let $R_1, R_2$ be two rectangles such that the intersection $R_1 \cap R_2$ is Markovian. Then there exists a neighbourhood $V$ of $\Id_S$ in $\Homeo(S)$ such that for any $g\in V$, the intersection $g(R_1) \cap R_2$ is Markovian.
\end{lemma}

The following definition is a variation over the concept of rotational horseshoe defined in \cite{MR3784518} and used in \cite{lct2}.

\begin{definition}\label{DefRotHorse}
Let $S$ be a surface with negative Euler characteristic and $f$ a homeomorphism of $S$. We denote by $\wt f$ the canonical lift of $f$ to $\wt S\simeq\Hy^2$.

We say that $f$ has a \emph{rotational horseshoe with deck transformations $T_1,\dots,T_k$} if there exists a rectangle $R$ of $\wt S$ such that, for any $1\le i \le k$, the intersection $T_i R \cap \wt f(R)$ is Markovian.
\end{definition}

For any finite set $\{1,\dots,k\}^\Z$, we denote by $\sigma : \{1,\dots,k\}^\Z \rightarrow \{1,\dots,k\}^\Z$ the shift map, \emph{i.e} the map which, to a sequence $(a_{i})_{i \in \Z}$, associates the sequence $(a_{i+1})_{i \in \Z}$. 

From Proposition~\ref{LemPointFixe}, it follows a ``semi-conjugacy'' result (which allows to link our notion of horseshoe with the one of \cite{lct2}), see Propositions~9.16 and 9.17 of \cite{pa}. 

\subsection{Heteroclinic connections of horseshoes and rotation sets}

\begin{definition}\label{DefConnecRect}
Let $R_1$ and $R_2$ be two rectangles. If there exists $n\in\N$ and $T\in \G$ such that the intersection $\wt f^{n}(R_1)\cap TR_{2}$ is Markovian, we denote $R_1\to R_2$. We will also use labels on the edges: in the above configuration we will denote $R_1\overset\tau\to R_2$, where $\tau = (n,T)$.
\end{definition}

This allows to talk about the graph spanned by a family of rectangles $(R_i)_{i\in I}\subset \wt S$ and Markivian intersections between them: $G$ is the (multi)graph whose vertices are the $(R_i)_{i\in I}$ and whose edges are of the form $R_i\overset\tau\to R_j$.

\begin{definition}\label{DefConnectRectEnsRot}
Let $f\in\Homeo_0(S)$ and $\wt f$ a lift of $f$ to $\wt S$.
Suppose that there exists a family $I$ (not necessarily finite) and rectangles $(R_i)_{i\in I}\subset \wt S$ such that for any $i\in I$, the rectangle $R_i$ is a rotational horseshoe with deck transformations $T_1^i,\dots,T_{k_i}^i$ for $f^{r_i}$. For $i\in I$, denote
\[\rot_i = \conv\left\{\frac{[T_j^i]}{r_i} \ \big|\  1\le j\le k_i\right\},\]
and
\[\rot(G) = \bigcup_{R_{i_1}\to R_{i_2}\to \cdots \to R_{i_\ell}}\conv\Big(\bigcup_{1\le k\le\ell}\rot_{i_k} \Big).\]
\end{definition}

Recall that a graph $G$ is \emph{strongly connected} if for any two edges of $G$ there exists a path going from the first one to the second one and a path going from the second one to the first one.
The following proposition says that if one considers a path in the graph spanned by rectangles, the elements of the convex hull of the rotation sets of rotational horseshoes associated to those rectangles are in fact rotation vectors of the homeomorphism. If one replaces ``path'' by ``strongly connected connected component'', then the obtained elements are moreover realised as rotation vectors of some orbit. 

\begin{prop}\label{PropConnectRectEnsRot}
Let $f\in\Homeo_0(S)$ and $\wt f$ a lift of $f$ to $\wt S$. Let us place ourselves within the framework of Definition~\ref{DefConnectRectEnsRot}.

Then:
\begin{enumerate}
\item\label{P1PropConnectRectEnsRot} We have
\[\overline{\rot(G)} \subset \rot(f);\]
\item\label{P2PropConnectRectEnsRot} if $G$ is strongly connected, then any element of $\overline{\rot(G)}$ is realised as the rotation vectors of a point;
\item\label{P3PropConnectRectEnsRot} if $G$ is strongly connected, then for any $\rho\in \inte(\rot(G))$, there exists a compact $f$-invariant set $K_\rho\subset S$ such that for any $x\in K_\rho$ we have $\rho(x) = \{\rho\}$;
\item\label{P4PropConnectRectEnsRot} if $G$ is strongly connected, then for all compact connected set $C\subset\inte(\rot(G))$, there exists $x\in S$ such that $\rho(x) = C$.
\end{enumerate}
\end{prop}


The proof of this proposition is quite technical in terms of notations but rather straightforward. Points~\ref{P3PropConnectRectEnsRot}.~ and \ref{P4PropConnectRectEnsRot}.~will be obtained as direct consequences of \cite[Theorem A]{zbMATH00009916} and \cite[Theorem~1, (iv)]{llibremackay} (the arguments for Markov partitions of pseudo-Anosov maps used in these papers adapt directly to the case of Markovian intersections).

\begin{proof}
\textbf{Proof of Point~\ref{P1PropConnectRectEnsRot}.}
The rotation set of $f$ being closed, it is sufficient to prove that for any $i_1,\dots,i_\ell$ such that $R_{i_1}\to R_{i_2}\to \cdots \to R_{i_\ell}$, we have
\[\conv\Big(\bigcup_{1\le k\le\ell}\rot_{i_k} \Big) \subset \rot(f).\]

For any edge $w$ of $G$, denote denote $\tau(w)$ its label: $\tau(w) = (n(w), T(w))\in \N^*\times\G$, and $s(w)$ and $e(w)$ its starting and ending vertices.

If $(w_k)_{0\le k\le k_0}$ is a finite path, one can define 
\[\rho(w_k) = \frac{\sum_{j=0}^{k_0} [T(w_j)]}{\sum_{j=0}^{k_0} n(w_j)} \in H_1(S,\R).\]

Let us consider a subgraph $G'$ of $G$ whose vertices are the $R_{i_1}, \cdots , R_{i_\ell}$ and whose edges are 
\begin{itemize}
\item the edges of $G$ from one rectangle $R_{i_m}$ to itself coming from the rotational horseshoe;
\item for any $1\le m<\ell$, one edge $w'_m$ from $R_{i_m}$ to $R_{i_{m+1}}$.
\end{itemize}
The graph $G'$ can be supposed to have the following form:
\begin{center}
\begin{tikzpicture}
\node (A) at (0,0) {$R_{i_1}$};
\node (B) at (2.5,0) {$R_{i_2}$};
\node (C) at (5,0) {$\cdots$};
\node (D) at (7.5,0) {$R_{i_\ell}$};
\draw[->, >=stealth] (A) -- (B) node[midway, above] {$\tau_{i_1}$};
\draw[->, >=stealth] (B) -- (C) node[midway, above] {$\tau_{i_1}$};
\draw[->, >=stealth] (C) -- (D) node[midway, above] {$\tau_{i_{\ell-1}}$};
\draw[->, >=stealth] (A) to [looseness=5, out= 70, in=110]node[midway, above] {$(r_{i_1}, T_1^{i_1})$} (A);
\draw[->, >=stealth] (A) to [looseness=5, out= -70, in=-110]node[midway, below] {$(r_{i_1}, T_{k_{i_1}}^{i_1})$} (A);
\draw[->, >=stealth] (B) to [looseness=5, out= 70, in=110]node[midway, above] {$(r_{i_2}, T_1^{i_2})$} (B);
\draw[->, >=stealth] (B) to [looseness=5, out= -70, in=-110]node[midway, below] {$(r_{i_2}, T_{k_{i_2}}^{i_2})$} (B);
\draw[->, >=stealth] (D) to [looseness=5, out= 70, in=110]node[midway, above] {$(r_{i_\ell}, T_1^{i_\ell})$} (D);
\draw[->, >=stealth] (D) to [looseness=5, out= -70, in=-110]node[midway, below] {$(r_{i_\ell}, T_{k_{i_\ell}}^{i_\ell})$} (D);
\end{tikzpicture}
\end{center}


Let $\rho\in \conv\big\{\rot_{i_m} \mid 1\le k\le m\big\}$. This implies that there is $\sigma_{i_1},\dots,\sigma_{i_\ell} \in [0,1]^\ell$ such that $\sum_{m=1}^\ell \sigma_{i_m} = 1$ and, for all $1\le m\le \ell$, some $\rho_m\in \rot_{i_m}$ such that $\rho = \sum_{m=1}^\ell \sigma_{i_m} \rho_m$. We endow $H_1(S,\R)\simeq \R^{2g}$ with a norm $\|\cdot\|$. 

Given $\varep>0$, each $\rho_m$ is approximated by the rotation vector of some finite path $(w_k^m)_{0\le k\le k_m}$ living in the subgraph of $G'$ made of all edges going from $R_{i_m}$ to $R_{i_m}$: 
\begin{equation}\label{EqQQ0}
\big\|\rho_m-\rho((w_k^m))\big\|\le\varep.
\end{equation} 

For any $q\in\N$ large enough, choose a family $(p_{m}^q)_{1\le m\le \ell}$ of positive integers such that for any $1\le m\le \ell$, we have 
\begin{equation}\label{EqQQ2}
\frac{p_{m}^q k_m r_{i_m}}{q}\underset{q\to +\infty}\longrightarrow \sigma_{i_m}.
\end{equation}
This implies that 
\begin{equation}\label{EqQQ1}
\sum_{1\le m\le \ell} p_{m}^q k_m r_{i_m} \underset{q\to+\infty}\sim q.
\end{equation}
For any $1\le m\le \ell-1$, denote $w'_m$ the edge linking $R_{i_m}$ to $R_{i_{m+1}}$. 

Using Proposition~\ref{LemPointFixe}, for any path $(w_k)_k$ in $G$ (finite or infinite), there exists $\wt x\in R_{s(w_0)}\subset \wt S$ such that for any $k$, we have
\begin{equation}\label{EqDisplaceW}
\wt f^{\,\sum_{j=0}^{k} n(w_j)}(\wt x) \in T(w_0)T(w_1) \cdots T(w_{k})\, R_{e(w_{k})}.
\end{equation}
This is in particular true for the path 
\[(W^q) := (w_k^1)^{p_1^q}w'_1 (w_k^2)^{p_2^q}w'_2\cdots w'_{\ell-1}(w_k^\ell)^{p_\ell^q}\]
of $G'$, so there exists $\wt x^q \in R_{i_1}$ and $T'_q\in\G$ such that \eqref{EqDisplaceW} holds for the path $(W^q)$; in other words
$\wt f^{\tau^q}(\wt x^q) \in T'_q R_{i_\ell}$, with $\tau^q = \sum_{m=1}^{\ell} p_m^q k_m r_{i_m} + \sum_{m=1}^{\ell-1}n(w'_m)$. 
A fundamental domain $D\subset \wt S$ of $S$ being fixed, there exists $T_1^q, T_\ell^q\in \G$ such that $\wt x^q \in T_1^q D$ and $\wt f^{\tau^q}(\wt x^q) \in T'_q T_\ell^q D$;
as the Hausdorff distances between $R_{i_1}$ and $D$, and between $R_{i_\ell}$ and $D$, are finite, the homology classes $[T_1^q]$ and $[T_\ell^q]$ are uniformly bounded in $q$. 
It remains to compute
\[\frac{\big[T'_q T_\ell^q(T_1^q)^{-1}\big]}{\tau^q} = \frac{\sum_{m=1}^{\ell} p_{m}^q k_m r_{i_m} \rho((w_k^m)) + \sum_{m=1}^{\ell-1}[T(w'_m)] + [T_\ell^q] - [T_1^q]}{\sum_{m=1}^{\ell} p_m^q k_m r_{i_m} + \sum_{m=1}^{\ell-1}n(w'_m)}\]
Because of \eqref{EqQQ1}, and because of the boundedness of $\sum_{m=1}^{\ell-1}[T(w'_m)] + [T_\ell^q] - [T_1^q]$ and $\sum_{m=1}^{\ell-1}n(w'_m)$, we deduce that
\[\frac{\big[T'_q T_\ell^q(T_1^q)^{-1}\big]}{\tau^q} \underset{q\to+\infty}\sim \frac{\sum_{m=1}^{\ell} p_{m}^q k_m r_{i_m} \rho((w_k^m))}{q}\underset{q\to+\infty}\sim \sum_{m=1}^{\ell}  \sigma_{i_m} \rho((w_k^m))\]
(the second equivalence is due to \eqref{EqQQ2}). Recall that $\sum_{m=1}^\ell \sigma_{i_m} = 1$, hence by \eqref{EqQQ0} for any $q$ large enough we have 
\[\left\|\frac{\big[T'_q T_\ell^q(T_1^q)^{-1}\big]}{\tau^q} - \rho\right\|
 = \left\|\frac{\big[T'_q T_\ell^q(T_1^q)^{-1}\big]}{\tau^q} - \sum_{m=1}^{\ell}  \sigma_{i_m} \rho_m\right\| \le 2\varep.\]
\bigskip

\noindent\textbf{Proof of Point~\ref{P2PropConnectRectEnsRot}.}
The general idea is quite similar to the one of the first part.

By the fact that $G$ is strongly connected, it suffices to prove that for any $i_0\in I$, any vector
\[\rho\in \overline{\left\{\conv\Big(\bigcup_{0\le k\le\ell}\rot_{i_k} \Big) \ \Big|\ R_{i_0} \to R_{i_1}\to \cdots \to R_{i_\ell} \to R_{i_0}\right\}} \]
is realised as the rotation vectors of a point.
This means that there exists a sequence $(\rho_s)_{s\in\N}$ such that $\rho_s\to \rho$ and $\rho_s\in \conv\big(\bigcup_{0\le k\le\ell}\rot_{i_ k} \big)$ with $R_{i_0^s} \to R_{i_1^s}\to \cdots \to R_{i^s_{\ell_s}} \to R_{i_0^s}$ with $i_0^s = i_0$. Up to taking a subsequence we can suppose that $\|\rho-\rho_s\|\le 2^{-s}$. 


By the proof of the first part of the proposition, we know that for any $s$ there is a word $(w_k^s)_{0\le k\le k_s}$ with $s(w_0^s) = e(w_{k_s}^s) = R_{i_0}$ such that $\|\rho((w_k^s)) - \rho_s\| < 2^{-s}$. For any sequence $(p_s)_{s\in\N}$ of integers, $(\omega_k):=(w_k^0)^{p_0}(w_k^1)^{p_1}(w_k^2)^{p_2}\dots$ is a path of $G$. Hence, there exists $\wt x\in R_{i_0}$ such that \eqref{EqDisplaceW} holds for the path $(\omega_k)$. Let us show that if $(p_s)_s$ grows sufficiently fast, then $\rho(\wt x) = \rho$. 

As already noticed, fixing a fundamental domain $D\subset \wt S$ of $S$, for any $k\in \N$ there exists $T'_k\in \G$ such that 
\[\wt f^{\,\sum_{j=0}^{k} n(\omega_j)}(\wt x) \in T(\omega_0)T(\omega_1) \cdots T(\omega_{k}) T'_k\, D,\]
while $\wt x\in T^{\prime -1}_0 D$.
For $k\in\N$, denote 
\[\rho'_k = \frac{\big[T^{\prime -1}_0 T(\omega_0)T(\omega_1) \cdots T(\omega_{k}) T'_k\big]}{\sum_{j=0}^{k} n(\omega_j)};\]
to prove the statement one has to prove that $\rho'_k$ tends to $\rho$.

As the rectangles $R_i$ are compact, the following is finite and independent of the choice of $(p_s)_s$:
\[M_{s_0} = \sup\Big\{\big\|[T'_k]\big\| \ \Big|\  0\le k\le \sum_{s=0}^{s_0} p_s k_s\Big\}.\] 
Denote also 
\[C_s = \sup\Big\{\big\|[T(w_k^s)]\big\| \ \big|\  0\le k\le k_s\Big\}.\] 

Let us build the sequence $(p_s)$ by induction, so that for any $s\in\N$:
\begin{enumerate}[label= \alph*)]
\item\label{ItemARect} for any $\sum_{s=0}^{s_0} p_s k_s \le k' < \sum_{s=0}^{s_0+1} p_s k_s$ we have $\|\rho'_{k'} - \rho\|\le 2^{-s+3}$;
\item\label{ItemBRect} for $k' = \sum_{s=0}^{s_0} p_s k_s$ we have $\|\rho'_{k'}  - \rho\|\le 2^{-s+2}$.
\end{enumerate}
So suppose that the sequence $(p_s)$ is built until $s_0-1\in\N$ and let us choose $p_{s_0}$. It can be easily seen that Condition~\ref{ItemBRect} is satisfied whenever $p_{s_0}$ is large enough: if $p_{s_0}$ is large enough then (for $k' = \sum_{s=0}^{s_0} p_s k_s$) $\rho'_{k'}$ is arbitrarily close to $\rho((\omega_k)_{0\le k\le k'})$ (the constant $M_s$ appearing in the bound of the difference between those two is divided by a number greater than $p_{s_0+1}$, hence this term can be made arbitrarily small), which itself is arbitrarily close to $\rho((w_k^{s_0}))$, which is at distance at most $2^{-s_0}$ of $\rho_{s_0}$, which is at distance at most $2^{-s_0}$ of $\rho$.

Let us prove that if $p_{s_0}$ is large enough, then Condition~\ref{ItemARect} holds for any $p_{s_0+1}\in\N$. Take $\sum_{s=0}^{s_0} p_s k_s \le k' < \sum_{s=0}^{s_0+1} p_s k_s$. One can write $\sum_{s=0}^{s_0} p_s k_s + p'k_{s_0+1} \le k' < \sum_{s=0}^{s_0} p_s k_s + (p'+1)k_{s_0+1}$ ($p'$ counts the number of complete paths $(w_k^{s_0+1})$ already browsed).	Note that for $k''=\sum_{s=0}^{s_0} p_s k_s + p'k_{s_0+1}$, one has $\|\rho((\omega_k)_{0\le k\le k''}) - \rho\|\le 2^{-s_0+1}$: in this case $\rho((\omega_k)_{0\le k\le k''})$ is a convex combination of $\rho((\omega_k)_{0\le k\le \sum_{s=0}^{s_0} p_s k_s})$ and of $\rho((w_k^{s_0+1}))$, both of them being at distance at most $2^{-s_0+1}$ of $\rho$. Using again the bound with the constant $M_s$, we deduce that $\|\rho'_{k''}  - \rho\|\le 2^{-s_0+2}$.
Now, we have 
\begin{align*}
\|\rho'_{k'}  - \rho\| \le & \|\rho'_{k'} - \rho'_{k''} \| + \|\rho'_{k''}  - \rho\|\\
\le & \left\|\frac{\big[T^{\prime -1}_0 T(\omega_0) \cdots T(\omega_{k'}) T'_{k'}\big]}{\sum_{j=0}^{k'} n(\omega_j)} - \frac{\big[T^{\prime -1}_0 T(\omega_0) \cdots T(\omega_{k''}) T'_{k''}\big]}{\sum_{j=0}^{k''} n(\omega_j)}\right\| + 2^{-s_0+2}\\
\le & \frac{\big\|[T'_{k'}]\big\| + \big\|[T'_{k''}]\big\| + \left\|\sum_{i=k''+1}^{k'}\big[T(\omega_i)\big]\right\|}{\sum_{j=0}^{k''} n(\omega_j)} \\
& + \big\|\rho'_{k''}\big\| \frac{\sum_{j=k''+1}^{k'} n(\omega_j)}{\sum_{j=0}^{k'} n(\omega_j)} + 2^{-s_0+2}\\
\le & \frac{2M_{s_0+1}+ k_{s_0}C_{s_0+1}}{p_{s_0}} + \big(\|\rho\|+1\big) \frac{\sum_{j=1}^{k_{s_0}} n(w_j^{s_0})}{p_{s_0}} + 2^{-s_0+2}.
\end{align*}
Choosing $p_{s_0}$ large enough, the latter can be made smaller than $2^{-s_0+3}$.

\bigskip

\noindent\textbf{Proof of Points~\ref{P3PropConnectRectEnsRot}.~and ~\ref{P4PropConnectRectEnsRot}.}
Point~\ref{P3PropConnectRectEnsRot}.~of the proposition is obtained by a straightforward application of the proof of \cite[Theorem A]{zbMATH00009916} (the fact about bounded deviations is not stated in the theorem but written explicitly in \cite[Equation~(9)]{zbMATH00009916}).

Similarly, point~\ref{P4PropConnectRectEnsRot}.~of the proposition is obtained by a straightforward application of the proof of \cite[Theorem~1, (iv)]{llibremackay}.
\end{proof}

\subsection{Creation of heteroclinic connections of horseshoes by forcing theory}

The following is an improvement of \cite[Theorem M]{lct2}:

\begin{theorem}\label{ThConnectionHorse}
Suppose there exist an admissible transverse path $\gamma:[a,b]\to \mathrm{dom}(\F)$ of order $r$, a lift $\wh \gamma$ of $\gamma$ to the universal covering space $\wt{\dom}(\F)$ and a covering automorphism $T$ such that $\wh\gamma$ and $T(\wh\gamma)$ have an $\wh{\mathcal F}$-transverse intersection at $\wh\gamma(t)=T(\wh\gamma)(s)$, where $s<t$.
Then for any $k\ge 1$, there exists a rectangle $\wh R\subset \wt{\dom}(\F)$ that is a rotational horseshoe with deck transformations $T, \dots, T^k$ for $f^{kr}$.

More generally, suppose there exist $a<b<c$ and $\gamma:[a,c]\to \mathrm{dom}(\F)$ a transverse path such that $\gamma|_{[a,b]}$ is admissible of order $r_1$ and $\gamma|_{[b,c]}$ is admissible of order $r_2$. Suppose also that there exist covering automorphisms $T_1,T_2$ such that $\wh\gamma|_{[a,b]}$ and $T_1(\wh\gamma|_{[a,b]})$ have an $\wh{\mathcal F}$-transverse intersection at $\wh\gamma(t_1)=T_1(\wh\gamma)(s_1)$, where $s_1<t_1$, and that $\wh\gamma|_{[b,c]}$ and $T_2(\wh\gamma|_{[b,c]})$ have an $\wh{\mathcal F}$-transverse intersection at $\wh\gamma(t_2)=T_2(\wh\gamma)(s_2)$, where $s_2<t_2$. Choose $k_1,k_2\ge 2$.
Denote $\wh R_1$ the rectangle given by the first part of the theorem for the path $\wh\gamma|_{[a,b]}$ and $k_1$, and $\wh R_2$ the rectangle given by the first part of the theorem for the path $\wh\gamma|_{[b,c]}$ and $k_2$.
Then there exists a deck transformation $U$ such that the intersection $\wh f^{k_1r_1+r_1+r_2}(\wh R_1)\cap U \wh R_2$ is Markovian.
\end{theorem}

The deck transformation $U$ appearing in Theorem~\ref{ThConnectionHorse} is a product $T_1^{\ell_1} T_2^{\ell_2}$: up to taking appropriate translates of $\wh R_1$ by a power of $T_1$ and translates of $\wh R_2$ by a power of $T_2$, one can say that the intersection $\wh f^{k_1+r_1+r_2}(\wh R_1)\cap  \wh R_2$ is Markovian.

\begin{proof}
The proof follows the strategy of \cite[Section 3]{lct2} (see also \cite[Section 9.6]{pa}). The reader should refer to these references for the parts of the proof that are not detailed here.
The beginning of the proof is depicted in Figures~\ref{FigConnectionHorse0} and \ref{FigConnectionHorse1}.

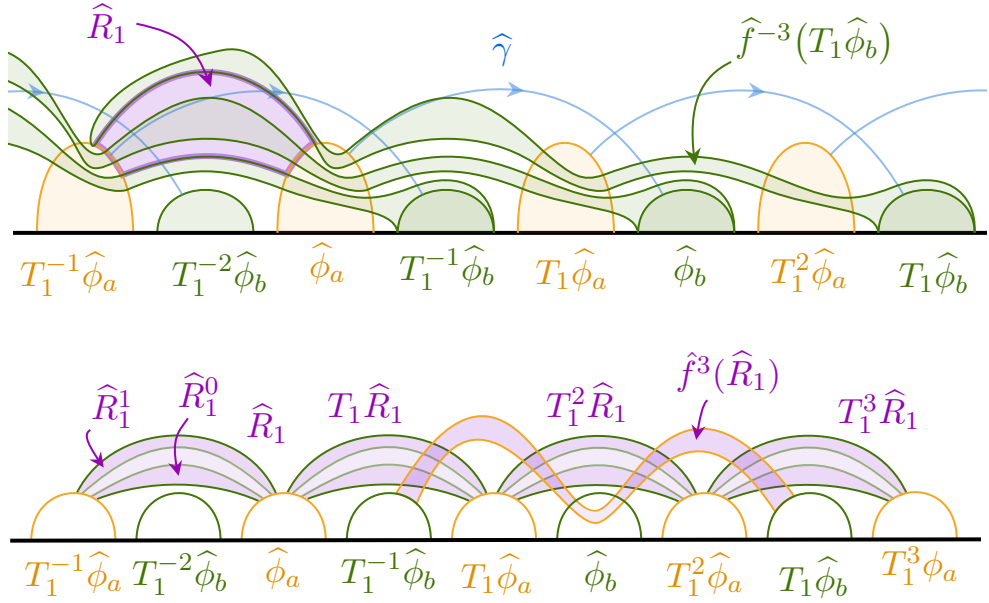
\begin{figure}
\begin{center}

\tikzset{every picture/.style={line width=0.75pt}} 

\begin{tikzpicture}[x=0.75pt,y=0.75pt,yscale=-1.2,xscale=1.2]

\clip (148, 92) rectangle (555, 220);

\draw  [draw opacity=0][fill={rgb, 255:red, 255; green, 255; blue, 255 }  ,fill opacity=0.5 ] (419.6,172.44) .. controls (427.1,178.44) and (431.6,179.53) .. (438.69,169.85) .. controls (440.62,170.44) and (439.85,170.44) .. (441.94,172.28) .. controls (430.27,184.28) and (429.35,185.19) .. (417.6,174.78) .. controls (418.52,173.28) and (418.52,173.11) .. (419.6,172.44) -- cycle ;
\draw  [color={rgb, 255:red, 127; green, 2; blue, 208 }  ,draw opacity=0.5 ][fill={rgb, 255:red, 127; green, 2; blue, 208 }  ,fill opacity=0.15 ][line width=2.25]  (184.3,152.87) .. controls (212.12,110.32) and (252.69,114.8) .. (275.19,152.14) .. controls (271.36,154.42) and (266.6,160.04) .. (263.87,165.44) .. controls (240.64,155.86) and (223.41,154.52) .. (195.55,165.73) .. controls (193.62,161.94) and (188.28,153.67) .. (184.3,152.87) -- cycle ;
\draw  [draw opacity=0][fill={rgb, 255:red, 255; green, 255; blue, 255 }  ,fill opacity=0.25 ] (188.73,155.73) .. controls (226.88,125.42) and (240.66,127.75) .. (270.73,155.87) .. controls (267.33,158.12) and (268.77,157.08) .. (266.58,160.58) .. controls (238.1,144.64) and (219.21,147.97) .. (192.73,160.15) .. controls (190.55,157.53) and (190.55,157.53) .. (188.73,155.73) -- cycle ;
\draw [color={rgb, 255:red, 74; green, 144; blue, 226 }  ,draw opacity=0.6 ]   (289.92,156.76) .. controls (314.07,129.94) and (382.16,106.14) .. (420.13,173.6) ;
\draw [shift={(362.59,130.46)}, rotate = 185.44] [fill={rgb, 255:red, 74; green, 144; blue, 226 }  ,fill opacity=0.6 ][line width=0.08]  [draw opacity=0] (7.14,-3.43) -- (0,0) -- (7.14,3.43) -- (4.74,0) -- cycle    ;
\draw [color={rgb, 255:red, 65; green, 117; blue, 5 }  ,draw opacity=1 ][fill={rgb, 255:red, 65; green, 117; blue, 5 }  ,fill opacity=0.1 ]   (210,190) .. controls (210.35,165.82) and (249.95,166.08) .. (250,190) ;
\draw [color={rgb, 255:red, 245; green, 166; blue, 35 }  ,draw opacity=1 ][fill={rgb, 255:red, 245; green, 166; blue, 35 }  ,fill opacity=0.1 ]   (160,190) .. controls (159.45,139.49) and (199.45,139.89) .. (200,190) ;
\draw [line width=1.5]    (150,190) -- (560,190) ;
\draw [color={rgb, 255:red, 65; green, 117; blue, 5 }  ,draw opacity=1 ][fill={rgb, 255:red, 65; green, 117; blue, 5 }  ,fill opacity=0.1 ]   (310,190) .. controls (310.35,165.82) and (349.95,166.08) .. (350,190) ;
\draw [color={rgb, 255:red, 65; green, 117; blue, 5 }  ,draw opacity=1 ][fill={rgb, 255:red, 65; green, 117; blue, 5 }  ,fill opacity=0.1 ]   (410,190) .. controls (410.35,165.82) and (449.95,166.08) .. (450,190) ;
\draw [color={rgb, 255:red, 245; green, 166; blue, 35 }  ,draw opacity=1 ][fill={rgb, 255:red, 245; green, 166; blue, 35 }  ,fill opacity=0.1 ]   (260,190) .. controls (259.45,139.49) and (299.45,139.89) .. (300,190) ;
\draw [color={rgb, 255:red, 245; green, 166; blue, 35 }  ,draw opacity=1 ][fill={rgb, 255:red, 245; green, 166; blue, 35 }  ,fill opacity=0.1 ]   (360,190) .. controls (359.45,139.49) and (399.45,139.89) .. (400,190) ;
\draw [color={rgb, 255:red, 245; green, 166; blue, 35 }  ,draw opacity=1 ][fill={rgb, 255:red, 245; green, 166; blue, 35 }  ,fill opacity=0.1 ]   (460,190) .. controls (459.45,139.49) and (499.45,139.89) .. (500,190) ;
\draw [color={rgb, 255:red, 65; green, 117; blue, 5 }  ,draw opacity=1 ][fill={rgb, 255:red, 65; green, 117; blue, 5 }  ,fill opacity=0.1 ]   (510,190) .. controls (510.35,165.82) and (549.95,166.08) .. (550,190) ;
\draw [color={rgb, 255:red, 74; green, 144; blue, 226 }  ,draw opacity=0.6 ]   (390.25,157.26) .. controls (414.41,130.44) and (482.49,106.64) .. (520.46,174.1) ;
\draw [shift={(462.92,130.96)}, rotate = 185.44] [fill={rgb, 255:red, 74; green, 144; blue, 226 }  ,fill opacity=0.6 ][line width=0.08]  [draw opacity=0] (7.14,-3.43) -- (0,0) -- (7.14,3.43) -- (4.74,0) -- cycle    ;
\draw [color={rgb, 255:red, 74; green, 144; blue, 226 }  ,draw opacity=0.6 ]   (490.08,157.26) .. controls (514.24,130.44) and (582.32,106.64) .. (620.29,174.1) ;
\draw [shift={(562.75,130.96)}, rotate = 185.44] [fill={rgb, 255:red, 74; green, 144; blue, 226 }  ,fill opacity=0.6 ][line width=0.08]  [draw opacity=0] (7.14,-3.43) -- (0,0) -- (7.14,3.43) -- (4.74,0) -- cycle    ;
\draw [color={rgb, 255:red, 74; green, 144; blue, 226 }  ,draw opacity=0.6 ]   (190.43,157.6) .. controls (214.59,130.77) and (282.67,106.98) .. (320.64,174.43) ;
\draw [shift={(263.1,131.29)}, rotate = 185.44] [fill={rgb, 255:red, 74; green, 144; blue, 226 }  ,fill opacity=0.6 ][line width=0.08]  [draw opacity=0] (7.14,-3.43) -- (0,0) -- (7.14,3.43) -- (4.74,0) -- cycle    ;
\draw [color={rgb, 255:red, 74; green, 144; blue, 226 }  ,draw opacity=0.6 ]   (90.43,157.6) .. controls (114.59,130.77) and (182.67,106.98) .. (220.64,174.43) ;
\draw [shift={(163.1,131.29)}, rotate = 185.44] [fill={rgb, 255:red, 74; green, 144; blue, 226 }  ,fill opacity=0.6 ][line width=0.08]  [draw opacity=0] (7.14,-3.43) -- (0,0) -- (7.14,3.43) -- (4.74,0) -- cycle    ;
\draw [color={rgb, 255:red, 65; green, 117; blue, 5 }  ,draw opacity=1 ][fill={rgb, 255:red, 65; green, 117; blue, 5 }  ,fill opacity=0.1 ]   (310,190) .. controls (308.71,179.04) and (285.71,180.79) .. (261.96,171.54) .. controls (238.21,162.29) and (222.71,161.04) .. (197.21,172.04) .. controls (189.82,173.98) and (178.88,168.52) .. (166.58,160.58) .. controls (146.06,147.29) and (116.81,147.29) .. (92.73,160.15) .. controls (81.28,166.02) and (80.61,162.13) .. (75.19,152.14) .. controls (56,118.14) and (13.33,107.31) .. (-15.7,152.87) .. controls (-19.9,158.57) and (-23.68,125.24) .. (37.69,113.32) .. controls (65.13,108.52) and (74.61,166.57) .. (88.73,155.73) .. controls (126.81,123.29) and (141.89,128.7) .. (170.73,155.87) .. controls (182.48,166.12) and (186.62,168.78) .. (195.55,165.73) .. controls (222.31,155.79) and (238.81,155.29) .. (263.87,165.44) .. controls (288.93,175.59) and (295.96,178.79) .. (316.21,170.79) .. controls (336.46,162.79) and (350.96,172.29) .. (350,190) ;
\draw [color={rgb, 255:red, 65; green, 117; blue, 5 }  ,draw opacity=1 ][fill={rgb, 255:red, 65; green, 117; blue, 5 }  ,fill opacity=0.1 ]   (510,190) .. controls (508.71,179.04) and (485.71,180.79) .. (461.96,171.54) .. controls (438.21,162.29) and (422.71,161.04) .. (397.21,172.04) .. controls (389.82,173.98) and (378.88,168.52) .. (366.58,160.58) .. controls (346.06,147.29) and (316.81,147.29) .. (292.73,160.15) .. controls (281.28,166.02) and (280.61,162.13) .. (275.19,152.14) .. controls (256,118.14) and (213.33,107.31) .. (184.3,152.87) .. controls (180.1,158.57) and (176.32,125.24) .. (237.69,113.32) .. controls (265.13,108.52) and (274.61,166.57) .. (288.73,155.73) .. controls (326.81,123.29) and (341.89,128.7) .. (370.73,155.87) .. controls (382.48,166.12) and (386.62,168.78) .. (395.55,165.73) .. controls (422.31,155.79) and (438.81,155.29) .. (463.87,165.44) .. controls (488.93,175.59) and (495.96,178.79) .. (516.21,170.79) .. controls (536.46,162.79) and (550.96,172.29) .. (550,190) ;
\draw [color={rgb, 255:red, 148; green, 4; blue, 177 }  ,draw opacity=1 ]   (201.33,102.33) .. controls (220.43,106.43) and (225.76,111.71) .. (231.35,126.21) ;
\draw [shift={(232.33,128.83)}, rotate = 250.02] [fill={rgb, 255:red, 148; green, 4; blue, 177 }  ,fill opacity=1 ][line width=0.08]  [draw opacity=0] (7.14,-3.43) -- (0,0) -- (7.14,3.43) -- (4.74,0) -- cycle    ;
\draw [color={rgb, 255:red, 65; green, 117; blue, 5 }  ,draw opacity=1 ]   (448.75,118.33) .. controls (435.95,124.41) and (430.65,136.33) .. (433.37,158.8) ;
\draw [shift={(433.75,161.67)}, rotate = 261.75] [fill={rgb, 255:red, 65; green, 117; blue, 5 }  ,fill opacity=1 ][line width=0.08]  [draw opacity=0] (7.14,-3.43) -- (0,0) -- (7.14,3.43) -- (4.74,0) -- cycle    ;
\draw [color={rgb, 255:red, 65; green, 117; blue, 5 }  ,draw opacity=1 ][fill={rgb, 255:red, 65; green, 117; blue, 5 }  ,fill opacity=0.1 ]   (410,190) .. controls (408.71,179.04) and (385.71,180.79) .. (361.96,171.54) .. controls (338.21,162.29) and (322.71,161.04) .. (297.21,172.04) .. controls (289.82,173.98) and (278.88,168.52) .. (266.58,160.58) .. controls (246.06,147.29) and (216.81,147.29) .. (192.73,160.15) .. controls (181.28,166.02) and (180.61,162.13) .. (175.19,152.14) .. controls (156,118.14) and (113.33,107.31) .. (84.3,152.87) .. controls (80.1,158.57) and (76.32,125.24) .. (137.69,113.32) .. controls (165.13,108.52) and (174.61,166.57) .. (188.73,155.73) .. controls (226.81,123.29) and (241.89,128.7) .. (270.73,155.87) .. controls (282.48,166.12) and (286.62,168.78) .. (295.55,165.73) .. controls (322.31,155.79) and (338.81,155.29) .. (363.87,165.44) .. controls (388.93,175.59) and (395.96,178.79) .. (416.21,170.79) .. controls (436.46,162.79) and (450.96,172.29) .. (450,190) ;

\draw (173.5,196.4) node [anchor=north] [inner sep=0.75pt]  [color={rgb, 255:red, 225; green, 141; blue, 0 }  ,opacity=1 ,xscale=1.2,yscale=1.2]  {$T_{1}^{-1} \wh\phi _{a}$};
\draw (236,196.4) node [anchor=north] [inner sep=0.75pt]  [color={rgb, 255:red, 65; green, 117; blue, 5 }  ,opacity=1 ,xscale=1.2,yscale=1.2]  {$T_{1}^{-2} \wh\phi _{b}$};
\draw (280.83,192.6) node [anchor=north] [inner sep=0.75pt]  [color={rgb, 255:red, 225; green, 141; blue, 0 }  ,opacity=1 ,xscale=1.2,yscale=1.2]  {$\wh\phi _{a}$};
\draw (330.73,193.74) node [anchor=north] [inner sep=0.75pt]  [color={rgb, 255:red, 65; green, 117; blue, 5 }  ,opacity=1 ,xscale=1.2,yscale=1.2]  {$T_{1}^{-1} \wh\phi _{b}$};
\draw (382.43,195.4) node [anchor=north] [inner sep=0.75pt]  [color={rgb, 255:red, 225; green, 141; blue, 0 }  ,opacity=1 ,xscale=1.2,yscale=1.2]  {$T_{1} \wh\phi _{a}$};
\draw (431.76,194.04) node [anchor=north] [inner sep=0.75pt]  [color={rgb, 255:red, 65; green, 117; blue, 5 }  ,opacity=1 ,xscale=1.2,yscale=1.2]  {$\wh\phi _{b}$};
\draw (481.23,195) node [anchor=north] [inner sep=0.75pt]  [color={rgb, 255:red, 225; green, 141; blue, 0 }  ,opacity=1 ,xscale=1.2,yscale=1.2]  {$T_{1}^{2} \wh\phi _{a}$};
\draw (532.96,196.84) node [anchor=north] [inner sep=0.75pt]  [color={rgb, 255:red, 65; green, 117; blue, 5 }  ,opacity=1 ,xscale=1.2,yscale=1.2]  {$T_{1}\wh \phi _{b}$};
\draw (199.33,102.33) node [anchor=east] [inner sep=0.75pt]  [color={rgb, 255:red, 148; green, 4; blue, 177 }  ,opacity=1 ,xscale=1.2,yscale=1.2]  {$\wh R_{1}$};
\draw (347.83,105.21) node [anchor=north west][inner sep=0.75pt]  [color={rgb, 255:red, 8; green, 103; blue, 215 }  ,opacity=1 ,xscale=1.2,yscale=1.2]  {$\wh\gamma $};
\draw (448.88,109.67) node [anchor=west] [inner sep=0.75pt]  [color={rgb, 255:red, 65; green, 117; blue, 5 }  ,opacity=1 ,xscale=1.2,yscale=1.2]  {$\wh{f}^{-3}\big( T_{1}\wh{\phi }_{b}\big)$};

\end{tikzpicture}

\tikzset{every picture/.style={line width=0.75pt}} 

\begin{tikzpicture}[x=0.75pt,y=0.75pt,yscale=-1.05,xscale=1.05]

\draw  [draw opacity=0][fill={rgb, 255:red, 127; green, 2; blue, 208 }  ,fill opacity=0.15 ] (410.58,135.09) .. controls (429.11,98.99) and (488.93,98.99) .. (505.15,135.38) .. controls (501.75,137.63) and (500.93,136.26) .. (496.88,139.11) .. controls (473.65,129.53) and (448.02,126.99) .. (420.15,138.2) .. controls (416.2,136.44) and (414.56,135.9) .. (410.58,135.09) -- cycle ;
\draw  [draw opacity=0][fill={rgb, 255:red, 255; green, 255; blue, 255 }  ,fill opacity=0.5 ] (412.89,135.68) .. controls (442.37,105.44) and (478.82,107.66) .. (503.11,136.08) .. controls (499.71,138.32) and (503.04,134.97) .. (498.99,137.82) .. controls (471.32,118.16) and (449.32,116.04) .. (417.17,137.11) .. controls (413.22,135.35) and (416.87,136.48) .. (412.89,135.68) -- cycle ;
\draw [color={rgb, 255:red, 65; green, 117; blue, 5 }  ,draw opacity=1 ]   (410.58,135.09) .. controls (429.06,98.38) and (489.06,98.74) .. (505.15,135.38) ;
\draw [color={rgb, 255:red, 65; green, 117; blue, 5 }  ,draw opacity=1 ]   (420.15,138.2) .. controls (445.24,128.38) and (475.06,129.29) .. (496.88,139.11) ;
\draw [color={rgb, 255:red, 65; green, 117; blue, 5 }  ,draw opacity=0.5 ]   (412.89,135.68) .. controls (446.42,105.39) and (472.42,106.39) .. (503.11,136.08) ;
\draw [color={rgb, 255:red, 65; green, 117; blue, 5 }  ,draw opacity=0.5 ]   (417.17,137.11) .. controls (451.14,116.68) and (467.14,116.39) .. (498.99,137.82) ;

\draw  [draw opacity=0][fill={rgb, 255:red, 127; green, 2; blue, 208 }  ,fill opacity=0.15 ] (210.82,134.88) .. controls (229.36,98.78) and (289.17,98.78) .. (305.4,135.17) .. controls (302,137.42) and (301.17,136.05) .. (297.13,138.9) .. controls (273.9,129.32) and (248.26,126.78) .. (220.4,137.99) .. controls (216.45,136.23) and (214.81,135.69) .. (210.82,134.88) -- cycle ;
\draw  [draw opacity=0][fill={rgb, 255:red, 255; green, 255; blue, 255 }  ,fill opacity=0.5 ] (213.14,135.47) .. controls (242.62,105.23) and (279.07,107.45) .. (303.36,135.87) .. controls (299.96,138.12) and (303.29,134.77) .. (299.24,137.61) .. controls (271.57,117.95) and (249.57,115.83) .. (217.42,136.9) .. controls (213.47,135.14) and (217.12,136.27) .. (213.14,135.47) -- cycle ;
\draw [color={rgb, 255:red, 65; green, 117; blue, 5 }  ,draw opacity=1 ]   (210.82,134.88) .. controls (229.31,98.17) and (289.31,98.54) .. (305.4,135.17) ;
\draw [color={rgb, 255:red, 65; green, 117; blue, 5 }  ,draw opacity=1 ]   (220.4,137.99) .. controls (245.49,128.17) and (275.31,129.08) .. (297.13,138.9) ;
\draw [color={rgb, 255:red, 65; green, 117; blue, 5 }  ,draw opacity=0.5 ]   (213.14,135.47) .. controls (246.67,105.18) and (272.67,106.18) .. (303.36,135.87) ;
\draw [color={rgb, 255:red, 65; green, 117; blue, 5 }  ,draw opacity=0.5 ]   (217.42,136.9) .. controls (251.39,116.47) and (267.39,116.18) .. (299.24,137.61) ;
\draw  [draw opacity=0][fill={rgb, 255:red, 127; green, 2; blue, 208 }  ,fill opacity=0.15 ] (310.52,135.19) .. controls (329.05,99.09) and (388.87,99.09) .. (405.09,135.48) .. controls (401.69,137.72) and (400.87,136.36) .. (396.82,139.21) .. controls (373.59,129.63) and (347.96,127.09) .. (320.09,138.3) .. controls (316.14,136.54) and (314.5,135.99) .. (310.52,135.19) -- cycle ;
\draw  [draw opacity=0][fill={rgb, 255:red, 255; green, 255; blue, 255 }  ,fill opacity=0.5 ] (312.83,135.78) .. controls (342.31,105.54) and (378.76,107.76) .. (403.05,136.18) .. controls (399.65,138.42) and (402.98,135.07) .. (398.94,137.92) .. controls (371.26,118.26) and (349.26,116.14) .. (317.12,137.2) .. controls (313.16,135.45) and (316.82,136.58) .. (312.83,135.78) -- cycle ;
\draw [color={rgb, 255:red, 65; green, 117; blue, 5 }  ,draw opacity=1 ]   (310.52,135.19) .. controls (329,98.48) and (389,98.84) .. (405.09,135.48) ;
\draw [color={rgb, 255:red, 65; green, 117; blue, 5 }  ,draw opacity=1 ]   (320.09,138.3) .. controls (345.18,128.48) and (375,129.39) .. (396.82,139.21) ;
\draw [color={rgb, 255:red, 65; green, 117; blue, 5 }  ,draw opacity=0.5 ]   (312.83,135.78) .. controls (346.36,105.49) and (372.36,106.49) .. (403.05,136.18) ;
\draw [color={rgb, 255:red, 65; green, 117; blue, 5 }  ,draw opacity=0.5 ]   (317.12,137.2) .. controls (351.08,116.78) and (367.08,116.49) .. (398.94,137.92) ;

\draw  [color={rgb, 255:red, 245; green, 166; blue, 35 }  ,draw opacity=1 ][fill={rgb, 255:red, 127; green, 2; blue, 208 }  ,fill opacity=0.15 ] (262.52,135.32) .. controls (314.81,46.59) and (336.86,145.61) .. (357.57,143.75) .. controls (373.86,140.75) and (396.41,61.39) .. (451.78,136.76) .. controls (449.12,138.01) and (446.3,139.19) .. (442.77,144.25) .. controls (400.01,77.39) and (373.86,146.46) .. (358,149.46) .. controls (340.86,152.18) and (302.81,64.99) .. (272.17,140.4) .. controls (269.95,138.04) and (267.24,136.75) .. (262.52,135.32) -- cycle ;
\draw  [draw opacity=0][fill={rgb, 255:red, 255; green, 255; blue, 255 }  ,fill opacity=0.5 ] (347.67,139.44) .. controls (355.17,145.44) and (359.67,146.53) .. (366.76,136.85) .. controls (368.69,137.44) and (367.92,137.44) .. (370.01,139.28) .. controls (358.34,151.28) and (357.42,152.19) .. (345.67,141.78) .. controls (346.59,140.28) and (346.59,140.11) .. (347.67,139.44) -- cycle ;
\draw  [draw opacity=0][fill={rgb, 255:red, 127; green, 2; blue, 208 }  ,fill opacity=0.15 ] (110.75,134.96) .. controls (129.28,98.85) and (189.1,98.85) .. (205.32,135.25) .. controls (201.92,137.49) and (201.1,136.13) .. (197.05,138.98) .. controls (173.82,129.4) and (148.19,126.85) .. (120.32,138.07) .. controls (116.37,136.31) and (114.73,135.76) .. (110.75,134.96) -- cycle ;
\draw  [draw opacity=0][fill={rgb, 255:red, 255; green, 255; blue, 255 }  ,fill opacity=0.5 ] (113.06,135.55) .. controls (142.54,105.31) and (178.99,107.53) .. (203.28,135.95) .. controls (199.88,138.19) and (203.21,134.84) .. (199.17,137.69) .. controls (171.49,118.03) and (149.49,115.91) .. (117.35,136.97) .. controls (113.39,135.22) and (117.05,136.35) .. (113.06,135.55) -- cycle ;
\draw [color={rgb, 255:red, 65; green, 117; blue, 5 }  ,draw opacity=1 ]   (110.75,134.96) .. controls (129.23,98.25) and (189.23,98.61) .. (205.32,135.25) ;
\draw [color={rgb, 255:red, 65; green, 117; blue, 5 }  ,draw opacity=1 ]   (120.32,138.07) .. controls (145.42,128.25) and (175.23,129.16) .. (197.05,138.98) ;
\draw [color={rgb, 255:red, 65; green, 117; blue, 5 }  ,draw opacity=0.5 ]   (113.06,135.55) .. controls (146.6,105.26) and (172.6,106.26) .. (203.28,135.95) ;
\draw [color={rgb, 255:red, 65; green, 117; blue, 5 }  ,draw opacity=0.5 ]   (117.35,136.97) .. controls (151.31,116.55) and (167.31,116.26) .. (199.17,137.69) ;

\draw [color={rgb, 255:red, 65; green, 117; blue, 5 }  ,draw opacity=1 ]   (239,158) .. controls (239.31,127.84) and (279.09,127.4) .. (279,158) ;
\draw [color={rgb, 255:red, 245; green, 166; blue, 35 }  ,draw opacity=1 ]   (289,158) .. controls (289.31,127.84) and (329.09,127.4) .. (329,158) ;
\draw [color={rgb, 255:red, 65; green, 117; blue, 5 }  ,draw opacity=1 ]   (139,158) .. controls (139.31,127.84) and (179.09,127.4) .. (179,158) ;
\draw [color={rgb, 255:red, 245; green, 166; blue, 35 }  ,draw opacity=1 ]   (89,158) .. controls (89.31,127.84) and (129.09,127.4) .. (129,158) ;
\draw [color={rgb, 255:red, 65; green, 117; blue, 5 }  ,draw opacity=1 ]   (339,158) .. controls (339.31,127.84) and (379.09,127.4) .. (379,158) ;
\draw [color={rgb, 255:red, 245; green, 166; blue, 35 }  ,draw opacity=1 ]   (189,158) .. controls (189.31,127.84) and (229.09,127.4) .. (229,158) ;
\draw [color={rgb, 255:red, 245; green, 166; blue, 35 }  ,draw opacity=1 ]   (389,158) .. controls (389.31,127.84) and (429.09,127.4) .. (429,158) ;
\draw [color={rgb, 255:red, 65; green, 117; blue, 5 }  ,draw opacity=1 ]   (439,158) .. controls (439.31,127.84) and (479.09,127.4) .. (479,158) ;
\draw [color={rgb, 255:red, 148; green, 4; blue, 177 }  ,draw opacity=1 ]   (410.44,90.75) .. controls (408.15,92.18) and (406.6,100.52) .. (405.57,106.85) ;
\draw [shift={(405.11,109.75)}, rotate = 278.97] [fill={rgb, 255:red, 148; green, 4; blue, 177 }  ,fill opacity=1 ][line width=0.08]  [draw opacity=0] (7.14,-3.43) -- (0,0) -- (7.14,3.43) -- (4.74,0) -- cycle    ;
\draw [color={rgb, 255:red, 245; green, 166; blue, 35 }  ,draw opacity=1 ]   (489,157.56) .. controls (489.31,127.4) and (529.09,126.95) .. (529,157.56) ;
\draw [line width=1.5]    (79,158) -- (540.51,157.38) ;
\draw [color={rgb, 255:red, 148; green, 4; blue, 177 }  ,draw opacity=1 ]   (163.49,99.39) .. controls (159.33,105.43) and (155.17,108.62) .. (158.44,124.15) ;
\draw [shift={(159.09,126.99)}, rotate = 256.26] [fill={rgb, 255:red, 148; green, 4; blue, 177 }  ,fill opacity=1 ][line width=0.08]  [draw opacity=0] (7.14,-3.43) -- (0,0) -- (7.14,3.43) -- (4.74,0) -- cycle    ;
\draw [color={rgb, 255:red, 148; green, 4; blue, 177 }  ,draw opacity=1 ]   (117.09,104.99) .. controls (114.58,106.55) and (113.89,109.09) .. (121.85,120.66) ;
\draw [shift={(123.49,122.99)}, rotate = 234.48] [fill={rgb, 255:red, 148; green, 4; blue, 177 }  ,fill opacity=1 ][line width=0.08]  [draw opacity=0] (7.14,-3.43) -- (0,0) -- (7.14,3.43) -- (4.74,0) -- cycle    ;

\draw (109.3,161.6) node [anchor=north] [inner sep=0.75pt]  [color={rgb, 255:red, 225; green, 141; blue, 0 }  ,opacity=1 ,xscale=1.2,yscale=1.2]  {$T_{1}^{-1} \wh\phi _{a}$};
\draw (159.82,161.04) node [anchor=north] [inner sep=0.75pt]  [color={rgb, 255:red, 65; green, 117; blue, 5 }  ,opacity=1 ,xscale=1.2,yscale=1.2]  {$T_{1}^{-2} \wh\phi _{b}$};
\draw (208.9,159.6) node [anchor=north] [inner sep=0.75pt]  [color={rgb, 255:red, 225; green, 141; blue, 0 }  ,opacity=1 ,xscale=1.2,yscale=1.2]  {$\wh\phi _{a}$};
\draw (259.02,160.24) node [anchor=north] [inner sep=0.75pt]  [color={rgb, 255:red, 65; green, 117; blue, 5 }  ,opacity=1 ,xscale=1.2,yscale=1.2]  {$T_{1}^{-1} \wh\phi _{b}$};
\draw (310.5,162.4) node [anchor=north] [inner sep=0.75pt]  [color={rgb, 255:red, 225; green, 141; blue, 0 }  ,opacity=1 ,xscale=1.2,yscale=1.2]  {$T_{1} \wh\phi _{a}$};
\draw (359.82,161.04) node [anchor=north] [inner sep=0.75pt]  [color={rgb, 255:red, 65; green, 117; blue, 5 }  ,opacity=1 ,xscale=1.2,yscale=1.2]  {$\wh\phi _{b}$};
\draw (409.3,162) node [anchor=north] [inner sep=0.75pt]  [color={rgb, 255:red, 225; green, 141; blue, 0 }  ,opacity=1 ,xscale=1.2,yscale=1.2]  {$T_{1}^{2} \wh\phi _{a}$};
\draw (461.02,163.84) node [anchor=north] [inner sep=0.75pt]  [color={rgb, 255:red, 65; green, 117; blue, 5 }  ,opacity=1 ,xscale=1.2,yscale=1.2]  {$T_{1} \wh\phi _{b}$};
\draw (189.4,112.78) node [anchor=south west] [inner sep=0.75pt]  [color={rgb, 255:red, 148; green, 4; blue, 177 }  ,opacity=1 ,xscale=1.2,yscale=1.2]  {$\wh R_{1}$};
\draw (248.25,103.13) node [anchor=south] [inner sep=0.75pt]  [color={rgb, 255:red, 148; green, 4; blue, 177 }  ,opacity=1 ,xscale=1.2,yscale=1.2]  {$T_{1} \wh R_{1}$};
\draw (353.34,104.98) node [anchor=south] [inner sep=0.75pt]  [color={rgb, 255:red, 148; green, 4; blue, 177 }  ,opacity=1 ,xscale=1.2,yscale=1.2]  {$T_{1}^{2} \wh R_{1}$};
\draw (420.54,89.41) node [anchor=south] [inner sep=0.75pt]  [color={rgb, 255:red, 148; green, 4; blue, 177 }  ,opacity=1 ,xscale=1.2,yscale=1.2]  {$\hat{f}^{3}( \wh R_{1})$};
\draw (470.79,107.98) node [anchor=south west] [inner sep=0.75pt]  [color={rgb, 255:red, 148; green, 4; blue, 177 }  ,opacity=1 ,xscale=1.2,yscale=1.2]  {$T_{1}^{3} \wh R_{1}$};
\draw (511.63,159) node [anchor=north] [inner sep=0.75pt]  [color={rgb, 255:red, 225; green, 141; blue, 0 }  ,opacity=1 ,xscale=1.2,yscale=1.2]  {$T_{1}^{3} \phi _{a}$};
\draw (115.8,105.98) node [anchor=south west] [inner sep=0.75pt]  [color={rgb, 255:red, 148; green, 4; blue, 177 }  ,opacity=1 ,xscale=1.2,yscale=1.2]  {$\wh R_{1}^{1}$};
\draw (157,99.98) node [anchor=south west] [inner sep=0.75pt]  [color={rgb, 255:red, 148; green, 4; blue, 177 }  ,opacity=1 ,xscale=1.2,yscale=1.2]  {$\wh R_{1}^{0}$};

\end{tikzpicture}

\caption{Beginning of the proof of Theorem~\ref{ThConnectionHorse} for $k_1=3$: construction of the rectangle $\wh R_1$ (top) and Markovian intersections of the image $\wh f^3(\wh R_1)$ with $T_1\wh R_1$, $T_1^2\wh R_1$ and $T_1^3\wh R_1$ (bottom).}\label{FigConnectionHorse0}
\end{center}
\end{figure}

Fix $k_1, k_2\ge 2$. Denote $\wh \phi_a = \wh\phi_{\wh\gamma(a)}$, $\wh \phi_b = \wh\phi_{\wh\gamma(b)}$ and $\wh \phi_c = \wh\phi_{\wh\gamma(c)}$. By hypothesis, the paths $\wh\gamma|_{[a,b]}$ and $T_1^{-1}\wh\gamma|_{[a,b]}$ intersect $\wh\F$-transversally. By successive applications of Proposition~\ref{propFondalct1}, this implies that for any $-1\le j\le k_1-1$, we have $\wh f^{k_1}(\wh\phi_a) \cap T_1^j \wh\phi_b \neq\emptyset$ (see \cite[Lemma 9]{lct2} for details). 

As in \cite[Section 3.1]{lct2}, we define $R_a = \bigcap_{k\in \Z} R(T_1^k\wh\phi_a)$ and, for $p\in\Z$, the set $\X_{p}$ of paths joining $T_1^{-1}\wh\phi_a$ to $\wh\phi_a$ whose interior is a connected component of $T_1^p \wh f^{-k_1r_1}(\wh\phi_b) \cap R_a$. The following is \cite[Lemma 10]{lct2}:

\begin{lemma}\label{Lemma10}
Every simple path $\delta : [c,d]\to\wt\dom(\F)$ that joins $T_1^{-p_0}\wh \phi_a$ to $T_1^{p_1} \wh \phi_a$, with $p_0,p_1>0$, and which is $T_1$-free, meets $L(\wh \phi_a)$.
%
\end{lemma}

The same statement holds with $\wh\phi_b$ instead of $\wh\phi_a$.

\begin{figure}
\begin{center}

\tikzset{every picture/.style={line width=0.75pt}} 

\begin{tikzpicture}[x=0.75pt,y=0.75pt,yscale=-.9,xscale=.9]

\draw [color={rgb, 255:red, 65; green, 117; blue, 5 }  ,draw opacity=0.5 ][fill={rgb, 255:red, 65; green, 117; blue, 5 }  ,fill opacity=0.1 ]   (280,30) .. controls (274.42,88.15) and (350.57,26.55) .. (385.87,47.7) .. controls (421.16,68.85) and (453.71,66.32) .. (475.41,47.85) .. controls (497.1,29.39) and (547.93,50.22) .. (550,66.04) .. controls (552.07,81.86) and (531.53,143.87) .. (545.25,143.04) .. controls (558.97,142.22) and (565.08,83.88) .. (554.75,59.54) .. controls (544.42,35.21) and (495.83,23.96) .. (472.17,41.7) .. controls (448.52,59.44) and (415.36,59.35) .. (388.02,42.01) .. controls (360.68,24.67) and (325.35,53.93) .. (320,30) ;
\draw  [draw opacity=0][fill={rgb, 255:red, 208; green, 2; blue, 27 }  ,fill opacity=0.15 ] (385.87,47.7) .. controls (420.35,69.72) and (454.51,65.35) .. (475.41,47.85) .. controls (477.51,52.45) and (477.53,51.99) .. (481.28,54.76) .. controls (444.88,77.56) and (409.94,74.8) .. (380.81,53.55) .. controls (383.81,50.75) and (382.91,52.7) .. (385.87,47.7) -- cycle ;
\draw  [draw opacity=0][fill={rgb, 255:red, 127; green, 2; blue, 208 }  ,fill opacity=0.15 ] (373.67,201.48) .. controls (404.6,177.61) and (456.24,180.78) .. (483.3,203.05) .. controls (479.91,205.3) and (478.55,205.93) .. (476.4,209.56) .. controls (453.44,196.12) and (408.92,194) .. (382.12,207.2) .. controls (379.21,204.02) and (377.12,203.3) .. (373.67,201.48) -- cycle ;
\draw [color={rgb, 255:red, 245; green, 166; blue, 35 }  ,draw opacity=1 ]   (110,130) -- (600,130) ;
\draw [color={rgb, 255:red, 65; green, 117; blue, 5 }  ,draw opacity=1 ]   (110.4,130.88) -- (600.4,130.88) ;
\draw [color={rgb, 255:red, 65; green, 117; blue, 5 }  ,draw opacity=1 ]   (410,30) .. controls (410.18,59.85) and (449.58,59.65) .. (450,30) ;
\draw [color={rgb, 255:red, 245; green, 166; blue, 35 }  ,draw opacity=1 ]   (340,30) .. controls (341.08,69.79) and (389.58,69.79) .. (390,30) ;
\draw [color={rgb, 255:red, 245; green, 166; blue, 35 }  ,draw opacity=1 ]   (470,30) .. controls (471.08,69.79) and (519.58,69.79) .. (520,30) ;
\draw [color={rgb, 255:red, 65; green, 117; blue, 5 }  ,draw opacity=1 ]   (550,30) .. controls (550.18,59.85) and (589.58,59.65) .. (590,30) ;
\draw [color={rgb, 255:red, 65; green, 117; blue, 5 }  ,draw opacity=1 ]   (280,30) .. controls (280.18,59.85) and (319.58,59.65) .. (320,30) ;
\draw [color={rgb, 255:red, 245; green, 166; blue, 35 }  ,draw opacity=1 ]   (120,30) .. controls (121.08,69.79) and (169.58,69.79) .. (170,30) ;
\draw [color={rgb, 255:red, 65; green, 117; blue, 5 }  ,draw opacity=1 ]   (210.25,30.5) .. controls (210.43,60.35) and (249.83,60.15) .. (250.25,30.5) ;
\draw [color={rgb, 255:red, 245; green, 166; blue, 35 }  ,draw opacity=1 ]   (410,230) .. controls (410.31,199.84) and (450.09,199.4) .. (450,230) ;
\draw [color={rgb, 255:red, 65; green, 117; blue, 5 }  ,draw opacity=1 ]   (470,230) .. controls (470.31,190.06) and (519.86,189.84) .. (520,230) ;
\draw [color={rgb, 255:red, 245; green, 166; blue, 35 }  ,draw opacity=1 ]   (550,230) .. controls (550.31,199.84) and (590.09,199.4) .. (590,230) ;
\draw [color={rgb, 255:red, 65; green, 117; blue, 5 }  ,draw opacity=1 ]   (340,230) .. controls (340.31,190.06) and (389.86,189.84) .. (390,230) ;
\draw [color={rgb, 255:red, 245; green, 166; blue, 35 }  ,draw opacity=1 ]   (280,230) .. controls (280.31,199.84) and (320.09,199.4) .. (320,230) ;
\draw [color={rgb, 255:red, 245; green, 166; blue, 35 }  ,draw opacity=1 ]   (210,230) .. controls (210.31,199.84) and (250.09,199.4) .. (250,230) ;
\draw [color={rgb, 255:red, 65; green, 117; blue, 5 }  ,draw opacity=1 ]   (120,230) .. controls (120.31,190.06) and (169.86,189.84) .. (170,230) ;
\draw [line width=1.5]    (110,230) -- (600,230) ;
\draw [line width=1.5]    (110,30) -- (600,30) ;
\draw [color={rgb, 255:red, 65; green, 117; blue, 5 }  ,draw opacity=1 ]   (385.87,47.7) .. controls (420.48,69.24) and (453.97,66.07) .. (475.41,47.85) ;
\draw [color={rgb, 255:red, 245; green, 166; blue, 35 }  ,draw opacity=1 ]   (373.67,201.48) .. controls (404.71,177.09) and (456.58,181.34) .. (483.3,203.05) ;
\draw [color={rgb, 255:red, 245; green, 166; blue, 35 }  ,draw opacity=1 ]   (382.12,207.2) .. controls (409.46,193.59) and (454.46,196.09) .. (476.4,209.56) ;
\draw [color={rgb, 255:red, 245; green, 166; blue, 35 }  ,draw opacity=0.5 ][fill={rgb, 255:red, 245; green, 166; blue, 35 }  ,fill opacity=0.1 ]   (210,230) .. controls (203.71,192.83) and (262.17,193) .. (296.42,192) .. controls (330.67,191) and (333.15,208.75) .. (362.3,197.89) .. controls (391.44,187.03) and (405.49,172.04) .. (445.71,176.27) .. controls (485.94,180.49) and (511.71,226.28) .. (483.3,203.05) .. controls (454.9,179.82) and (401.79,178.71) .. (373.67,201.48) .. controls (345.55,224.24) and (326.58,195.89) .. (297.34,197.82) .. controls (268.1,199.75) and (258.87,196.46) .. (250,230) ;
\draw [color={rgb, 255:red, 245; green, 166; blue, 35 }  ,draw opacity=0.5 ][fill={rgb, 255:red, 245; green, 166; blue, 35 }  ,fill opacity=0.1 ]   (280,230) .. controls (270.42,198.63) and (303.64,200.08) .. (320.92,209.13) .. controls (338.19,218.17) and (356.46,220.67) .. (382.12,207.2) .. controls (407.78,193.74) and (452.66,195.74) .. (476.4,209.56) .. controls (500.15,223.38) and (508.09,202.84) .. (523.34,183.82) .. controls (538.59,164.8) and (546.9,125.38) .. (557.12,124.27) .. controls (567.34,123.16) and (519.9,216.53) .. (503.79,220.04) .. controls (487.67,223.56) and (457.67,206.45) .. (432.57,204.9) .. controls (407.46,203.34) and (374.05,218.67) .. (360.57,220.67) .. controls (347.09,222.67) and (321.68,219.79) .. (320,230) ;
\draw [color={rgb, 255:red, 65; green, 117; blue, 5 }  ,draw opacity=0.5 ][fill={rgb, 255:red, 65; green, 117; blue, 5 }  ,fill opacity=0.1 ]   (210.25,30.5) .. controls (202.58,85.49) and (259.92,81.61) .. (300.08,72.65) .. controls (340.25,63.69) and (351.53,43.99) .. (386.92,67.65) .. controls (422.31,91.32) and (452.92,82.65) .. (476.92,65.82) .. controls (500.92,48.99) and (514.97,32.04) .. (481.28,54.76) .. controls (447.58,77.49) and (409.58,74.49) .. (380.81,53.55) .. controls (352.05,32.61) and (329.25,63.69) .. (299.58,63.69) .. controls (269.92,63.69) and (255.6,54.43) .. (250.25,30.5) ;
\draw [color={rgb, 255:red, 65; green, 117; blue, 5 }  ,draw opacity=1 ]   (380.81,53.55) .. controls (410.08,74.49) and (445.42,77.99) .. (481.28,54.76) ;
\draw [color={rgb, 255:red, 74; green, 144; blue, 226 }  ,draw opacity=1 ]   (430.39,207.22) .. controls (429.72,181.22) and (429.88,140.8) .. (430.68,52.8) ;
\draw [shift={(430.09,127.36)}, rotate = 90.31] [fill={rgb, 255:red, 74; green, 144; blue, 226 }  ,fill opacity=1 ][line width=0.08]  [draw opacity=0] (7.14,-3.43) -- (0,0) -- (7.14,3.43) -- (4.74,0) -- cycle    ;
\draw [color={rgb, 255:red, 74; green, 144; blue, 226 }  ,draw opacity=1 ]   (570.34,206.64) .. controls (470.87,120.94) and (271.61,120.11) .. (147.42,200.12) ;
\draw [shift={(356.46,141.25)}, rotate = 359.61] [fill={rgb, 255:red, 74; green, 144; blue, 226 }  ,fill opacity=1 ][line width=0.08]  [draw opacity=0] (7.14,-3.43) -- (0,0) -- (7.14,3.43) -- (4.74,0) -- cycle    ;
\draw [color={rgb, 255:red, 74; green, 144; blue, 226 }  ,draw opacity=1 ]   (301.57,206.92) .. controls (324.5,155.29) and (464,135.79) .. (490.1,200.08) ;
\draw [shift={(398.34,159.77)}, rotate = 175.82] [fill={rgb, 255:red, 74; green, 144; blue, 226 }  ,fill opacity=1 ][line width=0.08]  [draw opacity=0] (7.14,-3.43) -- (0,0) -- (7.14,3.43) -- (4.74,0) -- cycle    ;
\draw [color={rgb, 255:red, 74; green, 144; blue, 226 }  ,draw opacity=1 ]   (229.93,207.1) .. controls (241.84,175.4) and (317.8,158.52) .. (360.56,200.81) ;
\draw [shift={(296.32,175.55)}, rotate = 178.72] [fill={rgb, 255:red, 74; green, 144; blue, 226 }  ,fill opacity=1 ][line width=0.08]  [draw opacity=0] (7.14,-3.43) -- (0,0) -- (7.14,3.43) -- (4.74,0) -- cycle    ;
\draw [color={rgb, 255:red, 74; green, 144; blue, 226 }  ,draw opacity=1 ]   (142.75,60.25) .. controls (259.53,141.86) and (452.23,136.61) .. (570.63,52.21) ;
\draw [shift={(360.74,118.36)}, rotate = 178.47] [fill={rgb, 255:red, 74; green, 144; blue, 226 }  ,fill opacity=1 ][line width=0.08]  [draw opacity=0] (7.14,-3.43) -- (0,0) -- (7.14,3.43) -- (4.74,0) -- cycle    ;
\draw [color={rgb, 255:red, 74; green, 144; blue, 226 }  ,draw opacity=1 ]   (495.25,59.75) .. controls (475.96,96.82) and (327.5,96.29) .. (300.92,52.25) ;
\draw [shift={(394.39,86.19)}, rotate = 2.15] [fill={rgb, 255:red, 74; green, 144; blue, 226 }  ,fill opacity=1 ][line width=0.08]  [draw opacity=0] (7.14,-3.43) -- (0,0) -- (7.14,3.43) -- (4.74,0) -- cycle    ;
\draw [color={rgb, 255:red, 74; green, 144; blue, 226 }  ,draw opacity=1 ]   (367.19,60) .. controls (351.99,77.6) and (295.55,98.3) .. (230.75,53.5) ;
\draw [shift={(295.95,79.62)}, rotate = 5.45] [fill={rgb, 255:red, 74; green, 144; blue, 226 }  ,fill opacity=1 ][line width=0.08]  [draw opacity=0] (7.14,-3.43) -- (0,0) -- (7.14,3.43) -- (4.74,0) -- cycle    ;
\draw [color={rgb, 255:red, 208; green, 87; blue, 102 }  ,draw opacity=1 ]   (395,102.78) .. controls (416.44,102.14) and (413.61,88.6) .. (414.24,66.28) ;
\draw [shift={(414.33,63.44)}, rotate = 92.39] [fill={rgb, 255:red, 208; green, 87; blue, 102 }  ,fill opacity=1 ][line width=0.08]  [draw opacity=0] (7.14,-3.43) -- (0,0) -- (7.14,3.43) -- (4.74,0) -- cycle    ;
\draw [color={rgb, 255:red, 141; green, 43; blue, 204 }  ,draw opacity=1 ]   (415.33,263.11) .. controls (396.58,262.46) and (391.31,228.89) .. (394.48,199.3) ;
\draw [shift={(394.8,196.57)}, rotate = 97.17] [fill={rgb, 255:red, 141; green, 43; blue, 204 }  ,fill opacity=1 ][line width=0.08]  [draw opacity=0] (7.14,-3.43) -- (0,0) -- (7.14,3.43) -- (4.74,0) -- cycle    ;

\draw (132.46,122.48) node [anchor=south] [inner sep=0.75pt]  [color={rgb, 255:red, 225; green, 141; blue, 0 }  ,opacity=1 ,xscale=1.2,yscale=1.2]  {$\wh \phi _{b}$};
\draw (130.12,133) node [anchor=north] [inner sep=0.75pt]  [color={rgb, 255:red, 65; green, 117; blue, 5 }  ,opacity=1 ,xscale=1.2,yscale=1.2]  {$\wh \phi _{b}$};
\draw (570.62,233.4) node [anchor=north] [inner sep=0.75pt]  [color={rgb, 255:red, 225; green, 141; blue, 0 }  ,opacity=1 ,xscale=1.2,yscale=1.2]  {$T_{1} \wh \phi _{a}$};
\draw (430.12,231) node [anchor=north] [inner sep=0.75pt]  [color={rgb, 255:red, 225; green, 141; blue, 0 }  ,opacity=1 ,xscale=1.2,yscale=1.2]  {$\wh \phi _{a}$};
\draw (300.62,233.9) node [anchor=north] [inner sep=0.75pt]  [color={rgb, 255:red, 225; green, 141; blue, 0 }  ,opacity=1 ,xscale=1.2,yscale=1.2]  {$T_{1}^{-1} \wh \phi _{a}$};
\draw (231.12,231.4) node [anchor=north] [inner sep=0.75pt]  [color={rgb, 255:red, 225; green, 141; blue, 0 }  ,opacity=1 ,xscale=1.2,yscale=1.2]  {$T_{1}^{-2} \wh \phi _{a}$};
\draw (496.12,230.4) node [anchor=north] [inner sep=0.75pt]  [color={rgb, 255:red, 65; green, 117; blue, 5 }  ,opacity=1 ,xscale=1.2,yscale=1.2]  {$T_{1}^{-1} \wh\phi _{b}$};
\draw (367.62,232.9) node [anchor=north] [inner sep=0.75pt]  [color={rgb, 255:red, 65; green, 117; blue, 5 }  ,opacity=1 ,xscale=1.2,yscale=1.2]  {$T_{1}^{-2} \wh \phi _{b}$};
\draw (143.62,231.4) node [anchor=north] [inner sep=0.75pt]  [color={rgb, 255:red, 65; green, 117; blue, 5 }  ,opacity=1 ,xscale=1.2,yscale=1.2]  {$T_{1} \wh \phi _{b}$};
\draw (431.12,23.48) node [anchor=south] [inner sep=0.75pt]  [color={rgb, 255:red, 65; green, 117; blue, 5 }  ,opacity=1 ,xscale=1.2,yscale=1.2]  {$\wh \phi _{c}$};
\draw (572.12,26.38) node [anchor=south] [inner sep=0.75pt]  [color={rgb, 255:red, 65; green, 117; blue, 5 }  ,opacity=1 ,xscale=1.2,yscale=1.2]  {$T_{2}^{-1} \wh \phi _{c}$};
\draw (303.62,23.98) node [anchor=south] [inner sep=0.75pt]  [color={rgb, 255:red, 65; green, 117; blue, 5 }  ,opacity=1 ,xscale=1.2,yscale=1.2]  {$T_{2} \wh \phi _{c}$};
\draw (232.12,23.48) node [anchor=south] [inner sep=0.75pt]  [color={rgb, 255:red, 65; green, 117; blue, 5 }  ,opacity=1 ,xscale=1.2,yscale=1.2]  {$T_{2}^{2} \wh \phi _{c}$};
\draw (497.62,24.98) node [anchor=south] [inner sep=0.75pt]  [color={rgb, 255:red, 225; green, 141; blue, 0 }  ,opacity=1 ,xscale=1.2,yscale=1.2]  {$T_{2} \wh \phi _{b}$};
\draw (364.62,23.48) node [anchor=south] [inner sep=0.75pt]  [color={rgb, 255:red, 225; green, 141; blue, 0 }  ,opacity=1 ,xscale=1.2,yscale=1.2]  {$T_{2}^{2}\wh \phi _{b}$};
\draw (142.62,25.98) node [anchor=south] [inner sep=0.75pt]  [color={rgb, 255:red, 225; green, 141; blue, 0 }  ,opacity=1 ,xscale=1.2,yscale=1.2]  {$T_{2}^{-1} \wh \phi _{b}$};
\draw (550.4,161.59) node [anchor=west] [inner sep=0.75pt]  [color={rgb, 255:red, 203; green, 140; blue, 35 }  ,opacity=1 ,xscale=1.2,yscale=1.2]  {$\wh {f}^{k_{1} r_{1}}\big( T_{1}^{-1}\wh {\phi }_{a}\big)$};
\draw (564.8,64.6) node [anchor=north west][inner sep=0.75pt]  [color={rgb, 255:red, 65; green, 117; blue, 5 }  ,opacity=1 ,xscale=1.2,yscale=1.2]  {$\wh {f}^{-k_{2} r_{2}}\big( T_{2}\wh {\phi }_{c}\big)$};
\draw (348,100) node [anchor=west] [inner sep=0.75pt]  [font=\small,color={rgb, 255:red, 208; green, 87; blue, 102 }  ,opacity=1 ,xscale=1.2,yscale=1.2]  {$T_{2}^{2}\wh {R}_{2}$};
\draw (413.61,275) node [anchor=west] [inner sep=0.75pt]  [font=\small,color={rgb, 255:red, 141; green, 43; blue, 204 }  ,opacity=1 ,xscale=1.2,yscale=1.2]  {$\wh {f}^{k_{1} r_{1}}\big( T_{1}^{-i-1}\wh {R}_{1}^{i}\big)$};

\end{tikzpicture}

\caption{Proof of Theorem~\ref{ThConnectionHorse} in the case $k_1 = k_2 = 2$: construction of the rectangles $\wh R_1$ and $\wh R_2$.}\label{FigConnectionHorse1}
\end{center}
\end{figure}
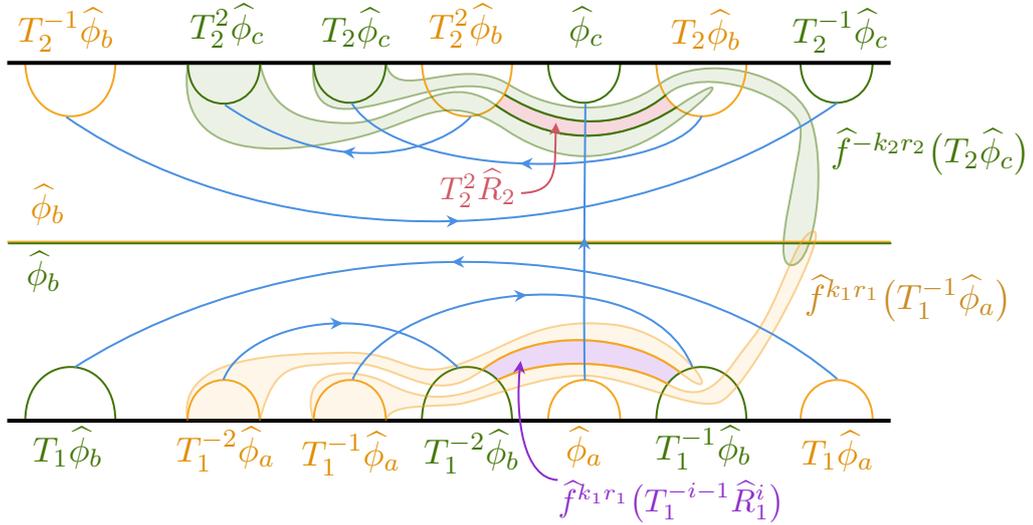

From this lemma one can deduce the following \cite[Lemma 11]{lct2}:

\begin{lemma}\label{Lemma11}
We have the following:
\begin{enumerate}
\item For any $-1\le p\le k_1-2$, we have $\X_p\neq\emptyset$;
\item For $-1\le p_0<p_2<p_1\le k_1-2$, for every $\delta_0\in \X_{p_0}$ and $\delta_1\in \X_{p_1}$, there exist at least two paths in $\X_{p_2}$ between $\delta_0$ and $\delta_1$. 
\end{enumerate}
\end{lemma}

This allows to pick, for any $-1< p\le k_1-2$, one path $\delta_p\in\X_p$, and for $-1\le p< k_1-2$, one path $\delta'_p\in\X_p$, such that the family $\delta'_{-1}, \delta_0, \delta'_0,\dots, \delta'_{k_1-3}, \delta_{k_1-2}$ is well ordered\footnote{This orientation is given by \cite[Lemma~9.29]{pa}, but we will not need this fact.}. 
Applying Lemma~\ref{Lemma11} again, one can moreover suppose that for $0\le i\le k_1-2$, there is no element of $\bigcup_p \X_p$ between $\delta'_{i-1}$ and $\delta_{i}$.

This allows to define $\wh R_1$ as the set delimited by $\delta'_{-1}$, $\delta_{k_1-2}$ and the pieces of $\wh\phi_a$ and $T_1^{-1}\wh\phi_a$ lying between $\delta'_{-1}$ and $\delta_{k_1-2}$ (using Schoenflies theorem). By convention, $\delta'_{-1}$ and $\delta_{k_1-2}$ are supposed to be the horizontal sides of $\wh R_1$. 
This rectangle $\wh R_1$ has $k_1-1$ horizontal subrectangles $\wh R_1^i$ (for $0\le i\le k_1-2$) delimited by the paths $\delta'_{i-1}$ and $\delta_{i}$ (note that if $k_1=2$, then this subrectangle is equal to $\wh R_1$). 
By Lemma~\ref{Lemma10} applied to $\wh\phi_b$ and the sets $\wh f^k(\wh\phi_a)$, and the hypothesis made on the $\delta_i$ and $\delta'_i$, the interior of the subrectangles $\wh R_1^i$ do not intersect elements of $T_1^j \wh f^{-k_1r_1}(\wh\phi_b)$ (for any $j\in\Z$).

\begin{lemma}
For any $0\le i\le k_1-2$, the rectangle $\wh f^{k_1r_1}(\wh R_1^i)$ has a pre-Markovian intersection (in the sense of Definition~\ref{def:markov}) with both $T_1^{i+1} \wh R_1$ and $T_1^{i+2} \wh R_1$.
\end{lemma}

This lemma is depicted in the bottom of Figure~\ref{FigConnectionHorse0}. 

\begin{proof}
We explain the proof for the intersection with $T_1^{i+1} \wh R_1$, the other one being identical.

Because of the property stated before the lemma, one can apply Homma's Theorem (Theorem~\ref{PropHomma}) to the rectangle $T_1^{i+1} \wh R_1$ and the leaves $T_1^{i}\wh\phi_a$ and $T_1^{i+1}\wh\phi_a$ to get a homeomorphism $h : \wt\dom(\F)\to\R^2$ sending the horizontal sides of $T_1^{i+1} \wh R_1$ on $\{0\}\times [0,1]$ and $\{1\}\times [0,1]$, the vertical sides of $T_1^{i+1} \wh R_1$ on $[0,1]\times\{0\}$ and $[0,1]\times\{1\}$ and the leaves $T_1^{i}\wh\phi_a$ and $T_1^{i+1}\wh\phi_a$ to respectively $\big((-\infty,0]\times\{1\}\big)\cup \big(\{0\}\times[0,1]\big) \cup\big((-\infty,0]\times\{0\}\big)$ and $\big([1,+\infty)\times\{1\}\big) \cup \big(\{1\}\times[0,1]\big) \cup \big([1,+\infty)\times\{1\}\big)$.

The horizontal sides of the rectangle $\wh f^{k_1}(\wh R_1^i)$ are made of pieces of $T_1^{i-1}\wh\phi_b$ and $T_1^i\wh\phi_b$, hence are disjoint from the horizontal sides of the rectangle $T_1^{i+1} \wh R_1$, that are pieces of some $\X_j$ (because $\wh\phi_b$ is a Brouwer line). They are also disjoint from the vertical sides of the rectangle $T_1^{i+1} \wh R_1$, that are pieces of some $T_1^j \wh\phi_a$ (because of the transverse intersections, we have that $T_1^j\wh\phi_a\cap \wh\phi_b = \emptyset$ for any $j\in\Z$).

The vertical sides of the rectangle $\wh f^{k_1}(\wh R_1^i)$ are made of pieces of $\wh f^{k_1}(T_1^{i}\wh\phi_a)$ and $\wh f^{k_1}(T_1^{i+1}\wh\phi_a)$. Hence, they are disjoint from the vertical sides of the rectangle $T_1^{i+1} \wh R_1$ (that are made of pieces of $T_1^j\wh\phi_a$). 
Finally, the horizontal sides of the rectangle $\wh f^{k_1}(\wh R_1^i)$ lie in different connected components of $T_1^{i+1} \wh R_1 \cup T_1^{i}\wh\phi_a \cup T_1^{i+1}\wh\phi_a$. This proves we are in the configuration of Definition~\ref{def:markov}. 
\end{proof}

We can define similarly a rectangle $\wh R_2$ having its vertical sides included in $T_2^{-1}\wh\phi_b$ and $\wh\phi_b$, and some horizontal sub-rectangles $(\wh R_2^i)_{0\le i\le k_2-2}$ with the property that for any $0\le i\le k_2-2$, the rectangle $\wh f^{k_2}(\wh R_2^i)$ has a pre-Markovian intersection with both $T_2^{i+1} \wh R_2$ and $T_2^{i+2} \wh R_2$ (see Figure~\ref{FigConnectionHorse1}).

\begin{lemma}
For any $0\le i\le k_1-2$, the intersection $\wh f^{k_1r_1+r_1+r_2}(T_1^{-i-1}\wh R_1^i)\cap T_2^2\wh R_2$ is Markovian.
\end{lemma}

\begin{proof}
The configuration of this lemma is depicted in Figure~\ref{FigConnectionHorse2}.

\begin{figure}
\begin{center}

\tikzset{every picture/.style={line width=0.75pt}} 

\begin{tikzpicture}[x=0.75pt,y=0.75pt,yscale=-1.2,xscale=1.2]

\draw  [draw opacity=0][fill={rgb, 255:red, 208; green, 2; blue, 27 }  ,fill opacity=0.15 ] (257.87,103.15) .. controls (292.35,125.17) and (326.51,120.8) .. (347.41,103.3) .. controls (349.51,107.9) and (349.53,107.44) .. (353.28,110.21) .. controls (316.88,133.01) and (281.94,130.25) .. (252.81,109) .. controls (255.81,106.2) and (254.91,108.15) .. (257.87,103.15) -- cycle ;
\draw  [draw opacity=0][fill={rgb, 255:red, 127; green, 2; blue, 208 }  ,fill opacity=0.15 ] (245.67,191.9) .. controls (276.6,168.03) and (328.24,171.2) .. (355.3,193.47) .. controls (351.91,195.72) and (350.55,196.35) .. (348.4,199.98) .. controls (325.44,186.54) and (280.92,184.43) .. (254.12,197.63) .. controls (251.21,194.44) and (249.12,193.72) .. (245.67,191.9) -- cycle ;
\draw [color={rgb, 255:red, 65; green, 117; blue, 5 }  ,draw opacity=1 ]   (282,85.45) .. controls (282.18,115.3) and (321.58,115.1) .. (322,85.45) ;
\draw [shift={(298.06,107.4)}, rotate = 2.44] [fill={rgb, 255:red, 65; green, 117; blue, 5 }  ,fill opacity=1 ][line width=0.08]  [draw opacity=0] (7.14,-3.43) -- (0,0) -- (7.14,3.43) -- (4.74,0) -- cycle    ;
\draw [color={rgb, 255:red, 245; green, 166; blue, 35 }  ,draw opacity=1 ]   (212,85.45) .. controls (213.08,125.24) and (261.58,125.24) .. (262,85.45) ;
\draw [shift={(239.83,115.15)}, rotate = 180.93] [fill={rgb, 255:red, 245; green, 166; blue, 35 }  ,fill opacity=1 ][line width=0.08]  [draw opacity=0] (7.14,-3.43) -- (0,0) -- (7.14,3.43) -- (4.74,0) -- cycle    ;
\draw [color={rgb, 255:red, 245; green, 166; blue, 35 }  ,draw opacity=1 ]   (342,85.45) .. controls (343.08,125.24) and (391.58,125.24) .. (392,85.45) ;
\draw [shift={(369.83,115.15)}, rotate = 180.93] [fill={rgb, 255:red, 245; green, 166; blue, 35 }  ,fill opacity=1 ][line width=0.08]  [draw opacity=0] (7.14,-3.43) -- (0,0) -- (7.14,3.43) -- (4.74,0) -- cycle    ;
\draw [color={rgb, 255:red, 245; green, 166; blue, 35 }  ,draw opacity=1 ]   (282,220.42) .. controls (282.31,190.26) and (322.09,189.82) .. (322,220.42) ;
\draw [shift={(297.69,198.17)}, rotate = 355.31] [fill={rgb, 255:red, 245; green, 166; blue, 35 }  ,fill opacity=1 ][line width=0.08]  [draw opacity=0] (8.04,-3.86) -- (0,0) -- (8.04,3.86) -- (5.34,0) -- cycle    ;
\draw [color={rgb, 255:red, 65; green, 117; blue, 5 }  ,draw opacity=1 ]   (342,220.42) .. controls (342.31,180.49) and (391.86,180.26) .. (392,220.42) ;
\draw [shift={(370.41,190.62)}, rotate = 180.8] [fill={rgb, 255:red, 65; green, 117; blue, 5 }  ,fill opacity=1 ][line width=0.08]  [draw opacity=0] (8.04,-3.86) -- (0,0) -- (8.04,3.86) -- (5.34,0) -- cycle    ;
\draw [color={rgb, 255:red, 65; green, 117; blue, 5 }  ,draw opacity=1 ]   (212,220.42) .. controls (212.31,180.49) and (261.86,180.26) .. (262,220.42) ;
\draw [shift={(239.67,190.53)}, rotate = 178.95] [fill={rgb, 255:red, 65; green, 117; blue, 5 }  ,fill opacity=1 ][line width=0.08]  [draw opacity=0] (7.14,-3.43) -- (0,0) -- (7.14,3.43) -- (4.74,0) -- cycle    ;
\draw [line width=1.5]    (182,220.31) -- (422,220.31) ;
\draw [line width=1.5]    (182,85.33) -- (422,85.33) ;
\draw [color={rgb, 255:red, 65; green, 117; blue, 5 }  ,draw opacity=1 ]   (257.87,103.15) .. controls (292.48,124.69) and (325.97,121.52) .. (347.41,103.3) ;
\draw [color={rgb, 255:red, 245; green, 166; blue, 35 }  ,draw opacity=1 ]   (245.67,191.9) .. controls (276.71,167.52) and (328.58,171.77) .. (355.3,193.47) ;
\draw [color={rgb, 255:red, 245; green, 166; blue, 35 }  ,draw opacity=1 ]   (254.12,197.63) .. controls (281.46,184.02) and (326.46,186.52) .. (348.4,199.98) ;
\draw [color={rgb, 255:red, 65; green, 117; blue, 5 }  ,draw opacity=1 ]   (252.81,109) .. controls (282.08,129.94) and (317.42,133.44) .. (353.28,110.21) ;
\draw  [draw opacity=0][fill={rgb, 255:red, 127; green, 2; blue, 208 }  ,fill opacity=0.15 ] (226.43,202.04) .. controls (272.29,55.24) and (342.17,55.36) .. (376.32,198.71) .. controls (372.82,199.88) and (370.49,201.04) .. (366.99,203.04) .. controls (332.21,65.04) and (282.04,65.38) .. (237.93,204.04) .. controls (233.76,203.54) and (230.76,202.71) .. (226.43,202.04) -- cycle ;
\draw [color={rgb, 255:red, 245; green, 166; blue, 35 }  ,draw opacity=0.5 ]   (226.43,202.04) .. controls (272.29,55.12) and (342.38,55.55) .. (376.32,198.71) ;
\draw [color={rgb, 255:red, 245; green, 166; blue, 35 }  ,draw opacity=0.5 ]   (237.93,204.04) .. controls (282.25,65.32) and (332.29,65.36) .. (366.99,203.04) ;
\draw [color={rgb, 255:red, 65; green, 117; blue, 5 }  ,draw opacity=0.5 ]   (226.43,202.04) .. controls (230.99,202.76) and (234.26,203.17) .. (237.93,204.04) ;
\draw [color={rgb, 255:red, 65; green, 117; blue, 5 }  ,draw opacity=0.5 ]   (366.99,203.04) .. controls (369.93,200.89) and (372.59,199.78) .. (376.32,198.71) ;
\draw [color={rgb, 255:red, 74; green, 144; blue, 226 }  ,draw opacity=1 ]   (306.54,107.26) .. controls (310.56,150.48) and (309.97,157.02) .. (306.97,198.02) ;
\draw [shift={(309.35,148.69)}, rotate = 89.68] [fill={rgb, 255:red, 74; green, 144; blue, 226 }  ,fill opacity=1 ][line width=0.08]  [draw opacity=0] (7.14,-3.43) -- (0,0) -- (7.14,3.43) -- (4.74,0) -- cycle    ;
\draw [color={rgb, 255:red, 1; green, 59; blue, 151 }  ,draw opacity=1 ]   (303.12,91.68) .. controls (302.96,94.3) and (305.7,96.62) .. (305.77,100.17) ;
\draw [color={rgb, 255:red, 1; green, 59; blue, 151 }  ,draw opacity=1 ]   (320.63,80.65) .. controls (321.75,86.41) and (315.66,90.84) .. (309.33,93.58) ;
\draw [shift={(306.63,94.65)}, rotate = 340.2] [fill={rgb, 255:red, 1; green, 59; blue, 151 }  ,fill opacity=1 ][line width=0.08]  [draw opacity=0] (7.14,-3.43) -- (0,0) -- (7.14,3.43) -- (4.74,0) -- cycle    ;
\draw [color={rgb, 255:red, 208; green, 87; blue, 102 }  ,draw opacity=1 ]   (366.06,132.37) .. controls (352.89,137.03) and (343.41,138.01) .. (327.26,122.4) ;
\draw [shift={(325.2,120.37)}, rotate = 45.47] [fill={rgb, 255:red, 208; green, 87; blue, 102 }  ,fill opacity=1 ][line width=0.08]  [draw opacity=0] (7.14,-3.43) -- (0,0) -- (7.14,3.43) -- (4.74,0) -- cycle    ;
\draw [color={rgb, 255:red, 141; green, 43; blue, 204 }  ,draw opacity=1 ]   (221.7,135.22) .. controls (237.09,134.68) and (241.25,135.16) .. (253.78,144.38) ;
\draw [shift={(256.06,146.08)}, rotate = 217.1] [fill={rgb, 255:red, 141; green, 43; blue, 204 }  ,fill opacity=1 ][line width=0.08]  [draw opacity=0] (7.14,-3.43) -- (0,0) -- (7.14,3.43) -- (4.74,0) -- cycle    ;
\draw [color={rgb, 255:red, 141; green, 43; blue, 204 }  ,draw opacity=1 ]   (380.64,165.44) .. controls (361.89,164.8) and (353.8,166.97) .. (331.8,184.11) ;
\draw [shift={(329.72,185.75)}, rotate = 321.68] [fill={rgb, 255:red, 141; green, 43; blue, 204 }  ,fill opacity=1 ][line width=0.08]  [draw opacity=0] (7.14,-3.43) -- (0,0) -- (7.14,3.43) -- (4.74,0) -- cycle    ;

\draw (302.12,221.66) node [anchor=north] [inner sep=0.75pt]  [color={rgb, 255:red, 225; green, 141; blue, 0 }  ,opacity=1 ,xscale=1.2,yscale=1.2]  {$\wh \phi _{a}$};
\draw (368.12,220.82) node [anchor=north] [inner sep=0.75pt]  [color={rgb, 255:red, 65; green, 117; blue, 5 }  ,opacity=1 ,xscale=1.2,yscale=1.2]  {$T_{1}^{-1} \wh \phi _{b}$};
\draw (236.79,221.66) node [anchor=north] [inner sep=0.75pt]  [color={rgb, 255:red, 65; green, 117; blue, 5 }  ,opacity=1 ,xscale=1.2,yscale=1.2]  {$T_{1}^{-2} \wh \phi _{b}$};
\draw (277,79.93) node [anchor=south] [inner sep=0.75pt]  [color={rgb, 255:red, 65; green, 117; blue, 5 }  ,opacity=1 ,xscale=1.2,yscale=1.2]  {$\wh \phi _{c}$};
\draw (380,80.43) node [anchor=south] [inner sep=0.75pt]  [color={rgb, 255:red, 225; green, 141; blue, 0 }  ,opacity=1 ,xscale=1.2,yscale=1.2]  {$T_{2} \wh \phi _{b}$};
\draw (236.62,78.93) node [anchor=south] [inner sep=0.75pt]  [color={rgb, 255:red, 225; green, 141; blue, 0 }  ,opacity=1 ,xscale=1.2,yscale=1.2]  {$T_{2}^{2} \wh \phi _{b}$};
\draw (327,80) node [anchor=south] [inner sep=0.75pt]  [color={rgb, 255:red, 1; green, 59; blue, 151 }  ,opacity=1 ,xscale=1.2,yscale=1.2]  {$\wh f^{r_1+r_2}(\wh {\sigma } ')$};
\draw (367.36,127.98) node [anchor=west] [inner sep=0.75pt]  [color={rgb, 255:red, 208; green, 87; blue, 102 }  ,opacity=1 ,xscale=1.2,yscale=1.2]  {$T_{2}^{2}\wh {R}_{2}$};
\draw (220.74,137.41) node [anchor=east] [inner sep=0.75pt]  [font=\small,color={rgb, 255:red, 141; green, 43; blue, 204 }  ,opacity=1 ,xscale=1.2,yscale=1.2]  {$\wh {f}^{k_{1} r_{1} +r_{1} +r_{2}}\big( T_{1}^{i-1}\wh {R}_{1}^{i}\big)$};
\draw (382.28,162.74) node [anchor=west] [inner sep=0.75pt]  [font=\small,color={rgb, 255:red, 141; green, 43; blue, 204 }  ,opacity=1 ,xscale=1.2,yscale=1.2]  {$\wh {f}^{k_{1} r_{1}}\big( T_{1}^{i-1}\wh {R}_{1}^{i}\big)$};

\end{tikzpicture}

\caption{Proof of Theorem~\ref{ThConnectionHorse}: the Markovian intersection $\wh f^n(\wh R_1)\cap \wh R_2$.}\label{FigConnectionHorse2}
\end{center}
\end{figure}
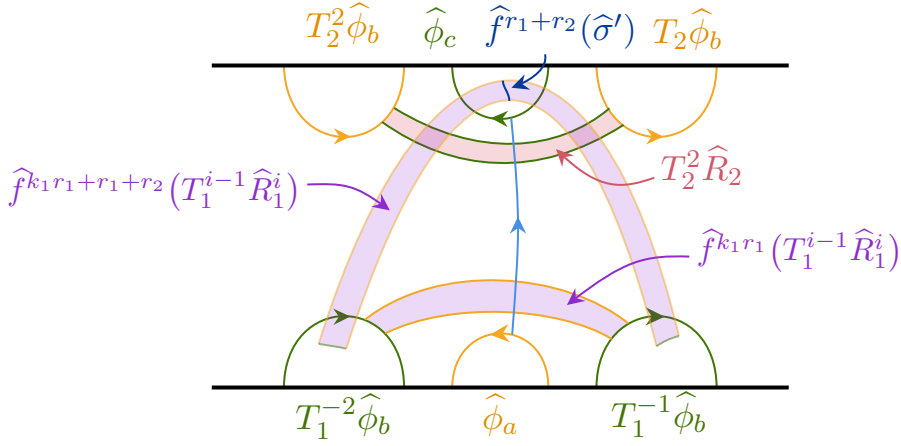

Note that the leaf $\wh \phi_b$ separates $T_1^{-i-1}\wh f^{k_1r_1}(\wh R_1^i)$ from $T_2^2\wh R_2$ (see Figure~\ref{FigConnectionHorse1}): recall that by Lemma~\ref{Lemma10} (more precisely, its version consisting in replacing $\wh\phi_a$ with $\wh\phi_b$), the vertical sides of $T_1^{-i-1}\wh f^{k_1r_1}(\wh R_1^i)$ --- that are made of pieces of $T_1^j\wh f^{k_1r_1}(\wh\phi_a)$ --- are disjoint from $\wh\phi_b$ (by the choice of $\delta_m$ and $\delta'_m$ made after Lemma~\ref{Lemma11}); a similar property holds for $T_2^2\wh R_2$.
Moreover, the leaves $\wh\phi_a$ and $\wh\phi_b$ are included in different connected components of the complement of $T_1^{-i-1}\wh f^{k_1r_1}(\wh R_1^i) \cup T_1^{-2}R(\wh\phi_b) \cup R(T_1^{-1} \phi_b)$. 
Similarly, the leaves $\wh\phi_b$ and $\wh\phi_c$ are included in different connected components of the complement of $T_2^2\wh R_2 \cup L(T_2\wh\phi_b) \cup L(T_2^2\wh \phi_b)$.

By hypothesis, we have that $\wh f^{-r_1-r_2}(\wh\phi_c)\cap \wh\phi_a\neq\emptyset$. As $R(\wh\phi_c)\subset R(\wh f^{-r_1-r_2}(\wh\phi_c))$, and as $R(\wh\phi_c)$ is a topological disk, there exists a path $\wh\sigma$, included in $R(\wh f^{-r_1-r_2}(\wh\phi_c))$, and linking $\wh\phi_a$ to $\wh\phi_c$. Note that by the above remark, the path $\wh\sigma$ is disjoint from $R(T_1^{-2}\wh\phi_b) \cup R(T_1^{-1} \wh\phi_b)$. As it links points of different connected components of the complement of $\wh f^{k_1r_1}(T_1^{-i-1}\wh R_1^i) \cup R(T_1^{-2}\wh\phi_b) \cup R(T_1^{-1} \wh\phi_b)$, it has to cross both vertical sides of $\wh f^{k_1r_1}(T_1^{-i-1}\wh R_1^i)$. Hence, there is a subpath $\wh\sigma'$ of $\sigma$ whose interior is included in the interior of the rectangle $\wh f^{k_1r_1}(T_1^{-i-1}\wh R_1^i)$ and that links both vertical sides of $T_1^{-i-1}\wh f^{k_1r_1}(\wh R_1^i)$. 

This path $\wh f^{-k_1r_1}(\wh\sigma')$ delimits two horizontal subrectangles of $T_1^{-i-1}\wh R_1^i$, that we denote $T_1^{-i-1}\wh R_{1}^{i,T}$ and $T_1^{-i-1}\wh R_1^{i,B}$. 

The image $\wh f^{r_1+r_2}(\wh\sigma')$ is included in $R(\wh\phi_c)$, while the horizontal sides of $\wh f^{k_1r_1+r_1+r_2}(T_1^{-i-1}\wh R_1^i)$ are included in $R(T_1^{-2}\wh\phi_b)$ and $R(T_1^{-1}\wh\phi_b)$, which are both included in $L(\wh\phi_b)$. Hence, the horizontal sides of both rectangles $\wh f^{k_1r_1+r_1+r_2}(T_1^{-i-1}\wh R_{1}^{i,T})$ and $\wh f^{k_1r_1+r_1+r_2}(T_1^{-i-1}\wh R_{1}^{i,B})$ lie in different connected components of the complement of $T_2^2\wh R_2 \cup L(T_2\wh\phi_b) \cup L(T_2^2\wh \phi_b)$. 

For their part, the vertical sides of both rectangles $\wh f^{k_1r_1+r_1+r_2}(T_1^{-i-1}\wh R_{1}^{i,T})$ and $\wh f^{k_1r_1+r_1+r_2}(T_1^{-i-1}\wh R_{1}^{i,B})$ are pieces of $\wh f^{k_1r_1+r_1+r_2}(T_1^j \wh\phi_a)$ and hence are disjoint from $L(T_2\wh\phi_b) \cup L(T_2^2\wh \phi_b)$; indeed for orientation reasons the leaves $T_1^j \wh\phi_a$, with $j\in\Z$, are included in $L(\wh\phi_b)$ which is disjoint from all the $L(T_2^\ell\wh\phi_b)$ for $\ell\in\Z$. 

We have proved we are in the configuration of Definition~\ref{def:markov}, this implies that the intersections $\wh f^{k_1r_1+r_1+r_2}(T_1^{-i-1}\wh R_1^{i,T})\cap T_2^2\wh R_2$ and $\wh f^{k_1r_1+r_1+r_2}(T_1^{-i-1}\wh R_1^{i,B})\cap T_2^2\wh R_2$ are pre-Markovian, proving the lemma.
\end{proof}

This lemma finishes the proof of our theorem, as the $\wh R_1^i$ are horizontal subrectangles of $\wh R_1$.
\end{proof}

\section{Heteroclinic connections between chaotic classes}

\subsection{A graph $G$}\label{SubSecConstrG}

We define an infinite graph coding the rotational behaviour of $f$ on $\bigcup_{i\in I_{\mathrm h}}\rho_i$. This construction is not canonical.

The vertices of this graph are some rectangles and the edges are given by Markovian intersections. We will build these rectangles in two steps, first getting some periodic points $z'_\omega$ whose trajectories are not simple, and then building from these rectangles $R_\omega$ and hence rotational horseshoes. From these we will get a second family of periodic points $z_\omega$ rotating as these horseshoes.
\bigskip

By Theorem~\ref{thm:ShapeRotationSet} and the construction of the surfaces $S_i$ following it, there exists a countable family $(r'_\omega)_{\omega\in \Omega} \in \rote(f)\cap H_1(S,\Q)$ that is dense in $\bigcup_{i\in I_{\mathrm h}} \rho_i$. 
Each $r'_\omega$ can be supposed to be the rotation vector of a periodic point $z'_\omega$, whose tracking geodesic $\gamma'_\omega$ is non simple (Proposition~\ref{LemNotSimpleTracking}). We also suppose (thanks to Proposition~\ref{LemNotSimpleTracking}) that the tracking geodesics $\gamma'_\omega$ are dense in $\bigcup_{i\in I_{\mathrm{h}}} \dot\Lambda_i$. Denote $q'_\omega$ the period of $z'_\omega$.

As the tracking geodesic $\gamma'_\omega$ is not simple, by Lemma~\ref{LemNotSimpleIntersect} (or alternatively \cite[Proposition 9.18]{pa}), the transverse trajectory $I^\Z_\F(z'_\omega)$ has a self $\F$-transverse intersection at $I^{t_\omega^-}_\F(z'_\omega) = I^{t_\omega^+}_\F(z'_\omega)$, with $t_\omega^-<t_\omega^+$. Note that for any $n\in\N$, the transverse trajectory $I^\Z_\F(z'_\omega)$ also has a self $\F$-transverse intersection at 
\begin{equation}\label{EqPlaceInter}
I^{t_\omega^-}_\F(z'_\omega) = I^{t_\omega^++nq'_\omega}_\F(z'_\omega),
\end{equation}
and that
\[\frac{\big[I^{[t_\omega^-,t_\omega^++nq'_\omega]}_\F(z'_\omega)\big]}{t_\omega^++nq'_\omega-t_\omega^-} \underset{n\to+\infty}\longrightarrow r'_\omega.\]
Therefore, for any $\omega\in\Omega$ one can choose $n_\omega$ large enough so that the family 
\[(r_\omega)_{\omega\in \Omega} := \left(\frac{\big[I^{[t_\omega^-,t_\omega^++n_\omega q'_\omega]}_\F(z'_\omega)\big]}{t_\omega^++n_\omega q'_\omega-t_\omega^-}\right)_{\omega\in \Omega}\]
of elements of $H_1(S,\Q)$ is dense in $\bigcup_{i\in I_{\mathrm h}} \rho_i$. 
Finally, we require that there exists $u_0\in (t_\omega^-, t_\omega^++n_\omega q'_\omega)$ such that the transverse trajectories $I^{(-\infty, u_0)}_\F(z'_\omega)$ and $I^{(u_0, +\infty)}_\F(z'_\omega)$ intersect $\F$-transversally at $I^{t_\omega^-}_\F(z'_\omega) = I^{t_\omega^++n_\omega q'_\omega}_\F(z'_\omega)$ (\emph{i.e.}~we require the intervals where the transverse intersection holds to be disjoint).

By Theorem~\ref{ThConnectionHorse}, this allows to build, for any $\omega\in\Omega$, a rectangle $R_\omega\subset \wt S$, an integer $q_\omega>0$ and a deck transformation $T_\omega\in\G$ such that (recall that Markovian intersections were defined in Definition~\ref{def:markov})
\begin{equation}\label{EqDefROmega}
\wt f^{q_\omega}(R_\omega)\cap_M T_\omega R_\omega 
\qquad \text{and}\qquad
\frac{[T_\omega]}{q_\omega} = r_\omega
\end{equation}
(with the notations above, one has $q_\omega = t_\omega^++n_\omega q'_\omega-t_\omega^-$ and $T_\omega$ is a deck transformation associated to the closed loop $I^{[t_\omega^-,t_\omega^++n_\omega q'_\omega]}_\F(z'_\omega)$). 
By Proposition~\ref{LemPointFixe}, this implies the existence of a point $\wt z_\omega\in \wt S$ such that $\wt f^{q_\omega}(\wt z_\omega) = T_\omega \wt z_\omega$. In particular, the projection $z_\omega$ of $\wt z_\omega$ on $S$ is periodic. If $n_\omega$ is large enough, then the uniform measure on the orbit of $z_\omega$ belongs to $\cl_i$ where $i\in I_{\mathrm h}$ is such that $z'_\omega \in \cl_i$ (by abuse of notation, we will denote $z_\omega\in\cl_i$), and the tracking geodesics $\gamma_{z_\omega}$ are dense in $\bigcup_{i\in I_{\mathrm{h}}} \dot\Lambda_i$; in the sequel we suppose these properties satisfied (in particular, if $r'_\omega\in\rho_i$, then $r_\omega \in\rho_i$ too).

\begin{definition}
The graph $G$ is defined as follows. Its vertices are the rectangles $R_\omega$ for $\omega\in\Omega$. 
Its edges are given by the relation $\to$ of Definition~\ref{DefConnecRect}. 
\end{definition}

To this graph $G$ are naturally associated subgraphs $(G_i)_{i\in I_{\mathrm{h}}}$ as follows: for $i\in I_{\mathrm{h}}$, the graph $G_i$ is the complete subgraph of $G$ whose vertices are the $R_\omega$ for which $z_\omega\in \rho_i$. 

The following lemma enlightens the structure of $G$.

\begin{lemma}\label{LemLotConnectionsRectangles}
Let $f\in\Homeo_0(S)$.
Let $i\in I_{\mathrm h}$. Then for any $\omega, \omega'\in G_i$, we have $R_\omega\to R_{\omega'}$.
\end{lemma}

\begin{proof}
Let $i\in I_{\mathrm h}$ and $R_\omega$, $R_{\omega'}\in G_i$. Denote $\mu$ and $\mu'$ the uniform measures on the periodic orbits of respectively $z'_\omega$ and $z'_{\omega'}$. By definition, there exist $\mu=\nu_1, \nu_2,\dots,\nu_\ell = \mu'$ such that for any $k$, there exists a geodesic in $\dot\Lambda_{\nu_k}$ and a geodesic in $\dot \Lambda_{\nu_{k+1}}$ that intersect transversally. 
Theorem~\ref{thm:equidistributiontheoremintro} ensures that tracking geodesics of typical points are dense in the $\dot\Lambda_{\nu_k}$, so for $\nu_k$-a.e.\ $z_k$ and $\nu_{k+1}$-a.e.\ $z_{k+1}$ the tracking geodesics $\gamma_{z_k}$ and $\gamma_{z_{k+1}}$ intersect transversally. 
By Proposition~\ref{LemNotSimpleTracking}, one can suppose that each $z_k$ is a periodic point whose tracking geodesic is not simple. 

By Lemma~\ref{LemNotSimpleIntersect}, for any $1\le k<\ell$ the transverse trajectories $I^\Z_\F(z_k)$ and $I^\Z_\F(z_{k+1})$ intersect $\F$-transversally, as well as both $I^\Z_\F(z_1)$ and $I^\Z_\F(z_\ell)$ have a self $\F$-transverse intersection.
Hence, 
\begin{itemize}
\item for any $1\le k<\ell$ there exists $s_k<t_k<u_k$ and $s'_k<t'_k<u'_k$ such that $I^{[s_k, u_k]}_\F(z_k)$ and $I^{[s'_k, u'_k]}_\F(z_{k+1})$ intersect $\F$-transversally at $I^{t_k}_\F(z_k)=I^{t'_k}_\F(z_{k+1})$ (see \eqref{EqPlaceInter});
\item there exists $s_0<t_0<u_0<s'_0<t'_0<u'_0$ such that $I^{[s_0, u_0]}_\F(z_1)$ and $I^{[s'_0, u'_0]}_\F(z_{1})$ intersect $\F$-transversally at $I^{t_0}_\F(z_1)=I^{t'_0}_\F(z_{1})$; moreover $t_0 = t_\omega^-$ and $t'_0 = t_\omega^+ +n_\omega q'_\omega$ (we consider the same self $\F$-transverse intersection of the trajectory of $z_1$ as the one used to create the rectangle $R_\omega$);
\item there exists $s_\ell<t_\ell<u_\ell<s'_\ell<t'_\ell<u'_\ell$ such that $I^{[s_\ell, u_\ell]}_\F(z_\ell)$ and $I^{[s'_\ell, u'_\ell]}_\F(z_{\ell})$ intersect $\F$-transversally at $I^{t_\ell}_\F(z_\ell)=I^{t'_\ell}_\F(z_{\ell})$; moreover $t_\ell = t_{\omega'}^-$ and $t'_\ell = t_{\omega'}^+ +n_{\omega'} q'_{\omega'}$ (see \eqref{EqPlaceInter}).
\end{itemize}
By periodicity of the points $z_k$, one can suppose that for any $1\le k\le \ell$ one has $u'_{k-1}\le s_k$. 

This allows to apply \cite[Corollary 21]{lct1} (which basically consists in applying $\ell-1$ times Proposition~\ref{propFondalct1}) that ensures that there exists $y\in S$ such that the concatenation 
\[I^{[s_0, t_1]}_\F(z_{1}) I^{[t'_1, t_2]}_\F(z_2)\dots I^{[t_{\ell-2}, t'_{\ell-1}]}_\F(z_{\ell-1}) I^{[t_{\ell-1}, u'_{\ell}]}_\F(z_{\ell})\]

of transverse trajectories is $\F$-equivalent to a subpath of $I^\Z_\F(y)$. Recall that the subpath $I^{[s_0, t_1]}_\F(z_{1})$ has a self $\F$-transverse intersection that creates the rectangle $R_\omega$, and that the subpath $I^{[t_{\ell-1}, u'_{\ell}]}_\F(z_{\ell})$ has a self $\F$-transverse intersection that creates the rectangle $R_{\omega'}$.
This allows to apply the second part of Theorem~\ref{ThConnectionHorse}, which implies that $R_\omega\to R_{\omega'}$.
\end{proof}

%
%

\begin{rem}
One may wonder if it is possible to get a stronger result of the following kind: for any finite graph $G'\subset G$, there exists a semi-conjugation of $f$ on a compact subset of $S$ to the Markov chain given by the subgraph $G'$. Such a result may require some freeness of the subgroup of the $\pi_1(S)$ generated by the deck transformations associated to the rectangles, as in \cite[Proposition 9.16]{pa} (the result we have in our case is Proposition~\ref{PropConnectRectEnsRot}, that corresponds to \cite[Proposition 9.17]{pa}).
\end{rem}

\begin{lemma}\label{PropRotFRotG0}
Let $f\in\Homeo_0(S)$.
Then for any $i\in I_{\mathrm h}$, we have $\overline{\rho_i} = \overline{\rot(G_i)}$.
\end{lemma}

\begin{proof}
The inclusion $\rot(G_i)\subset \overline{\rho_i}$ is trivial by construction of $G$ (by Theorem~\ref{thm:ShapeRotationSet}, the set $\overline{\rho_i}$ is convex); this implies that $\overline{\rot(G_i)}\subset \overline{\rho_i}$. The other inclusion $\rho_i \subset \overline{\rot(G_i)}$ comes from the density of the $(r_\omega)_{r_\omega\in\rho_i}$ in $\rho_i$.
\end{proof}

\begin{proof}[Proof of Proposition~\ref{PropRealByCompact}]
This is a direct consequence of Proposition~\ref{PropConnectRectEnsRot}, as the graphs $G_i$ are strongly connected (Lemma~\ref{LemLotConnectionsRectangles}).
\end{proof}
%
%
%
%
%
%
%

\subsection{Connections between chaotic classes}

Let us define five relations between classes; these relations will turn out being equivalent (Theorem~\ref{TheoEquiConnec2}) and correspond to heteroclinic connections between chaotic classes.

The first relation deals with convergence of empirical measures.

\begin{definition}\label{DefRelTo}
If $\mu_1$ and $\mu_2$ are measures of $\Merg(f)$ belonging to chaotic classes, we note $\mu_1 \to \mu_2$ if there exist $(x_k)\in S^\N$ and four sequences of times $n_k^{1,-} < n_k^{1,+} < n^{2,-}_k < n^{2,+}_k$ with $\lim_k n_k^{1,+}-n_k^{1,-} = \lim_k n_k^{2,+}-n_k^{2,-} =+\infty$ and such that
\begin{equation}\label{EqEtoile}
\frac {1}{n^{1,+}_k - n^{1,-}_k}\sum_{i=n^{1,-}_k}^{n^{1,+}_k-1} \delta_{f^i(x_k)} \xrightharpoonup[k\to+\infty]{} \mu_1 
\quad \textrm{and} \quad
\frac {1}{n^{2,+}_k - n^{2,-}_k}\sum_{i=n^{2,-}_k}^{n^{2,+}_k-1} \delta_{f^i(x_k)} \xrightharpoonup[k\to+\infty]{} \mu_2
\end{equation}
(for the weak-$*$ topology).

If $i, j\in I_{\mathrm h}$, we note $\cl_i \overset *\to \cl_j$ if there exist $\mu_1\in \cl_i$ and $\mu_2\in \cl_j$ such that $\mu_1\to\mu_2$.
\end{definition}

The second relation is formulated in terms of the forcing theory.

\begin{definition}\label{DefToStar}
For $i,j\in I_{\mathrm{h}}$, we write $\cl_i\overset{\F}{\to} \cl_j$ if there exist $a<b<c$, a transverse admissible path $\beta:[a,c]\to \mathrm{dom}(I)$, a lift $\wt\beta$ of $\beta$ to $\wt S$ and covering automorphisms $T_1,T_2\in \G$ such that:
\begin{itemize}
\item $\wt\beta|_{[a,b]}$ and $T_1(\wt\beta|_{[a,b]})$ have an $\wt{\mathcal F}$-transverse intersection at $\wt\beta(t_1)=T_1(\wt\beta)(s_1)$, where $s_1<t_1$, 
\item $\wt\beta|_{[b,c]}$ and $T_2(\wt\beta|_{[b,c]})$ have an $\wt{\mathcal F}$-transverse intersection at $\wt\beta(t_2)=T_2(\wt\beta)(s_2)$, where $s_2<t_2$,
\item for $k=1, 2$, denoting $\gamma_k$ the closed geodesic in the free homotopy class of the closed loop $\beta|_{[s_k, t_k]}$, we have $\gamma_1\subset\Lambda_{\cl_i}$ and $\gamma_2\subset \Lambda_{\cl_j}$.
\end{itemize}
\end{definition}

This relation depends \emph{a priori} on the choice of the isotopy $I$ and the foliation $\F$; however we will see it is in fact independent from these.

The third relation is about intersections of essential curves.

\begin{definition}\label{DefToWedge}
For $i,j\in I_{\mathrm{h}}$, we write $\cl_i\overset{\wedge}{\to} \cl_j$ if for any essential closed loops $\alpha_i, \alpha_j$ of $S$ such that $[\alpha_i]\in \pi_1(S_i, \Z)$ and $[\alpha_j]\in \pi_1(S_j, \Z)$, there exists $n\ge 0$ such that $f^n(\alpha_i)\cap \alpha_j \neq\emptyset$.
\end{definition}

The fourth relation concerns intersections of open essential sets.

\begin{definition}\label{DefToOpen}
For $i,j\in I_{\mathrm{h}}$, we write $\cl_i\overset{O}{\to} \cl_j$ if for any open subsets $B_i^-, B_j^+$ of $S$ such that:
\begin{itemize}
\item for any $\mu\in \cl_i$ and any $\mu'\in \cl_j$ we have $\mu(B_i^-) = \mu'(B_j^+) = 1$;
\item $f^{-1}(B_i^-) \subset B_i^-$ and $f(B_j^+)\subset B_j^+$;
\item $i_* \pi_1(S_i,\R)\subset i_* \pi_1(B_i^-, \R)$ and  $i_* \pi_1(S_j,\R)\subset i_* \pi_1(B_j^+, \R)$;
\end{itemize}
there exists $n\ge 0$ such that $f^n(B_i^-)\cap B_j^+ \neq\emptyset$.
\end{definition}

Finally, the last definition involves Markovian intersections in the graph $G$. 

\begin{definition}\label{DefRelMarkov}
Let $i,j\in I_{\mathrm{h}}$. We write $\cl_i\overset M\to \cl_j$ if there exist $\omega, \omega'\in\Omega$ such that $z_\omega\in\cl_i$, $z_{\omega'}\in\cl_j$ and a path in $G$ going from $R_\omega$ to $R_{\omega'}$.
\end{definition}

The fact that the $G_i$ are strongly connected (Lemma~\ref{LemLotConnectionsRectangles}) implies the following property. 
Let $i, j\in I_{\mathrm h}$ such that $\cl_i\overset M\to \cl_j$. Then for any $\omega, \omega'\in\Omega$ such that $z_\omega\in \cl_i$ and $z_{\omega'}\in\cl_j$, there is an oriented path in $G$ from $R_\omega$ to $R_{\omega'}$. Note that this property holds for $i=j$.

Lemma~\ref{LemLotConnectionsRectangles} also implies that the $G_i$ are strongly connected in $G$. We will see later (Proposition~\ref{PropToOrderRel}) that they actually coincide with the strong connected components of $G$. 

In view of the proof of Theorem~\ref{TheoEquiConnec2}, let us establish some implications between the relations between classes.

\begin{lemma}\label{LemFImpliesM}
Let $f\in\Homeo_0(S)$ and $i,j\in I_{\mathrm{h}}$. If $\cl_i\overset \F\to \cl_j$ then $\cl_i\overset M\to \cl_j$.
\end{lemma}

\begin{proof}
As $\cl_i\overset \F\to \cl_j$, there exist $a<b<c$, a transverse admissible path $\beta:[a,c]\to \mathrm{dom}(I)$, and covering automorphisms $T_1,T_2\in \G$ such that:
\begin{itemize}
\item $\wt\beta|_{[a,b]}$ and $T_1(\wt\beta|_{[a,b]})$ have an $\wt{\mathcal F}$-transverse intersection at $\wt\beta(t_1)=T_1(\wt\beta)(s_1)$, where $s_1<t_1$, 
\item $\wt\beta|_{[b,c]}$ and $T_2(\wt\beta|_{[b,c]})$ have an $\wt{\mathcal F}$-transverse intersection at $\wt\beta(t_2)=T_2(\wt\beta)(s_2)$, where $s_2<t_2$,
\item for $i=1, 2$, denoting $\gamma_i$ the closed geodesic in the free homotopy class of the closed loop $\beta|_{[s_i,t_i]}$, we have $\gamma_1\subset\Lambda_{\cl_i}$ and $\gamma_2\subset \Lambda_{\cl_j}$.
\end{itemize}

For $k=1,2$ denote $\alpha_k$ the transverse loop $\beta_{[s_k,t_k]}$. By Theorem~\ref{ThConnectionHorse} (or more simply \cite[Theorem M]{lct2}) there exists $z_k$ an $f$-periodic orbit whose transverse trajectory is freely homotopic to $\alpha_k$; by hypothesis, one has $z_1\in\cl_i$ and $z_2\in\cl_j$.

As in the proof of Lemma~\ref{LemLotConnectionsRectangles}, using Theorem~\ref{thm:equidistributiontheoremintro} and Proposition~\ref{LemNotSimpleTracking}, for $k=1,2$ we can find $z'_k$ a periodic orbit whose tracking geodesic $\gamma_{z'_k}$ is not simple and intersects $\gamma_{z_k}$.
As the tracking geodesics of the $z'_\omega$ are dense in $\dot\Lambda_{\cl_i}$, one can suppose that $z'_k = z'_{\omega_k}$ for some $\omega_k\in\Omega$.

\begin{claim}\label{ClaimCool}
For any $M\in\N$ there exists $a'<b'<c'$ and a transverse admissible path $\beta':[a',c']\to \mathrm{dom}(I)$ such that $\beta'|_{[a',b']}$ has a subpath $\F$-equivalent to $I^{[0,M]}_\F(z'_1)$ and $\beta'|_{[b', c']}$ has a subpath $\F$-equivalent to $I^{[0,M]}_\F(z'_2)$.
\end{claim}

\begin{proof}
We will see this is a consequence of Theorem~\ref{ThGT}. Apply this theorem to the periodic points $z'_1$ and $z'_2$; this gives us a constant $D'>0$.

Using Lemma~\ref{LemUseResidFinite}, and fixing lifts $\wt z'_k$ of $z'_k$ (for $k=1,2$) to $\wt S$, we deduce that there exists $N>0$ such that for $k=1,2$, denoting $\wt\alpha_k$ a lift of $\alpha_k$ to $\wt S$, the trajectory $\wt\alpha_k^N$ crosses 4 of the sets $(R_k^j V_{D'}(\wt\gamma_{\wt z'_{k}}))_{1\le j\le 4}$ that have the same orientation (with $R_1^j\in\G$). 
By \cite[Proposition~23]{lct1} (this is also a consequence of the proof of Theorem~\ref{ThConnectionHorse}), the concatenation $\alpha_1^N \beta_{[t_1, s_2]} \alpha_2^N$ is admissible, and crosses all the sets $(T_1R_1^j V_{D'}(\wt\gamma_{\wt z'_{1}}))_{1\le j\le 4}$ and $(T_2R_2^j V_{D'}(\wt\gamma_{\wt z'_{2}}))_{1\le j\le 4}$ for some $T_1, T_2\in\G$. By considering a bigger $N$ if necessary and using Lemma~\ref{LemUseResidFinite}, one can suppose that the sets $(T_1R_1^j V_{D'}(\wt\gamma_{\wt z'_{1}}))_{1\le j\le 4}$ and $(T_2R_2^j V_{D'}(\wt\gamma_{\wt z'_{2}}))_{1\le j\le 4}$ are pairwise disjoint.

Theorem~\ref{ThGT} applied to the trajectories $\alpha_1^N \beta_{[t_1, s_2]} \alpha_2^N$, $I^\Z_\F(z'_1)$ and $I^\Z_\F(z'_2)$ then asserts that there exists an $\wt f$-admissible transverse path $\beta'$ as well as $a'<b'<c'$ such that $\beta'|_{[a',b']}$ has a subpath $\F$-equivalent to $I^{[0,M]}(z'_1)$ and $\beta'|_{[b', c']}$ has a subpath $\F$-equivalent to $I^{[0,M]}(z'_2)$.
\end{proof}

Let us come back to the proof of Lemma~\ref{LemFImpliesM}. 
As in the proof of Lemma~\ref{LemLotConnectionsRectangles}, we use Theorem~\ref{ThConnectionHorse} that implies that $R_{\omega_1}\to R_{\omega_2}$, and hence $\cl_i\overset M\to \cl_j$.
\end{proof}

\begin{lemma}\label{LemStarImpliesF}
Let $f\in\Homeo_0(S)$ and $i,j\in I_{\mathrm{h}}$. If $\cl_i\overset *\to \cl_j$ then $\cl_i\overset \F\to \cl_j$.
\end{lemma}

\begin{proof}
Let $i,j\in I_{\mathrm{h}}$ be such that $\cl_i\overset *\to \cl_j$. Then there exist $\mu_1\in \cl_i$, $\mu_2\in \cl_j$, $(x_k)\in S^\N$ and four sequences of times $n_k^{1,-} < n_k^{1,+} < n^{2,-}_k < n^{2,+}_k$ with $\lim_k n_k^{1,+}-n_k^{1,-} = \lim_k n_k^{2,+}-n_k^{2,-} =+\infty$ and such that \eqref{EqEtoile} holds. 
As $\mu_1\in \cl_i$, $\mu_2\in \cl_j$, there exists $\nu_1\in \cl_i$, $\nu_2\in \cl_j$ such that $\mu_1$ and $\nu_1$ are dynamically transverse, and  $\mu_2$ and $\nu_2$ are dynamically transverse. By Proposition~\ref{LemNotSimpleTracking}, we can suppose that $\nu_1$ and $\nu_2$ are periodic measures whose tracking geodesics are closed and non simple. Let $z_1$, $z_2$ be associated periodic points, $q_1, q_2$ their periods and $\wt z_1$, $\wt z_2$ some lifts of them to $\wt S$. 

By Lemma~\ref{LemNotSimpleIntersect} (or \cite[Proposition 9.18]{pa}), the transverse trajectory $I^\Z_\F(z_1)$ has a self $\F$-transverse intersection at $I^{s_1}_\F(z_1) = I^{t_1}_\F(z_1)$, $s_1<t_1$, and the transverse trajectory $I^\Z_\F(z_2)$ has a self $\F$-transverse intersection at $I^{s_2}_\F(z_2) = I^{t_2}_\F(z_2)$, $s_2<t_2$. 
By \cite[Proposition~8.5]{pa} we see that for some $\ell\in\{0,1\}$, the geodesic in the free homotopy class of $I^{[s_1,t_1+\ell q_1]}_\F(z_1)$ crosses $\gamma_{z_1}$ and is non simple. From now we replace $t_1$ by $t_1+\ell q_1$ (and the same for $z_2$).

Let us apply Theorem~\ref{ThGT} to the periodic points $z_{1}$ and $z_{2}$, and $M = \max(s_1-t_1+2q_1, s_2-t_2+2q_2)$, which gives us a constant $D'$.
Using Lemma~\ref{LemUseResidFinite}, and fixing lifts $\wt x_k$ of $x_k$ to $\wt S$, we deduce that for any $k$ large enough the trajectory $I^{[n_k^{1,-}, n_k^{1,+}]}_{\wt\F}(\wt x_k)$ crosses $4$ of the sets $(R_1^j V_{D'}(\wt\gamma_{\wt z_{1}}))_{1\le j\le 4}$ (with $R_1^j\in\G$) that are pairwise disjoint and have the same orientation, and the trajectory $I^{[n_k^{2,-}, n_k^{2,+}]}_{\wt\F}(\wt x_k)$ crosses $4$ of the sets $(R_2^j V_{D'}(\wt\gamma_{\wt z_{2}}))_{1\le j\le 4}$ (with $R_2^j\in\G$) that are pairwise disjoint and have the same orientation.

Theorem~\ref{ThGT} then asserts that there exists an $\wt f$-admissible transverse path $\wt\beta$ made of the concatenation of some paths $I^{[s'_1,t'_1]}_\F(z_1)$, $I^{[u_1, u_2]}_\F(y_0)$ and $I^{[s'_2,t'_2]}_\F(z_2)$, with $t'_1-s'_1\ge M$ and $t'_2-s'_2\ge M$. 
Hence the subpath $I^{[s'_1,t'_1]}_\F(z_1)$ has a self $\F$-transverse intersection, and by what we have stated above the geodesic in the free homotopy class of the loop created by the self $\F$-transverse intersection is included in $\Lambda_{\cl_i}$. 
Similarly, the subpath $I^{[s'_2,t'_2]}_\F(z_1)$ has a self $\F$-transverse intersection, and by what we have stated above the geodesic in the free homotopy class of the loop created by the self $\F$-transverse intersection is included in $\Lambda_{\cl_j}$.
\end{proof}

\begin{lemma}\label{LemMImpliesStar}
Let $f\in\Homeo_0(S)$ and $i,j\in I_{\mathrm{h}}$. If $\cl_i\overset M\to \cl_j$ then $\cl_i\overset *\to \cl_j$.

Moreover, the deck transformations $T_1$ and $T_2$ of Definition~\ref{DefToStar} of the relation $\overset *\to$ can be supposed to have non simple axes. 
\end{lemma}

This lemma will be the consequence of the following claim, that will also be used in next lemma.

\begin{claim}\label{ClaimLemMImpliesStar}
Let $f\in\Homeo_0(S)$ and $i,j\in I_{\mathrm{h}}$. If $\cl_i\overset M\to \cl_j$ then for any $\omega, \omega'\in\Omega$ such that $z_\omega\in\cl_i$, $z_{\omega'}\in\cl_j$, there exist $\wt x\in \wt S$, $D>0$, lifts $\wt\gamma_{\omega}$ and $\wt\gamma_{\omega}$ of tracking geodesics of $z_\omega$ and $z_{\omega'}$, and $\vartheta_i>0$ and $\vartheta_j>0$, such that for any $n\ge 0$,
\begin{equation}\label{EqKTrack2}
d\big(\wt f^{-n}(\wt x), \wt\gamma_{\omega}(-n\vartheta_i)\big)\le D
\qquad \text{and}\qquad
d\big(\wt f^n(\wt x), \wt\gamma_{\omega'}(n\vartheta_j)\big)\le D.
\end{equation}
\end{claim}

\begin{rem}\label{RemClaimLemMImpliesStar}
The conclusion of the claim persists if we replace the hypothesis $\cl_i\overset M\to \cl_j$ by a path of connections $\cl_i\overset M\to \cl_{i_1}\overset M\to\dots \overset M\to \cl_j$. This implies that the conclusion of Lemma~\ref{LemMImpliesStar} also persists under this weaker conclusion.
\end{rem}

\begin{proof}
Suppose that $\cl_i\overset M\to \cl_j$. Lemma~\ref{LemLotConnectionsRectangles} implies that for any $\omega, \omega'\in\Omega$ such that $z_\omega\in\cl_i$, $z_{\omega'}\in\cl_j$, there exists a path in $\G$ going from $R_\omega$ to $R_{\omega'}$. We treat the case where this path has length 1 (\emph{i.e.}~$R_\omega\to R_{\omega'}$), the general case being more technical but similar.

So there are $n_\omega, n_{\omega, \omega'}, n_{\omega'}\in\N$ and $T_\omega, T_{\omega, \omega'}, T_{\omega'}\in\G$ such that the following intersections are Markovian (in $\wt S$): $\wt f^{n_\omega}(R_\omega) \cap T_\omega(R_\omega)$, $\wt f^{n_{\omega, \omega'}}(R_\omega) \cap T_{\omega, \omega'}(R_{\omega'})$ and $\wt f^{n_{\omega'}}(R_{\omega'}) \cap T_{\omega'}(R_{\omega'})$. By Proposition~\ref{PropConnectRectEnsRot}, the following intersection is nonempty and compact: 
\[\wt K = \left(\bigcap_{\ell\in\N} \wt f^{-\ell n_{\omega'}-n_{\omega,\omega'}}(T_{\omega,\omega'} T_{\omega'}^\ell R_{\omega'})\right)\cap \left(\bigcap_{\ell\in\N} \wt f^{\ell n_{\omega}}(T_{\omega}^{-\ell} R_{\omega})\right).\]

Let $\wt\gamma_\omega$ and $\wt\gamma_{\omega'}$ be the geodesic axes of respectively $T_\omega$ and $T_{\omega,\omega'}T_{\omega'}T_{\omega,\omega'}^{-1}$. 
Let also
\[K_i = \pr_S\left(\bigcup_{j=0}^{n_{\omega'}-1}\bigcap_{\ell\in\N} \wt f^{-\ell n_{\omega'}-n_{\omega,\omega'}-j}(T_{\omega'}^\ell T_{\omega,\omega'} R_{\omega'})\right)\]
and 
\[K_j = \pr_S\left(\bigcup_{j=0}^{n_{\omega}-1}\bigcap_{\ell\in\N} \wt f^{\ell n_{\omega}+j}(T_{\omega}^{-\ell} R_{\omega})\right).\]
Note that the set $K_i$ is backward invariant and $K_j$ is forward invariant.
Because the rectangles $R_\omega$ and $R_{\omega'}$ are bounded, there exist $D>0$, $\vartheta_i>0$ and $\vartheta_j>0$ such that for any $x_i\in K_i$, there is a lift $\wt x_i$ of $x_i$ to $\wt S$ and for any $x_j\in K_j$, there is a lift $\wt x_j$ of $x_j$ to $\wt S$ such that for any $n\in\N$, we have 
\begin{equation}\label{EqKTrack}
d\big(\wt f^{-n}(\wt x_i), \wt\gamma_{\omega}(-n\vartheta_i)\big)\le D
\qquad \text{and}\qquad
d\big(\wt f^n(\wt x_j), \wt\gamma_{\omega'}(n\vartheta_j)\big)\le D.
\end{equation}
This implies that for any $x\in K$, \eqref{EqKTrack2} holds. 
\end{proof}

\begin{proof}[Proof of Lemma~\ref{LemMImpliesStar}]
Consider any $\omega, \omega'\in\Omega$ such that $z_\omega\in\cl_i$, $z_{\omega'}\in\cl_j$, and that the closed geodesics $\gamma_\omega$ and $\gamma_{\omega'}$ are non simple.

Let $x_0$ given by Claim~\ref{ClaimLemMImpliesStar}. 
Choose $\mu_j$ and $\mu_i$ some weak-$*$ limits of the respective sequences:
\[\frac{1}{n}\sum_{k=0}^{n-1} \delta_{f^k(x)}
\qquad \text{and}\qquad
\frac{1}{n}\sum_{k=0}^{n-1} \delta_{f^{-k}(x)}.\]
Choose some ergodic measures $\nu_i$ and $\nu_j$ that are typical for the ergodic decompositions of respectively $\mu_i$ and $\mu_j$. 
Because $K_i$ is backward invariant and $K_j$ is forward invariant, we deduce that $\supp (\nu_i)\subset K_i$ and $\supp(\nu_j)\subset K_j$. Hence for $\nu_i$-a.e.~$x_i$ and $\nu_j$-a.e.~$x_j$, \eqref{EqKTrack} holds. 
In particular, $\nu_i,\nu_j\in \Merg(f)$. Moreover, $\wt\gamma_{\wt\omega}$ is a tracking geodesic for $\wt x_i$ and $\wt\gamma_{\omega'}$ is a tracking geodesic for $\wt x_j$; this implies (because $\gamma_\omega$ and $\gamma_{\omega'}$ are non simple) that $\nu_i\in \cl_i$ and $\nu_j\in \cl_j$.

Finally, as $\nu_i$ is typical for the ergodic decomposition of $\mu_i$, any point of the support of $\nu_i$ is accumulated by negative iterates of $x$; similarly any point of the support of $\nu_j$ is accumulated by positive iterates of $x$. This implies the existence of $n_k^{1,-} < n_k^{1,+} < n^{2,-}_k < n^{2,+}_k$ with $\lim_k n_k^{1,+}-n_k^{1,-} = \lim_k n_k^{2,+}-n_k^{2,-} =+\infty$ and such that
\begin{equation*}
\frac {1}{n^{1,+}_k - n^{1,-}_k}\sum_{i=n^{1,-}_k}^{n^{1,+}_k-1} \delta_{f^i(x)} \xrightharpoonup[k\to+\infty]{} \nu_i 
\quad \textrm{and} \quad
\frac {1}{n^{2,+}_k - n^{2,-}_k}\sum_{i=n^{2,-}_k}^{n^{2,+}_k-1} \delta_{f^i(x)} \xrightharpoonup[k\to+\infty]{} \nu_j.
\end{equation*}
This shows that $\nu_i\overset *\to \nu_j$.
\end{proof}

\begin{lemma}\label{LemMImpliesWedge}
Let $f\in\Homeo_0(S)$ and $i,j\in I_{\mathrm{h}}$. If $\cl_i\overset M\to \cl_j$ then $\cl_i\overset \wedge\to \cl_j$.
\end{lemma}

\begin{proof}
Consider two essential closed loops $\alpha_i, \alpha_j$ of $S$ such that $[\alpha_i]\in \pi_1(S_i, \Z)$ and $[\alpha_j]\in \pi_1(S_j, \Z)$, as well as $\omega_i, \omega_j\in\Omega$ such that $z_{\omega_i}\in\cl_i$, $z_{\omega_j}\in\cl_j$, $\gamma_{z_{\omega_i}}$ and $[\alpha_i]$ intersect geometrically, and $\gamma_{z_{\omega_j}}$ and $[\alpha_j]$ intersect geometrically.

If $\alpha_i$ and $\alpha_j$ are not disjoint, there is nothing to prove, so we suppose this is not the case.

As $\cl_i\overset M\to \cl_j$, one can use Claim~\ref{ClaimLemMImpliesStar}: there exist $\wt x\in \wt S$, $D>0$, $\vartheta_i>0$ and $\vartheta_j>0$ such that \eqref{EqKTrack2} holds. Consider two lifts $\wt\alpha_i$ and $\wt\alpha_j$ of $\alpha_i$ and $\alpha_j$ to $\wt S$ such that both separate $\alpha(\wt\gamma_{\omega})$ from $\omega(\wt\gamma_{\omega'})$ (the $\alpha$ and $\omega$ limits of these geodesics, that belong to $\partial\wt S$).

This implies that there exist $n\in\N$ such that $\wt f^{-n}(\wt x)$ and $\wt f^n(\wt x)$ belong to different unbounded connected components of the complements of both $\wt\alpha_i$ and $\wt\alpha_j$. As the action of $\wt f$ on $\partial \wt S$ is the identity, we deduce that $\wt f^{2n}(\wt \alpha_i)\cap\alpha_j\neq\emptyset$. 
\end{proof}

\begin{prop}\label{LemOImpliesF}
Let $f\in\Homeo_0(S)$.
For any $i\in I_{\mathrm h}$, there exist three open filled connected sets $B_i^o, B_i^+, B_i^-$ of $S$, with $B_i^o\subset B_i^+\cap B_i^-$, such that 
\begin{itemize}
\item for any $\mu\in \cl_i$, we have $\mu(B_i^-) = \mu(B_i^+) = \mu(B_i^o) = 1$;
\item $f^{-1}(B_i^-) \subset B_i^-$, $f(B_i^+)\subset B_i^+$ and $f(B_i^o) = B_i^o$;
\item $i_* \pi_1(S_i)\subset i_* \pi_1(B_i^-)\cap i_* \pi_1(B_i^+) \cap i_* \pi_1(B_i^o)$
\end{itemize}
and satisfying:
\begin{itemize}
\item if $f^n(B_i^-) \cap B_j^+ \neq \emptyset$ for some $n\in\N$, then $\cl_i \overset *{\to}\cl_j$;
\item if $B_i^o \cap B_j^o \neq \emptyset$, then $\cl_i \overset *{\to}\cl_j$ or $\cl_j \overset *{\to}\cl_i$;
\item if $\cl_i \overset *{\to}\cl_j$, then $B_i^-\cap B_j^-\neq\emptyset$, $B_i^+\cap B_j^+\neq\emptyset$ and $B_i^+\cap B_j^-\neq\emptyset$, and there exists $n\in\N$ such that $f^n(B_i^-) \cap B_j^+ \neq \emptyset$.
\end{itemize}
\end{prop}

\begin{proof}
As in the beginning of Subsection~\ref{SubSecConstrG}, we use Theorem~\ref{thm:ShapeRotationSet}, this time to get a finite family of periodic points $(p_k)_k$ such that any closed geodesic included in one of the open surfaces $S_i$ ($i\in I_{\mathrm h}$) crosses one of the tracking geodesics $\gamma_{p_k}$.
By Proposition~\ref{LemNotSimpleTracking} one can moreover suppose that the tracking geodesic $\gamma_k$ of each $p_k$ is not simple. 

Let $D'>0$ be a constant given by Theorem~\ref{ThGT} working for all the couples of periodic points in the family $(p_k)_k$, and $N_0>0$ be the constant given by Lemma~\ref{LemUseResidFinite} applied to $D'$. 
Consider $i\in I_{\mathrm h}$, $\mu\in \cl_i$ and a point $z$ that is typical for $\mu$. There exists $k_\mu$ such that the periodic measure associated with $p_{k_\mu}$ belongs to $\cl_i$, and such that $\gamma_{p_{k_\mu}}$ intersects any tracking geodesic of a $\mu$-typical point (we use the density of tracking geodesics in $\Lambda_\mu$, see Theorem~\ref{thm:equidistributiontheoremintro}).


Define $E_i$ as the set of recurrent points $y\in S$ such that the following is true: there exists $n_y^-<m_y^-<0<m_y^+,<n_y^+$ and an open disk $V_y$ containing $y$ such that:
\begin{itemize}
\item for any $x\in V_y$, both trajectories $I^{[m_y^-,0]}(x)$ and $I^{[0,m_y^+]}(x)$ have geometric intersection numbers at least $N_0$ with $\gamma_{k_y}$ for some $k_y$ such that the periodic measure associated with $p_{k_y}$ belongs to $\cl_i$;
\item both trajectories $I^{[n_y^-,0]}(y)$ and $I^{[0,n_y^+]}(y)$ have geometric intersection numbers at least $N_0$ with $\gamma_{k_y}$;
\item $f^{n_y^-}(y), f^{n_y^+}(y)\in V_y$.
\end{itemize}
By the previous paragraph, the set $E_i$ has full $\mu$-measure for any $\mu\in\cl_i$.
%
%

We then set (recall that the \emph{fill} of an open set is the union of this set with the connected components of its complement whose lifts to $\wt S$ are bounded)
\[B_i = \bigcup_{y\in E_i}V_y,
\qquad
B_i^o = \operatorname{fill}\left(\bigcup_{n\in\Z} f^{n} (B_i)\right),\]
\[B_i^- = \operatorname{fill}\left(\bigcup_{n\ge 0} f^{-n} (B_i)\right)
\qquad\text{and}\qquad
B_i^+ = \operatorname{fill}\left(\bigcup_{n\ge 0} f^{n} (B_i)\right).\]

Consider a point $z\in S$ that is typical for some $\mu\in \cl_i$, whose tracking geodesic crosses any closed geodesic of $S_i$ and whose rotation vector is not rational. In particular, we have $z\in E_i \subset B_i$. Consider the connected component $W_i^+$ of $B_i^+$ containing $z$.
By \cite[Lemma~8.4 and Remark~8.5]{paper2PAF}, we get $i_* \pi_1(S_i)\subset i_* \pi_1(B_i^+)$. 
By construction, for any $\mu\in \cl_i$, we have $\mu(B_i^+) = 1$.

Let us prove these sets are connected.
Because $W_i^+$ is essential and open, any of its lifts $\wt W_i^+$ to $\wt S$ satisfies $\wt f(\wt W_i^+)\subset \wt W_i^+$.
Suppose there exists $U_i^+$ a connected component of $B_i^+$ different from $W_i^+$. Let 
$y\in E_i\cap U_i^+$, we have $f^{n_y^+}(U_i^+)\subset U_i^+$ (because $B_i^+$ is positively $f$-invariant).
Let $\wt U_i^+$ be a lift of $U_i^+$ to $\wt S$
and $T_0\in\G$ such that $\wt f^{n_y^+}(\wt U_i^+)\subset T_0\wt U_i^+$. By construction, the axis of $T_0$ crosses geometrically a path $\wt \beta_i$ in $\wt W_i^+$ lifting a simple closed loop of $W_i^+$. There exists $m\in\Z$ such that $U_i^+$ lies between $T_0^{-1}\wt \beta_i$ and $\wt \beta_i$. Hence $\wt f^{n_y^+}(\wt U_i^+)\subset T_0\wt U_i^+$ lies between $\wt \beta_i$ and $T_0\wt \beta_i$. But the union $\wt W_i^R$ of the connected components of the complement of $\wt W_i^+$ that lie to the right of $\wt \beta_i$ is negatively $\wt f$-invariant, so $\wt f(\wt U_i^+) \subset \wt W_i^R \subset \wt f(\wt W_i^R)$ and hence $\wt U_i^+ \subset \wt W_i^R$, a contradiction.
\bigskip

Suppose now that for some $n\in\N$ we have $f^n(B_i^-) \cap B_j^+ \neq \emptyset$. The fact that the lifts of $B_i^-$ and $B_j^+$ are unbounded implies that there exist $m\ge 0$ and $y\in B_i$ such that $f^m(y)\in B_j$.
Applying Theorem~\ref{ThGT} as in Claim~\ref{ClaimCool}, we get that $\cl_i \overset\F\to \cl_j$; applying Lemmas~\ref{LemFImpliesM} and \ref{LemMImpliesStar} implies that $\cl_i \overset*\to \cl_j$.
Hence, $\cl_i \overset *{\not\to}\cl_j$ implies that for any $n\in\N$ we have $f^n(B_i^-) \cap B_j^+ = \emptyset$; moreover the conditions $\cl_i \overset *{\not\to}\cl_j$ and $\cl_j \overset *{\not\to}\cl_i$ imply that for any $n\in\Z$ we have $f^n(B_i^o) \cap B_j^o = \emptyset$.

Suppose now that for some $n\in\Z$ we have $B_i^o \cap B_j^o \neq \emptyset$. As in the previous paragraph, this implies that there exists $m\in\Z$ and $y\in B_i$ such that $f^m(y)\in B_j$. If $m\ge 0$, then the previous paragraph implies that $\cl_i \overset*\to \cl_j$, and if $m\le 0$ we get that $\cl_j \overset*\to \cl_i$.

Finally, suppose that $\cl_i \overset*\to \cl_j$. By definition, there exist $\mu_1\in \cl_i$, $\mu_2\in \cl_j$, $(x_k)\in S^\N$ and four sequences of times $n_k^{1,-} < n_k^{1,+} < n^{2,-}_k < n^{2,+}_k$ with $\lim_k n_k^{1,+}-n_k^{1,-} = \lim_k n_k^{2,+}-n_k^{2,-} =+\infty$ and such that \eqref{EqEtoile} holds. By the fact that $\mu_1(B_i^-) = \mu_2(B_j^-) = 1$, there exists $z_1\in B_i^-, z_2\in B_j^-$ that are respectively $\mu_1$ and $\mu_2$-typical. Hence, for $k$ large enough there exists $m_1^k<m_2^k$ such that $f^{m_1^k}(x_k)\in B_i^-$ and $f^{m_2^k}(x_k)\in B_j^-$. In particular, there exists $m\ge 0$ such that $f^m(B_i) \cap B_j\neq \emptyset$. This proves the last point of the lemma and finishes the proof.
\end{proof}
%

\begin{proof}[Proof of Theorem~\ref{TheoEquiConnec2}]
We prove the following implications:
\begin{center}
\begin{tikzpicture}[scale=.6]
\tikzset{implication/.style={double, double equal sign distance, -implies}}

\node (O) at (0, 2) {$O$};
\node (F) at (2,2) {$\F$};
\node (W) at (0,0) {$\wedge$};
\node (M) at (2,0) {$M$};
\node (S) at (3.7, 1) {$*$};

\draw[implication] (O) -- (F);
\draw[implication] (F) -- (M);
\draw[implication] (M) -- (W);
\draw[implication] (W) -- (O);
\draw[implication] (M) -- (S);
\draw[implication] (S) -- (F);
\end{tikzpicture}
\end{center}

\smallskip\noindent\textbf{\textit{$\F$ $\implies$ $M$}:} This is Lemma~\ref{LemFImpliesM}.

\smallskip\noindent\textbf{\textit{$M$ $\implies$ $*$}:} This is Lemma~\ref{LemMImpliesStar}.

\smallskip\noindent\textbf{\textit{$*$ $\implies$ $\F$}:} This is Lemma~\ref{LemStarImpliesF}.

\smallskip\noindent\textbf{\textit{$M$ $\implies$ $\wedge$}:} This is Lemma~\ref{LemMImpliesWedge}.

\smallskip\noindent\textbf{\textit{$\wedge$ $\implies$ $O$}:} This implication is trivial.

\smallskip\noindent\textbf{\textit{$O$ $\implies$ $\F$}:} Proposition~\ref{LemOImpliesF} implies that if $\cl_i \overset \F{\not\to}\cl_j$ then $\cl_i \overset O{\not\to}\cl_j$.
\end{proof}

\subsection{Further properties of $\to$}\label{SubSecRelTG}

Using Theorem~\ref{TheoEquiConnec2}, one can replace from now the relations $\overset *\to$, $\overset \F\to$, $\overset M\to$, $\overset \wedge\to$ and $\overset O\to$ by a single relation $\to$. 

\begin{prop}\label{PropToOrderRel}
Let $f\in\Homeo_0(S)$. Then $\to $ is an order relation.
\end{prop}

\begin{rem}\label{RemGiStrongConnec}
This implies that the $G_i$ are the strong connected components of $G$. Indeed, they are strongly connected by Lemma~\ref{LemLotConnectionsRectangles}, and there are no bigger strongly connected sets because by Proposition~\ref{PropToOrderRel}, if $\cl_i\to\cl_j$ and $\cl_j\to \cl_i$, then $i=j$.
\end{rem}

\begin{proof}
We first prove that if $\cl_i\to\cl_j$ and $\cl_j\to \cl_i$, then $i=j$. 

We use the relation $\overset M\to$ to prove it: there is $\omega, \omega'\in\Omega$ such that $z_\omega\in\cl_i$ and $z_{\omega'}\in\cl_j$ and a path in $G$ linking $R_\omega$ to $R_{\omega'}$ as well as (by Lemma~\ref{LemLotConnectionsRectangles}) a path in $G$ linking $R_{\omega'}$ to $R_\omega$. 
Let $T_\omega, T_{\omega'}\in \G$ such that \eqref{EqDefROmega} (page~\pageref{EqDefROmega}) holds. By Proposition~\ref{LemPointFixe}, there exist $T, T'\in\G$ such that for any $\ell, \ell'\in\N$, there exists a periodic point $x$ and a lift $\wt x$ of $x$ such that 
\[\wt f^\tau(\wt x) = T_{\omega}^{\ell} T T_{\omega'}^{\ell'} T' \wt x\]
($\tau>0$ is the period of $x$, and the deck transformations $T$ and $T'$ correspond to the transitions between $R_\omega$ and $R_{\omega'}$, and between $R_{\omega'}$ and $R_\omega$). If $\ell$ and $\ell'$ are large enough, there are conjugates of $T_{\omega}^{\ell} T T_{\omega'}^{\ell'} T'$ whose geodesic axes are close to respectively the one of $T_\omega$ and $T_{\omega'}$ (see the end of \cite[Section~5.3]{alepablo} for more details about this fact).
This implies that if $\ell$ and $\ell'$ are large enough, then the tracking geodesic of $x$ crosses tracking geodesics of elements of both $\cl_i$ and $\cl_j$, hence that $i=j$.
\bigskip

The fact that $\cl_i\overset M\to\cl_j$ and $\cl_j\overset M\to \cl_k$ imply $\cl_i\overset *\to \cl_k$ was stated in Remark~\ref{RemClaimLemMImpliesStar}.
\end{proof}

\begin{lemma}\label{PropBetween}
If $i,j,k\in I_{\mathrm{h}}$ are such that $\cl_i\to \cl_k$ and $S_j$ separates $S_i$ from $S_k$ in $S$, then $\cl_i\to \cl_j \to \cl_k$.
\end{lemma}

\begin{proof}
Let us use the characterization $\overset O\to$ of $\to$. Consider $B_i^-, B_j^+\subset S$ such that $f^{-1}(B_i^-) \subset B_i^-$,  $f(B_j^+)\subset B_j^+$, $i_* \pi_1(S_i,\R)\subset i_* \pi_1(B_i^-, \R)$ and $i_* \pi_1(S_j,\R)\subset i_* \pi_1(B_j^+, \R)$. Let us prove that there exists $n\ge 0$ such that $f^n(B_i^-) \cap B_j^+\neq\emptyset$.
Suppose that $B_i^- \cap B_j^+\neq\emptyset$, otherwise the property is proved.

Let $\alpha_i\subset B_i^-$ and $\alpha_k\in\pi_1(S_k)$ be essential loops. We suppose that $\alpha_k$ is disjoint from the connected component of the complement of $B_j^+$ containing $B_i^-$ (such a loop exists by the hypothesis that $S_j$ separates $S_i$ from $S_k$ in $S$ and because $B_i^-$ and $B_j^+$ were supposed disjoint). By the characterization $\overset\wedge\to$ of $\to$, there exists $n\ge 0$ such that $f^n(\alpha_i)\cap\alpha_k \neq\emptyset$, which implies that $f^n(\alpha_i)\cap B_j^+ \neq\emptyset$, proving that $\cl_i\to \cl_j$. The proof of $\cl_j\to\cl_k$ is identical.
\end{proof}

\subsection{A graph $\Tr$ associated to the surface}\label{SubSecTree}

%

\begin{figure}[t]
\begin{center}

\tikzset{every picture/.style={line width=0.75pt}} 

\begin{tikzpicture}[x=0.75pt,y=0.75pt,yscale=-1,xscale=.9]

\draw  [color={rgb, 255:red, 189; green, 16; blue, 224 }  ,draw opacity=1 ][fill={rgb, 255:red, 144; green, 19; blue, 254 }  ,fill opacity=0.1 ] (478.72,127.11) .. controls (493.49,128.53) and (501.2,131.1) .. (513,124.89) .. controls (509.2,121.68) and (506.34,116.53) .. (503.01,114.25) .. controls (492.34,107.39) and (582.06,41.1) .. (587.87,50.25) .. controls (603.07,80.07) and (614.43,110.92) .. (617.61,139.7) .. controls (619.08,165.39) and (618.7,177.88) .. (592.7,186.98) .. controls (568.7,195.16) and (516.88,198.98) .. (472.15,197.4) .. controls (484.34,200.25) and (488.91,129.11) .. (478.72,127.11) -- cycle ;
\draw  [color={rgb, 255:red, 189; green, 16; blue, 224 }  ,draw opacity=1 ][fill={rgb, 255:red, 245; green, 166; blue, 35 }  ,fill opacity=0.15 ] (382.75,137.2) .. controls (398.44,135.96) and (405.32,134.83) .. (414.75,125.69) .. controls (439.01,115.68) and (455.3,111.4) .. (463.11,118.12) .. controls (466.15,121.97) and (472.72,123.97) .. (478.72,127.11) .. controls (488.15,130.25) and (484.44,198.25) .. (472.15,197.4) .. controls (442.15,195.68) and (401.3,192.82) .. (370.75,190.06) .. controls (382.15,191.11) and (393.87,139.11) .. (382.75,137.2) -- cycle ;
\draw  [color={rgb, 255:red, 189; green, 16; blue, 224 }  ,draw opacity=1 ][fill={rgb, 255:red, 74; green, 144; blue, 226 }  ,fill opacity=0.15 ] (383.6,15.2) .. controls (405.77,13.87) and (461.01,7.1) .. (499.87,10.25) .. controls (534.44,13.39) and (567.58,10.25) .. (587.87,50.25) .. controls (580.91,43.1) and (494.34,105.1) .. (503.01,114.25) .. controls (494.34,110.25) and (477.2,116.15) .. (473.41,124.69) .. controls (469.3,122.73) and (465.2,120.44) .. (463.11,118.12) .. controls (455.3,109.3) and (424.91,121.3) .. (414.75,125.69) .. controls (410.15,128.72) and (409.3,129.87) .. (404.64,132.46) .. controls (395.01,121.3) and (399.77,123.01) .. (385.03,120.06) .. controls (395.87,119.58) and (391.01,14.15) .. (383.6,15.2) -- cycle ;
\draw [color={rgb, 255:red, 74; green, 144; blue, 226 }  ,draw opacity=1 ][line width=1.5]  [dash pattern={on 1.69pt off 2.76pt}]  (489.87,112.82) .. controls (505.3,109.68) and (482.82,64.25) .. (497.1,60.82) ;
\draw  [fill={rgb, 255:red, 255; green, 255; blue, 255 }  ,fill opacity=1 ][line width=1.5]  (513,124.89) .. controls (501.51,129.01) and (489.72,131.43) .. (473.41,124.69) .. controls (483.38,108.17) and (503.28,108.06) .. (513,124.89) -- cycle ;
\draw [line width=1.5]    (463.11,118.12) .. controls (478.19,132.84) and (509.19,132.51) .. (523.11,118.12) ;

\draw  [color={rgb, 255:red, 189; green, 16; blue, 224 }  ,draw opacity=1 ][fill={rgb, 255:red, 248; green, 231; blue, 28 }  ,fill opacity=0.2 ] (273.25,110.37) .. controls (287.32,78.91) and (280.46,50.06) .. (271.82,34.08) .. controls (299.03,27.78) and (341.6,17.49) .. (383.6,15.2) .. controls (391.89,14.35) and (395.32,119.5) .. (385.03,120.06) .. controls (378.46,120.64) and (369.89,124.35) .. (365.05,132.26) .. controls (371.03,135.78) and (375.6,135.5) .. (382.75,137.2) .. controls (393.6,138.64) and (382.18,190.63) .. (370.75,190.06) .. controls (334.75,187.77) and (321.89,185.2) .. (273.25,182.94) .. controls (283.03,184.06) and (285.6,127.2) .. (273.53,127.8) .. controls (284.18,125.78) and (287.03,125.2) .. (293.14,123.1) .. controls (288.18,116.06) and (286.18,112.91) .. (273.25,110.37) -- cycle ;
\draw  [color={rgb, 255:red, 189; green, 16; blue, 224 }  ,draw opacity=1 ][fill={rgb, 255:red, 65; green, 117; blue, 5 }  ,fill opacity=0.1 ] (69.29,114.25) .. controls (88.96,104.94) and (109.82,107.51) .. (120.25,112.05) .. controls (140.1,128.94) and (166.39,124.08) .. (180.25,112.05) .. controls (201.82,100.08) and (229.82,107.8) .. (243.25,116.33) .. controls (250.1,122.08) and (265.82,128.37) .. (273.53,127.8) .. controls (285.25,127.22) and (283.53,182.94) .. (273.25,182.94) .. controls (241.82,181.22) and (200.39,181.22) .. (155.25,180.08) .. controls (100.68,177.8) and (71.53,162.94) .. (69.29,114.25) -- cycle ;
\draw  [color={rgb, 255:red, 189; green, 16; blue, 224 }  ,draw opacity=1 ][fill={rgb, 255:red, 208; green, 2; blue, 27 }  ,fill opacity=0.1 ] (69.29,114.25) .. controls (68.25,52.54) and (142.85,39.63) .. (222.39,39.8) .. controls (236.85,39.05) and (257.53,38.08) .. (271.82,34.08) .. controls (286.96,62.94) and (281.53,90.65) .. (273.25,110.37) .. controls (268.1,111.22) and (257.82,114.08) .. (253.55,122.91) .. controls (248.39,121.51) and (246.96,118.37) .. (243.25,116.33) .. controls (227.25,104.65) and (194.68,102.65) .. (180.25,112.05) .. controls (175.53,117.22) and (176.73,115.84) .. (170.14,118.82) .. controls (161.25,103.8) and (142.1,100.65) .. (130.55,118.62) .. controls (126.39,116.65) and (124.68,116.37) .. (120.25,112.05) .. controls (109.82,106.94) and (85.53,107.22) .. (69.29,114.25) -- cycle ;
\draw  [fill={rgb, 255:red, 255; green, 255; blue, 255 }  ,fill opacity=1 ][line width=1.5]  (170.14,118.82) .. controls (158.66,122.94) and (146.87,125.36) .. (130.55,118.62) .. controls (140.53,102.09) and (160.42,101.99) .. (170.14,118.82) -- cycle ;
\draw [line width=1.5]    (120.25,112.05) .. controls (135.33,126.77) and (166.33,126.44) .. (180.25,112.05) ;

\draw  [line width=1.5]  (69.29,114.25) .. controls (68.79,44.25) and (166.39,40.8) .. (222.39,39.8) .. controls (278.39,38.8) and (293.84,22.44) .. (381.29,15.25) .. controls (468.75,8.06) and (555.05,1.77) .. (579.01,36.53) .. controls (602.97,71.3) and (621.58,128.82) .. (617.58,162.54) .. controls (613.58,196.25) and (542.58,196.13) .. (496.91,197.87) .. controls (451.25,199.61) and (310.29,181.75) .. (221.79,181.25) .. controls (133.29,180.75) and (69.79,184.25) .. (69.29,114.25) -- cycle ;
\draw  [fill={rgb, 255:red, 255; green, 255; blue, 255 }  ,fill opacity=1 ][line width=1.5]  (293.14,123.1) .. controls (281.66,127.22) and (269.87,129.64) .. (253.55,122.91) .. controls (263.53,106.38) and (283.42,106.28) .. (293.14,123.1) -- cycle ;
\draw [line width=1.5]    (243.25,116.33) .. controls (258.33,131.06) and (289.33,130.72) .. (303.25,116.33) ;

\draw  [fill={rgb, 255:red, 255; green, 255; blue, 255 }  ,fill opacity=1 ][line width=1.5]  (404.64,132.46) .. controls (393.16,136.58) and (381.37,139) .. (365.05,132.26) .. controls (375.03,115.74) and (394.92,115.63) .. (404.64,132.46) -- cycle ;
\draw [line width=1.5]    (354.75,125.69) .. controls (369.83,140.41) and (400.83,140.08) .. (414.75,125.69) ;

\draw  [color={rgb, 255:red, 208; green, 2; blue, 27 }  ,draw opacity=1 ][fill={rgb, 255:red, 208; green, 2; blue, 27 }  ,fill opacity=0.15 ] (189.53,56.08) .. controls (205.82,51.22) and (213.82,49.22) .. (229.82,52.94) .. controls (234.1,66.94) and (233.53,84.65) .. (229.82,93.51) .. controls (218.1,90.65) and (206.1,90.65) .. (191.82,92.65) .. controls (187.25,78.94) and (186.96,64.37) .. (189.53,56.08) -- cycle ;
\draw [color={rgb, 255:red, 208; green, 2; blue, 27 }  ,draw opacity=1 ][line width=1.5]    (142.96,46.65) .. controls (150.96,45.22) and (169.82,63.51) .. (210.1,60.94) .. controls (250.39,58.37) and (250.39,38.37) .. (259.82,36.65) ;
\draw [color={rgb, 255:red, 208; green, 2; blue, 27 }  ,draw opacity=1 ][line width=1.5]  [dash pattern={on 1.69pt off 2.76pt}]  (259.82,36.65) .. controls (272.96,34.37) and (279.82,106.65) .. (267.25,110.94) ;
\draw [color={rgb, 255:red, 208; green, 2; blue, 27 }  ,draw opacity=1 ][line width=1.5]  [dash pattern={on 1.69pt off 2.76pt}]  (142.96,46.65) .. controls (132.1,49.22) and (141.25,107.51) .. (149.82,106.37) ;
\draw  [color={rgb, 255:red, 65; green, 117; blue, 5 }  ,draw opacity=1 ][fill={rgb, 255:red, 65; green, 117; blue, 5 }  ,fill opacity=0.15 ] (180.96,123.51) .. controls (193.82,118.94) and (207.53,115.51) .. (226.39,124.37) .. controls (230.1,137.51) and (232.39,165.51) .. (229.53,175.22) .. controls (216.96,174.37) and (199.82,173.22) .. (182.96,175.8) .. controls (175.53,159.22) and (175.82,134.94) .. (180.96,123.51) -- cycle ;
\draw [color={rgb, 255:red, 208; green, 2; blue, 27 }  ,draw opacity=1 ][line width=1.5]    (149.82,106.37) .. controls (157.82,104.94) and (170.68,83.8) .. (211.25,84.37) .. controls (251.82,84.94) and (239.85,127.34) .. (229.82,129.22) .. controls (219.79,131.11) and (188.28,123.28) .. (173.82,128.37) ;
\draw [color={rgb, 255:red, 65; green, 117; blue, 5 }  ,draw opacity=1 ][line width=1.5]    (147.53,122.94) .. controls (165.53,123.8) and (152.65,139.28) .. (203.74,138.37) .. controls (254.83,137.46) and (254.91,122.47) .. (263.19,126.19) ;
\draw [color={rgb, 255:red, 65; green, 117; blue, 5 }  ,draw opacity=1 ][line width=1.5]  [dash pattern={on 1.69pt off 2.76pt}]  (147.53,122.94) .. controls (138.83,122.92) and (141.37,179.46) .. (149.19,179.82) ;
\draw [color={rgb, 255:red, 65; green, 117; blue, 5 }  ,draw opacity=1 ][line width=1.5]  [dash pattern={on 1.69pt off 2.76pt}]  (263.19,126.19) .. controls (270.65,129.1) and (272.1,183.1) .. (263.55,182.37) ;
\draw [color={rgb, 255:red, 208; green, 2; blue, 27 }  ,draw opacity=1 ][line width=1.5]    (234.83,149.46) .. controls (206.54,138.03) and (146.26,151.25) .. (127.37,142.19) .. controls (108.48,133.12) and (98.76,126.58) .. (101.82,105.8) .. controls (104.87,85.01) and (122.06,82.9) .. (144.96,81.22) .. controls (167.87,79.55) and (182.1,69.8) .. (210.39,68.94) .. controls (238.68,68.08) and (260.94,113.49) .. (267.25,110.94) ;
\draw [color={rgb, 255:red, 65; green, 117; blue, 5 }  ,draw opacity=1 ][line width=1.5]    (171.74,171.64) .. controls (187.74,168.37) and (222.9,166.12) .. (235.74,168.55) .. controls (248.57,170.98) and (260.34,181.95) .. (264.46,182.19) ;
\draw  [fill={rgb, 255:red, 255; green, 255; blue, 255 }  ,fill opacity=1 ][line width=1.5]  (501.71,58.46) .. controls (490.23,62.58) and (478.44,65) .. (462.12,58.26) .. controls (472.1,41.74) and (491.99,41.63) .. (501.71,58.46) -- cycle ;
\draw [line width=1.5]    (451.82,51.69) .. controls (466.9,66.41) and (497.9,66.08) .. (511.82,51.69) ;

\draw  [color={rgb, 255:red, 217; green, 201; blue, 0 }  ,draw opacity=1 ][fill={rgb, 255:red, 248; green, 231; blue, 28 }  ,fill opacity=0.25 ] (316.82,41.01) .. controls (336.82,31.01) and (340.75,35.49) .. (359.6,38.92) .. controls (361.32,54.06) and (361.03,78.35) .. (359.03,93.2) .. controls (348.46,91.49) and (341.96,89.59) .. (322.82,93.01) .. controls (325.68,77.01) and (321.97,53.58) .. (316.82,41.01) -- cycle ;
\draw [color={rgb, 255:red, 65; green, 117; blue, 5 }  ,draw opacity=1 ][line width=1.5]    (149.19,179.82) .. controls (158.83,181.1) and (146.83,156.92) .. (197.92,156.01) .. controls (249.01,155.1) and (245.43,168.43) .. (285.89,146.34) .. controls (326.35,124.25) and (298.48,113.19) .. (296.18,93.77) .. controls (293.87,74.36) and (298.46,68.63) .. (308.46,64.06) .. controls (318.46,59.49) and (376.47,60.46) .. (378.75,61.49) ;
\draw [color={rgb, 255:red, 217; green, 201; blue, 0 }  ,draw opacity=1 ][line width=1.5]    (282.46,31.78) .. controls (287.53,30.33) and (294.62,40.42) .. (304.46,44.92) .. controls (314.3,49.42) and (323.65,42.63) .. (336.75,42.06) .. controls (349.84,41.5) and (359.81,47.83) .. (366.18,41.78) .. controls (372.54,35.72) and (365.03,16.92) .. (373.32,15.78) ;
\draw [color={rgb, 255:red, 217; green, 201; blue, 0 }  ,draw opacity=1 ][line width=1.5]  [dash pattern={on 1.69pt off 2.76pt}]  (373.32,15.78) .. controls (388.75,13.49) and (391.32,118.92) .. (380.18,120.06) ;
\draw [color={rgb, 255:red, 217; green, 201; blue, 0 }  ,draw opacity=1 ][line width=1.5]    (365.89,190.06) .. controls (362.32,189.46) and (338.68,186.83) .. (325.89,171.2) .. controls (313.1,155.58) and (310.9,112.36) .. (305.89,87.2) .. controls (300.89,62.05) and (349.06,70.16) .. (362.18,69.78) .. controls (375.29,69.39) and (367.08,120.63) .. (380.18,120.06) ;
\draw [color={rgb, 255:red, 217; green, 201; blue, 0 }  ,draw opacity=1 ][line width=1.5]  [dash pattern={on 1.69pt off 2.76pt}]  (282.46,31.78) .. controls (271.32,33.77) and (280.18,110.34) .. (284.46,112.91) ;
\draw [color={rgb, 255:red, 217; green, 201; blue, 0 }  ,draw opacity=1 ][line width=1.5]  [dash pattern={on 1.69pt off 2.76pt}]  (372.18,134.79) .. controls (387.6,132.5) and (377.03,190.91) .. (365.89,190.06) ;
\draw  [color={rgb, 255:red, 74; green, 144; blue, 226 }  ,draw opacity=1 ][fill={rgb, 255:red, 74; green, 144; blue, 226 }  ,fill opacity=0.2 ] (425.82,73.8) .. controls (442.1,68.94) and (450.1,66.94) .. (466.1,70.65) .. controls (470.39,84.65) and (469.82,102.37) .. (466.1,111.22) .. controls (454.39,108.37) and (442.39,108.37) .. (428.1,110.37) .. controls (423.53,96.65) and (423.25,82.08) .. (425.82,73.8) -- cycle ;
\draw [color={rgb, 255:red, 217; green, 201; blue, 0 }  ,draw opacity=1 ][line width=1.5]    (284.46,112.91) .. controls (290.18,118.63) and (281.03,62.34) .. (308.18,54.91) .. controls (335.32,47.49) and (379.1,47.1) .. (395.96,65.68) .. controls (412.82,84.25) and (372.53,101.1) .. (482.53,95.1) ;
\draw [color={rgb, 255:red, 74; green, 144; blue, 226 }  ,draw opacity=1 ][line width=1.5]    (417.96,86.25) .. controls (456.83,86.09) and (522.82,96.53) .. (525.96,65.68) .. controls (529.1,34.82) and (502.15,29.68) .. (473.39,28.53) .. controls (444.63,27.39) and (410.97,45.61) .. (410.53,61.68) .. controls (410.09,77.74) and (406.27,81.57) .. (459.95,77.8) .. controls (513.64,74.02) and (511.96,55.68) .. (497.1,60.82) ;
\draw [color={rgb, 255:red, 74; green, 144; blue, 226 }  ,draw opacity=1 ][line width=1.5]    (419.68,104.54) .. controls (505.1,97.68) and (479.87,115.69) .. (489.87,112.82) ;
\draw  [fill={rgb, 255:red, 255; green, 255; blue, 255 }  ,fill opacity=1 ][line width=1.5]  (580,158.46) .. controls (568.51,162.58) and (556.72,165) .. (540.41,158.26) .. controls (550.38,141.74) and (570.28,141.63) .. (580,158.46) -- cycle ;
\draw [line width=1.5]    (530.11,151.69) .. controls (545.19,166.41) and (576.19,166.08) .. (590.11,151.69) ;

\draw  [color={rgb, 255:red, 245; green, 166; blue, 35 }  ,draw opacity=1 ][fill={rgb, 255:red, 245; green, 166; blue, 35 }  ,fill opacity=0.2 ] (415.01,142.82) .. controls (431.3,137.96) and (433.87,136.25) .. (449.87,139.96) .. controls (454.15,153.96) and (453.01,176.53) .. (449.3,185.39) .. controls (439.01,183.1) and (426.72,183.96) .. (416.15,182.25) .. controls (418.15,171.39) and (419.01,152.53) .. (415.01,142.82) -- cycle ;
\draw [color={rgb, 255:red, 245; green, 166; blue, 35 }  ,draw opacity=1 ][line width=1.5]    (407.01,159.1) .. controls (426.44,157.39) and (451.77,149.48) .. (458.15,162.82) .. controls (464.53,176.16) and (458.83,196.71) .. (465.01,197.1) ;
\draw [color={rgb, 255:red, 245; green, 166; blue, 35 }  ,draw opacity=1 ][line width=1.5]  [dash pattern={on 1.69pt off 2.76pt}]  (473.41,124.69) .. controls (481.87,127.39) and (476.15,197.96) .. (465.01,197.1) ;
\draw [color={rgb, 255:red, 245; green, 166; blue, 35 }  ,draw opacity=1 ][line width=1.5]    (395.58,135.39) .. controls (405.87,132.25) and (391.01,151.68) .. (437.58,148.82) .. controls (484.15,145.96) and (468.15,121.68) .. (473.41,124.69) ;
\draw [color={rgb, 255:red, 245; green, 166; blue, 35 }  ,draw opacity=1 ][line width=1.5]  [dash pattern={on 1.69pt off 2.76pt}]  (395.58,135.39) .. controls (388.44,136.53) and (387.87,191.39) .. (396.15,192.25) ;
\draw [color={rgb, 255:red, 245; green, 166; blue, 35 }  ,draw opacity=1 ][line width=1.5]    (396.15,192.25) .. controls (402.03,192.68) and (402.1,180.79) .. (409.58,177.96) .. controls (417.06,175.13) and (433.01,175.39) .. (455.01,179.96) ;
\draw [color={rgb, 255:red, 217; green, 201; blue, 0 }  ,draw opacity=1 ][line width=1.5]    (455.58,169.39) .. controls (403.3,165.96) and (371.73,187.37) .. (346.18,163.49) .. controls (320.63,139.61) and (320.53,119.3) .. (316.53,99.01) .. controls (312.53,78.73) and (345.89,76.35) .. (362.18,84.63) .. controls (378.46,92.92) and (333.32,132.92) .. (344.75,136.92) .. controls (356.18,140.92) and (367.32,132.92) .. (372.18,134.79) ;
\draw [color={rgb, 255:red, 80; green, 227; blue, 194 }  ,draw opacity=1 ][line width=1.5]    (503.01,114.25) .. controls (492.44,105.39) and (582.72,42.53) .. (587.87,50.25) ;
\draw [color={rgb, 255:red, 80; green, 227; blue, 194 }  ,draw opacity=1 ][line width=1.5]  [dash pattern={on 1.69pt off 2.76pt}]  (503.01,114.25) .. controls (505.58,119.4) and (521.72,100.78) .. (543.02,84.8) .. controls (564.31,68.82) and (590.72,55.53) .. (587.87,50.25) ;
\draw [color={rgb, 255:red, 144; green, 19; blue, 254 }  ,draw opacity=1 ]   (533.42,196.8) .. controls (540.02,195.8) and (546.62,162.2) .. (553.62,162.6) ;
\draw [color={rgb, 255:red, 144; green, 19; blue, 254 }  ,draw opacity=1 ] [dash pattern={on 0.84pt off 2.51pt}]  (537.62,157) .. controls (548.02,163) and (560.02,194.4) .. (569.62,192.6) ;
\draw [color={rgb, 255:red, 144; green, 19; blue, 254 }  ,draw opacity=1 ]   (569.62,192.6) .. controls (575.62,191.93) and (601.82,150.8) .. (592.82,145) .. controls (583.82,139.2) and (575.02,152) .. (571.02,149) ;
\draw [color={rgb, 255:red, 144; green, 19; blue, 254 }  ,draw opacity=1 ]   (515.42,197.4) .. controls (522.02,196.4) and (537.62,156.8) .. (544.82,160.2) ;
\draw [color={rgb, 255:red, 144; green, 19; blue, 254 }  ,draw opacity=1 ]   (505.42,127.8) .. controls (499.16,130.16) and (494.83,143.68) .. (505.42,151) .. controls (516.01,158.32) and (534.66,155.61) .. (537.62,157) ;
\draw [color={rgb, 255:red, 144; green, 19; blue, 254 }  ,draw opacity=1 ] [dash pattern={on 0.84pt off 2.51pt}]  (544.82,160.2) .. controls (555.22,166.2) and (570.02,192.4) .. (579.62,190.6) ;
\draw [color={rgb, 255:red, 144; green, 19; blue, 254 }  ,draw opacity=1 ] [dash pattern={on 0.84pt off 2.51pt}]  (553.62,162.6) .. controls (565.62,164.6) and (589.02,188.6) .. (596.82,185.4) ;
\draw [color={rgb, 255:red, 144; green, 19; blue, 254 }  ,draw opacity=1 ]   (579.62,190.6) .. controls (585.62,189.93) and (616.62,142.6) .. (586.02,127.8) .. controls (555.42,113) and (599.82,75.4) .. (597.62,70) ;
\draw [color={rgb, 255:red, 144; green, 19; blue, 254 }  ,draw opacity=1 ]   (596.82,185.4) .. controls (602.82,184.73) and (622.22,138) .. (601.62,124.2) .. controls (581.02,110.4) and (604.02,87) .. (602.42,82) ;
\draw [color={rgb, 255:red, 144; green, 19; blue, 254 }  ,draw opacity=1 ] [dash pattern={on 0.84pt off 2.51pt}]  (510.42,120.8) .. controls (506.22,115.2) and (535.32,108.72) .. (559.82,98.8) .. controls (584.31,88.88) and (600.02,75.2) .. (602.42,82) ;
\draw [color={rgb, 255:red, 144; green, 19; blue, 254 }  ,draw opacity=1 ] [dash pattern={on 0.84pt off 2.51pt}]  (507.42,118) .. controls (504.02,115.4) and (524.92,104.12) .. (549.42,94.2) .. controls (573.91,84.28) and (595.22,63.2) .. (597.62,70) ;
\draw [color={rgb, 255:red, 144; green, 19; blue, 254 }  ,draw opacity=1 ]   (507.42,118) .. controls (511.42,120.2) and (546.03,95.88) .. (556.62,103.2) .. controls (567.21,110.52) and (564.62,143) .. (557.62,145.8) ;
\draw [color={rgb, 255:red, 144; green, 19; blue, 254 }  ,draw opacity=1 ]   (510.42,120.8) .. controls (514.42,123) and (541.03,102.68) .. (551.62,110) .. controls (562.21,117.32) and (555.02,146.8) .. (548.02,149.6) ;
\draw [color={rgb, 255:red, 144; green, 19; blue, 254 }  ,draw opacity=1 ] [dash pattern={on 0.84pt off 2.51pt}]  (505.42,127.8) .. controls (515.02,124.8) and (550.22,150.2) .. (558.02,147) ;
\draw [color={rgb, 255:red, 144; green, 19; blue, 254 }  ,draw opacity=1 ] [dash pattern={on 0.84pt off 2.51pt}]  (548.02,149.6) .. controls (543.82,153.4) and (520.58,136.9) .. (512.42,148.2) .. controls (504.25,159.5) and (510.07,197.4) .. (515.42,197.4) ;
\draw [color={rgb, 255:red, 144; green, 19; blue, 254 }  ,draw opacity=1 ]   (543.42,154.4) .. controls (547.62,148.8) and (509.82,127.2) .. (516.42,123.6) ;
\draw [color={rgb, 255:red, 144; green, 19; blue, 254 }  ,draw opacity=1 ] [dash pattern={on 0.84pt off 2.51pt}]  (516.42,123.6) .. controls (525.02,117.6) and (563.62,144.2) .. (571.02,149) ;
\draw [color={rgb, 255:red, 144; green, 19; blue, 254 }  ,draw opacity=1 ] [dash pattern={on 0.84pt off 2.51pt}]  (521.42,165.4) .. controls (513.25,176.7) and (528.07,196.8) .. (533.42,196.8) ;
\draw [color={rgb, 255:red, 144; green, 19; blue, 254 }  ,draw opacity=1 ] [dash pattern={on 0.84pt off 2.51pt}]  (543.42,154.4) .. controls (539.22,158.2) and (531.58,142.3) .. (523.42,153.6) ;

\draw (121.17,68) node  [color={rgb, 255:red, 174; green, 0; blue, 22 }  ,opacity=1 ,xscale=1.2,yscale=1.2]  {$\cl_{1}$};
\draw (103.17,151.19) node  [color={rgb, 255:red, 65; green, 117; blue, 5 }  ,opacity=1 ,xscale=1.2,yscale=1.2]  {$\cl_{2}$};
\draw (303.84,164.52) node  [color={rgb, 255:red, 170; green, 157; blue, 0 }  ,opacity=1 ,xscale=1.2,yscale=1.2]  {$\cl_{3}$};
\draw (450.51,129.19) node  [color={rgb, 255:red, 194; green, 121; blue, 0 }  ,opacity=1 ,xscale=1.2,yscale=1.2]  {$\cl_{4}$};
\draw (409.84,30.52) node  [color={rgb, 255:red, 0; green, 87; blue, 187 }  ,opacity=1 ,xscale=1.2,yscale=1.2]  {$\cl_{5}$};
\draw (553.17,48.52) node  [color={rgb, 255:red, 0; green, 165; blue, 128 }  ,opacity=1 ,xscale=1.2,yscale=1.2]  {$\cl_{6}$};
\draw (499.17,179.19) node  [color={rgb, 255:red, 144; green, 19; blue, 254 }  ,opacity=1 ,xscale=1.2,yscale=1.2]  {$\cl_{7}$};
\end{tikzpicture}

\vspace{30pt}

\tikzset{every picture/.style={line width=0.75pt}} 

\begin{tikzpicture}[x=0.75pt,y=0.75pt,yscale=-1,xscale=1]

\draw    (262.54,201.58) -- (311.55,180.04) ;
\draw [shift={(314.29,178.83)}, rotate = 156.27] [fill={rgb, 255:red, 0; green, 0; blue, 0 }  ][line width=0.08]  [draw opacity=0] (7.14,-3.43) -- (0,0) -- (7.14,3.43) -- (4.74,0) -- cycle    ;
\draw    (261.17,125.58) -- (260.57,181.08) ;
\draw [shift={(260.54,184.08)}, rotate = 270.61] [fill={rgb, 255:red, 0; green, 0; blue, 0 }  ][line width=0.08]  [draw opacity=0] (7.14,-3.43) -- (0,0) -- (7.14,3.43) -- (4.74,0) -- cycle    ;
\draw    (330,170) -- (380.83,194.05) ;
\draw [shift={(383.54,195.33)}, rotate = 205.32] [fill={rgb, 255:red, 0; green, 0; blue, 0 }  ][line width=0.08]  [draw opacity=0] (7.14,-3.43) -- (0,0) -- (7.14,3.43) -- (4.74,0) -- cycle    ;
\draw    (382.41,141.52) -- (330,170) ;
\draw [shift={(385.04,140.08)}, rotate = 151.47] [fill={rgb, 255:red, 0; green, 0; blue, 0 }  ][line width=0.08]  [draw opacity=0] (8.04,-3.86) -- (0,0) -- (8.04,3.86) -- (5.34,0) -- cycle    ;

\draw  [color={rgb, 255:red, 208; green, 2; blue, 27 }  ,draw opacity=1 ][fill={rgb, 255:red, 255; green, 221; blue, 227 }  ,fill opacity=1 ]  (260.5, 129) circle [x radius= 16.62, y radius= 16.62]   ;
\draw (260.5,129) node  [color={rgb, 255:red, 152; green, 0; blue, 17 }  ,opacity=1 ,xscale=1.2,yscale=1.2]  {$S_{1}$};
\draw  [color={rgb, 255:red, 74; green, 144; blue, 226 }  ,draw opacity=1 ][fill={rgb, 255:red, 224; green, 241; blue, 255 }  ,fill opacity=1 ]  (399.5, 131) circle [x radius= 16.62, y radius= 16.62]   ;
\draw (399.5,131) node  [color={rgb, 255:red, 26; green, 87; blue, 156 }  ,opacity=1 ,xscale=1.2,yscale=1.2]  {$S_{5}$};
\draw  [color={rgb, 255:red, 245; green, 166; blue, 35 }  ,draw opacity=1 ][fill={rgb, 255:red, 255; green, 247; blue, 229 }  ,fill opacity=1 ]  (400.5, 199) circle [x radius= 16.62, y radius= 16.62]   ;
\draw (400.5,199) node  [color={rgb, 255:red, 210; green, 135; blue, 13 }  ,opacity=1 ,xscale=1.2,yscale=1.2]  {$S_{4}$};
\draw  [color={rgb, 255:red, 217; green, 201; blue, 0 }  ,draw opacity=1 ][fill={rgb, 255:red, 255; green, 249; blue, 201 }  ,fill opacity=1 ]  (330, 170) circle [x radius= 16.62, y radius= 16.62]   ;
\draw (330,170) node  [color={rgb, 255:red, 186; green, 172; blue, 0 }  ,opacity=1 ,xscale=1.2,yscale=1.2]  {$S_{3}$};
\draw  [color={rgb, 255:red, 65; green, 117; blue, 5 }  ,draw opacity=1 ][fill={rgb, 255:red, 240; green, 250; blue, 226 }  ,fill opacity=1 ]  (260.5, 201) circle [x radius= 16.62, y radius= 16.62]   ;
\draw (260.5,201) node  [color={rgb, 255:red, 54; green, 100; blue, 1 }  ,opacity=1 ,xscale=1.2,yscale=1.2]  {$S_{2}$};

\end{tikzpicture}

\caption{An example of the classes $(\cl_i)_{i\in I_{\mathrm h}\cup I^1}$ (defined in Definitions~\ref{DefRelEqMeas} and \ref{DefClassesI1}) of a homeomorphism of a closed genus 6 surface. The classes $(\cl_i)_{1\le i\le 5}$ belong to $I_{\mathrm h}$, while the classes $\cl_6$ and $\cl_7$ belong to $I^1$. The geodesic lamination $\dot\Lambda_6$ is made of a single closed geodesic, while $\dot\Lambda_7$ is a minimal geodesic lamination with non closed leaves (\cite[Theorem~D]{alepablo}.}\label{FigTree0}
\end{center}
\end{figure}
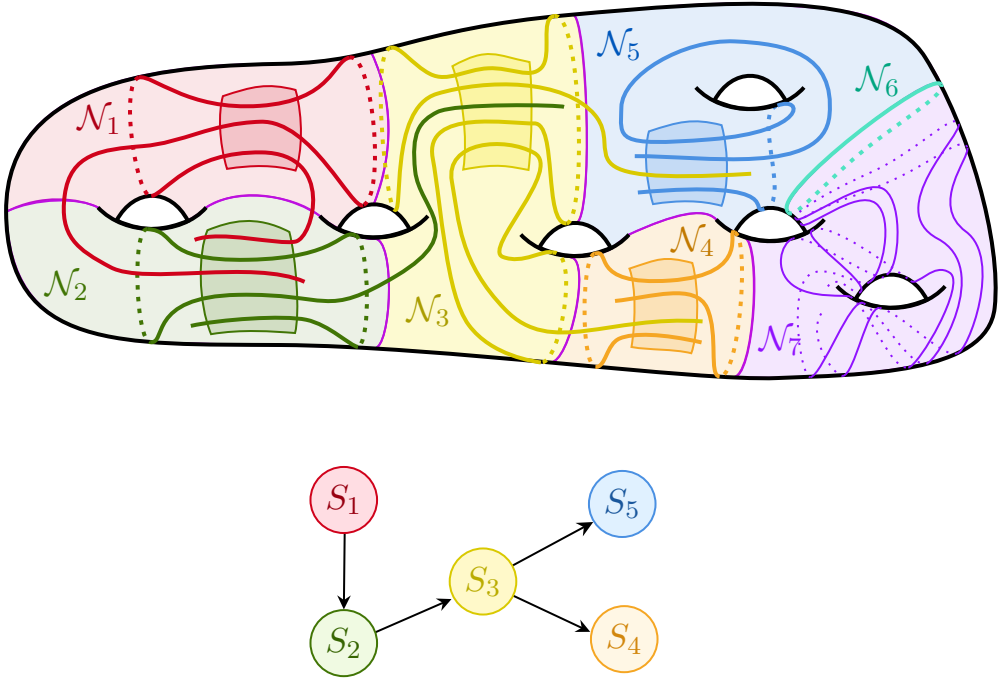

Let us consider the finite graph $\Tr$ (see Figure~\ref{FigTree0}) whose vertices are the surfaces $(S_i)_{i\in I_{\mathrm h}}$ and for which we put an oriented edge $S_i \to S_j$ if $i\neq j$, $\cl_i\to \cl_j$ and there is no $k\in I_{\mathrm h}$ such that $S_k$ separates $S_i$ from $S_j$. 
As the correspondence $\cl_i\leftrightarrow S_i$ is 1 to 1 in $I_{\mathrm h}$, we will sometimes label the vertices of $\Tr$ with the classes $\cl_i$. 
As $\to $ is an order relation (Proposition~\ref{PropToOrderRel}), this definition indeed leads to an oriented graph without closed loops. 

By Lemma~\ref{PropBetween}, having only the data of the relations $\cl_i\to \cl_j$ for $S_i$ and $S_j$ adjacent in the graph $\Tr$ allows to recover the whole relation $\to$: $S_i\to S_j$ iff there is a path in $\Tr$ from $S_i$ to $S_j$.

Note that the graph $G$ is a refinement of the graph $\Tr$  (see Remark~\ref{RemGiStrongConnec}), more precisely quotienting down the strong connected components of $G$ leads to the graph obtained from the graph $\Tr$ by adding all the edges $S_i\to S_j$ and not only the ones for ``adjacent'' surfaces $S_i$. 

\begin{prop}\label{PropRotfEqualRotTr}
Let $f\in\Homeo_0(S)$. Then
\[\overline{\rot(G)} = \bigcup_{\mathrm{p} \text{ path in }\Tr}\conv\bigg(\bigcup_{i\in \mathrm{p}}\overline{\rho_i}\bigg),\]
where a path in $\Tr$ is a sequence $\mathrm{p} = (i_k)_{1\le k\le \ell(\mathrm{p})}$ such that for any $1\le k<\ell(\mathrm{p})$ one has $\cl_{i_k}\to \cl_{i_{k+1}}$ (we allow the possibility $\ell(\mathrm{p}) = 1$, hence paths made of a single class $\cl_i$). 
\end{prop}

Combined with Proposition~\ref{PropConnectRectEnsRot}, this gives immediately the following corollary:

\begin{coro}\label{CoroPropRotfEqualRotTr}
Let $f\in\Homeo_0(S)$. Then
\[\bigcup_{\mathrm{p} \text{ path in }\Tr}\conv\bigg(\bigcup_{i\in \mathrm{p}}\overline{\rho_i}\bigg) \subset \rot(f).\]
\end{coro}

\begin{proof}[Proof of Proposition~\ref{PropRotfEqualRotTr}]
First, let $\mathrm{p} = (i_k)_{1\le k\le \ell(\mathrm{p})}$ be a path in $\Tr$, and pick $\rho\in \conv(\cup_{i\in \mathrm{p}}\overline{\rho_i})$ and $\varep>0$. We write $\rho = \sum_{i\in \mathrm{p}}\lambda_i v_i$, with $\lambda_i\ge 0$, $\sum_i\lambda_i = 1$ and $v_i\in\overline{\rho_i}$. As the $r_\omega$ are dense in $\bigcup_{i\in I_{\mathrm h}}\rho_i$, we can find a family $(\omega_i)_{i\in \mathrm{p}} \in \Omega^{\ell(\mathrm{p})}$ such that for any $k$ we have $\|v_i - r_{\omega_i}\|\le\varep$. 

For any $1\le k\le \ell(\mathrm{p})-1$ we have $\cl_{i_k}\to\cl_{i_{k+1}}$. Hence, using the characterization $\overset M\to$ of $\to $ (Theorem~\ref{TheoEquiConnec2}), for any $1\le k\le \ell(\mathrm{p})-1$ there exist $\omega,\omega'\in \Omega$ such that $z_\omega\in \cl_{i_k}$ and $z_{\omega'}\in \cl_{i_{k+1}}$; in other words there is a path starting in $G_{i_k}$ and finishing in $G_{i_{k+1}}$.
As the $G_i$ are strongly connected (Lemma~\ref{LemLotConnectionsRectangles}), this allows to find a path in $G$ visiting all the $(R_{\omega_i})_{i\in \mathrm{p}}$. It implies that $\conv( \{r_{\omega_i}\mid i\in \mathrm{p}\}) \subset \rot(G)$. 
We have proved that $\bigcup_{\mathrm{p} \text{ path in }\Tr}\conv\big(\bigcup_{i\in \mathrm{p}}\overline{\rho_i}\big) \subset \overline{\rot(G)}$.
\bigskip

Now, let $\rho\in\rot(G)$. This means there exists a path $R_{\omega_1}\to\dots\to R_{\omega_\ell}$ in $G$ such that $\rho\in \conv\big(\{r_{\omega_j}\mid 1\le j\le \ell\}\big)$. By Lemma~\ref{PropBetween} and Lemma~\ref{LemLotConnectionsRectangles}, for any $j$ we can replace $R_{\omega_j}\to R_{\omega_{j+1}}$ by some path in $G$ such that any two consecutive rectangles in this path are associated with adjacent (or equal) surfaces in $\Tr$; this gives a path $R_{\omega'_1}\to\dots\to R_{\omega'_{\ell'}}$ in $G$ containing all the $R_{\omega_j}$, and such that $\rho\in\conv\big(\{r_{\omega'_j}\mid 1\le j\le \ell' \}\big)$. 

Moreover, by Proposition~\ref{PropToOrderRel}, there exist $1= k_1\le\dots \le k_{m+1} = \ell'+1$ and $S_{i_1} \to \dots\to S_{i_m}$ such that for any $j$, we have $r_{\omega'_{k_j}},\dots, r_{\omega'_{k_{j+1}-1}}\in \rho_{i_j}$; in other words $R_{\omega'_{k_j}}\to\dots \to R_{\omega'_{k_{j+1}-1}}$ is a path of $G_i$. So 
\[\rho\in\conv\big(\{r_{\omega'_j} \mid 1\le j\le \ell' \}\big) \subset \conv\bigg(\bigcup_{ 1\le j\le m }\rho_{i_j}\bigg).\]

As we have chosen consecutive rectangles to be in adjacent surfaces, the path
$S_{i_1} \to \dots\to S_{i_m}$ is a path in $\Tr$.
\end{proof}

\subsection{Some open invariant sets associated to the graph $\Tr$}\label{SecOpenInv}

Let us finish with a few comments relative to Proposition~\ref{LemOImpliesF} and the graph $\Tr$. 

Let us consider the connected components $(G_\alpha)_{\alpha\in\mathcal A}$ of $G$, which correspond to the connected components of $\Tr$ (Remark~\ref{RemGiStrongConnec}). For any $\alpha\in\mathcal A$, identified with the set of $i\in I_{\mathrm h}$ such that $G_i\subset G_\alpha$, set 
\[B_\alpha = \bigcup_{i\in \alpha} B_i^o.\]
By Proposition~\ref{LemOImpliesF}, this gives a collection $(B_\alpha)_{\alpha\in\mathcal A}$ of pairwise disjoint, connected, essential and filled open sets, satisfying $f(B_\alpha) = B_\alpha$, and such that for any $\alpha\in\mathcal A$ we have
\[\left\langle i_* \pi_1(S_i) \mid {i\in \alpha}\right\rangle \subset i_*\pi_1(B_\alpha).\]
The set $B_\alpha$ is also of full measure for any $\mu\in\cl_i$ with $i\in\alpha$.

On each connected component $G_\alpha$, there is a filtration by open sets, that could be interpreted as a Lyapunov filtration for the rotational behaviour: for any $\alpha\in\mathrm A$ and any $i\in\alpha$, one can set 
\[U_i^+ = \bigcup_{\substack{j\in \alpha\\ \cl_i\to \cl_j}}B_j^+
\qquad\text{and}\qquad
U_i^- = \bigcup_{\substack{j\in \alpha\\ \cl_j\to \cl_i}}B_j^-.\]
By Proposition~\ref{LemOImpliesF}, these are connected, essential and filled open sets, satisfying $f(U_i^+) \subset U_i^+$ and $f^{-1}(U_i^-) \subset U_i^-$, and such that for any $i$ we have
\[\left\langle i_* \pi_1(S_j) \mid {\cl_i\to \cl_j}\right\rangle \subset i_*\pi_1(U_i^+) 
\quad\text{and}\quad
\left\langle i_* \pi_1(S_j) \mid {\cl_j\to \cl_i}\right\rangle \subset i_*\pi_1(U_i^-).\]
Proposition~\ref{LemOImpliesF} also implies that if $\cl_i\to \cl_j$ and $i\neq j$, then $U_i^- \cap U_j^+ = \emptyset$.
The sets $U_i^-$ are also of full measure for any $\mu\in\cl_j$ with $\cl_j\to\cl_i$, and the sets $U_i^+$ are of full measure for any $\mu\in\cl_j$ with $\cl_i\to\cl_j$.

\small

\bibliographystyle{alpha}
\bibliography{Biblio}

\end{document}